\documentclass{daj}

\usepackage[english]{babel}
\usepackage[utf8]{inputenc}
\usepackage{enumerate}

\usepackage[a4paper,inner=3cm,outer=3cm,top=4cm,bottom=4cm,pdftex]{geometry}
\usepackage{fancyhdr}
\usepackage{titlesec}
\usepackage{color}
\usepackage{bold-extra}
\usepackage{ mathrsfs }
\titleformat{\section}{\normalfont\scshape\centering}{\thesection}{1em}{}
  \titleformat{\subsection}{\bfseries}{\thesubsection}{1em}{}
  
\usepackage{comment}\usepackage{amsmath}
\usepackage{amsfonts}
\usepackage{amssymb}
\usepackage{graphics}
\usepackage{aliascnt}

\usepackage{amsmath}
\usepackage{amsfonts}
\usepackage{amssymb}
\usepackage{amsthm}
\usepackage{comment}
\usepackage{mathtools}

\newtheorem{theorem}{Theorem}[section]
\newtheorem{corollary}[theorem]{Corollary}

\newtheorem{lemma}[theorem]{Lemma}
\newtheorem{proposition}[theorem]{Proposition}
\theoremstyle{definition}
\newtheorem{definition}[theorem]{Definition}
\newtheorem{remark}[theorem]{Remark}

\numberwithin{equation}{section}

\newcommand{\Z}{\mathbb{Z}}
\newcommand{\C}{\mathbb{C}}
\newcommand{\R}{\mathbb{R}}
\newcommand{\ee}{\varepsilon}
\newcommand{\wt}{\widetilde}
\newcommand{\ve}{\mathbf}
\newcommand{\Mod}[1]{\ (\mathrm{mod}\ #1)}
\newcommand{\supp}{\textnormal{supp}}

\dajAUTHORdetails{%
  title = {The Bombieri--Vinogradov Theorem for Nilsequences}, %% please capitalize all significant words
  author = {Xuancheng Shao and Joni Ter\"{a}v\"{a}inen},
  plaintextauthor = {Xuancheng Shao, Joni Teravainen},
  keywords = {Bombieri-Vinogradov theorem, Nilsequences, Gowers norms},
}   %%% END \dajAUTHORdetails

\dajEDITORdetails{%
   year={2021},
   number={21},
   received={12 June 2020},   % received date: example: 7 January 2017
   revised={8 April 2021},    % Optional revised date (you may comment it out)
   published={11 October 2021},  % published date
   doi={10.19086/da.29048},       % XXX = number of paper, e.g. da006 for paper#6
}   %%% END \dajEDITORdetails

\begin{document}

\begin{frontmatter}[classification=text]
%% EDITOR: this will force the keywords to appear right after the Abstract.
%%   If the abstract is too long and would force the keywords off the
%%   front page, please comment out % [classification=text] above
%%   This way the keywords will be floated on the bottom of the first page
%%   even though the Abstract spills over to the next page.

%%% AUTHOR: Title goes here.  This line is optional.  You must use it
%%   if title has footnote attached or requires nontrivial typesetting,
%%   e.g., inclusion of linebreaks to force nice layout.
\title{The Bombieri--Vinogradov Theorem for Nilsequences} %% please capitalize all significant words

%%% AUTHOR:
%%% List all authors. If you wish, place grant acknowledgements in \thanks.
%%% In brackets include a short tag for each author.
\author[XS]{Xuancheng Shao\thanks{Supported by the NSF grant DMS-1802224.}}
\author[JT]{Joni Ter\"{a}v\"{a}inen\thanks{Supported by  a Titchmarsh Research Fellowship.}}

%%% AUTHOR: Abstract goes here
\begin{abstract}
We establish results of Bombieri--Vinogradov type for the von Mangoldt function $\Lambda(n)$ twisted by a nilsequence. In particular, we obtain Bombieri--Vinogradov type results for the von Mangoldt function twisted by any polynomial phase $e(P(n))$; the results obtained are as strong as the ones previously known in the case of linear exponential twists. We derive a number of applications of these results. Firstly, we show that the primes $p$ obeying a ``nil-Bohr set'' condition, such as $\|\alpha p^k\|<\varepsilon$, exhibit bounded gaps. Secondly, we show that the Chen primes are well-distributed in nil-Bohr sets, generalizing a result of Matom\"aki. Thirdly, we generalize the Green--Tao result on linear equations in the primes to primes belonging to an arithmetic progression to large modulus $q\leq x^{\theta}$, for almost all $q$. 
\end{abstract}
\end{frontmatter}

%%% AUTHOR: body of paper starts here
\section{Introduction}

The celebrated Bombieri--Vinogradov theorem states that
\begin{align*}
\sum_{d\leq x^{1/2-\varepsilon}}\max_{(c,d)=1}\Big|\sum_{\substack{n\leq x\\n\equiv c\Mod d}}\Lambda(n)-\frac{x}{\varphi(d)}\Big|\ll_{A,\varepsilon}\frac{x}{(\log x)^{A}},
\end{align*}
thus proving equidistribution of the von Mangoldt function in all residue classes to almost all moduli $d\leq x^{1/2-\varepsilon}$. The $x^{1/2-\varepsilon}$ threshold can further be improved to $x^{1/2}/(\log x)^{B}$ for suitable $B=B(A)$, but either with or without this improvement the conclusion can be stated as saying that the primes have \emph{level of distribution} $1/2$. 

It is natural to study whether other sequences related to the primes or various arithmetic functions also satisfy bounds of Bombieri--Vinogradov type.  Bombieri--Vinogradov type estimates for general $1$-bounded multiplicative functions were investigated in \cite{granville-shao, granville-shao2}, and there are a number of works on the level of distribution of the smooth numbers \cite{Fouvry-Tenenbaum, Harper} and of the $k$-fold divisor  functions $d_k(n)$ (e.g. \cite{divisor-FI, divisor-FI2, divisor-HB} for all moduli and~\cite{FKM-divisor, FR-divisor} for almost all moduli),  and the literature is swarming with many  other  interesting  examples  besides. 

Our object in this paper is to obtain level of distribution results for another natural class of functions, namely twists $\Lambda(n)e(P(n))$ of the von Mangoldt function by polynomial phases and, more generally, nilsequences. There has been previous work on the case of linear polynomials $P$; we recall these results later in this introduction. Let us first present the necessary definitions for stating our main theorems; in Section \ref{sec:apps} we present their applications. 

\subsection{Results for nilsequence twists}

In order to state our results for nilsequence twists, we need a few definitions. The results for polynomial phases will be deduced as special cases in Subsection \ref{subsec:poly}, and they do not require knowledge of nilsequences. For an in-depth discussion of nilsequences and their importance in additive combinatorics, see \cite{tao-higher}. 

\begin{definition}[Nilsequences] Let $G$ be a connected, simply-connected nilpotent Lie group, and let $\Gamma\leq G$ be a lattice. By a \emph{filtration} $G_{\bullet}=(G_i)_{i=0}^{\infty}$ on $G$, we mean an infinite sequence of subgroups of $G$ (which are also connected, simply-connected nilpotent Lie groups) such that
\begin{align*}
G=G_0=G_1\supset G_2\supset\cdots     
\end{align*}
and such that the commutators satisfy $[G_i,G_j]\subset G_{i+j}$, and with the additional properties that $\Gamma_i:=\Gamma\cap G_i$ is a lattice in $G_i$ for $i\geq 0$ and $G_{s+1}=\{\textnormal{id}\}$ for some $s$. 

The least such $s$ is called the \emph{degree} of $G_{\bullet}$.

A \emph{polynomial sequence} on $G$ (adapted to the filtration $G_{\bullet}$) is any sequence $g:\mathbb{Z}\to G$ satisfying the derivative condition
\begin{align*}
\partial_{h_1}\cdots \partial_{h_k}g(n) \in G_k     
\end{align*}
for all $k\geq 0$, $n\in \Z$ and all $h_1,\ldots, h_k\in \mathbb{Z}$, where $\partial_h g(n):=g(n+h)g(n)^{-1}$ is the discrete derivative with shift $h$.

Finally, if $\varphi:G/\Gamma\to \mathbb{C}$ is Lipschitz with respect to a natural metric on $G/\Gamma$ (induced by a Mal'cev basis; see definition below), we call a sequence of the form $n \mapsto \varphi(g(n)\Gamma)$ a \emph{nilsequence}.
\end{definition}

Since nilsequences are a vast class of functions, it is natural to restrict to those that have ``bounded complexity''. This is made precise in the following definition.

\begin{definition}[Bounded complexity nilsequences]\label{def:nil}
For a positive integer $s$ and real numbers $\Delta, K \geq 2$, we define $\Psi_s(\Delta, K)$ to be the class of all nilsequences $\psi \colon \Z \to \C$ of the form $\psi(n) = \varphi(g(n)\Gamma)$, where
\begin{enumerate}
\item $G/\Gamma$ is a nilmanifold of dimension at most $\Delta$, equipped with a filtration $G_{\bullet}$ of degree at most $s$ and a $K$-rational Mal'cev basis $\mathcal{X}$ (defined in \cite[Definition 2.1, Definition 2.4]{green-tao-nilmanifolds});
\item $g \colon \mathbb{Z} \to G$ is a polynomial sequence adapted to $G_{\bullet}$;
\item $\varphi \colon G/\Gamma \to \C$ is a Lipschitz function with $\|\varphi\|_{\operatorname{Lip}} \leq 1$, where the Lipschitz norm is defined as
\[ \|\varphi\|_{\operatorname{Lip}} = \|\varphi\|_{\infty} + \sup_{x,y \in G/\Gamma, x\neq y} \frac{|F(x) - F(y)|}{d_{\mathcal{X}}(x,y)}. \] 
The metric $d_{\mathcal{X}}$ appearing above is defined in \cite[Definition 2.2]{green-tao-nilmanifolds}) using the Mal'cev basis $\mathcal{X}$.
\end{enumerate}
\end{definition}

The important examples to keep in mind are the polynomial phase functions $\psi(n) = e(g(n))$, where $g$ is a polynomial of degree at most $s$. Here the relevant nilmanifold is $G/\Gamma=\mathbb{R}/\mathbb{Z}$ with the simple filtration $G_k=\mathbb{R}/\mathbb{Z}$ for $k\leq s$ and $G_k=\{0\}$ for $k>s$. We also note that any \emph{bracket polynomial}, such as $e(P_1(n) \lfloor P_2(n)\rfloor)$, with $P_1(x),P_2(x)\in \mathbb{R}[x]$, is essentially a nilsequence, in the sense that (by smoothing the fractional part function a bit, as in  \cite[p. 102]{tao-higher}) this function for $n\leq x$ be written as a linear combination with bounded coefficients of $\ll (\log x)^{A}$ nilsequences $\psi\in \Psi_{s}(\Delta,K)$ for some $s,\Delta, K\ll 1$, plus an error term that is $O((\log x)^{-A/10})$ in $\ell^1$ norm.

We make a technical remark on the function $\varphi$ appearing in Definition~\ref{def:nil}(3). When applying the machinery of nilsequences to problems in additive combinatorics, it is arguably more convenient to work with smooth functions $\varphi$ (with controlled smoothness norms) instead of Lipschitz functions, although the use of Lipschitz functions has by now become standard. All the results proved in this paper would remain true with this alternative definition of nilsequences.

In order to state our main theorems, we need the $W$-trick.  For $w \geq 2$, we write $\mathscr{P}(w) := \prod_{p \leq w}p$.

\begin{theorem}\label{thm:1/4}
 Let an integer $s\geq 1$, large real numbers $A, \Delta\geq 2$, and a small real number $\varepsilon\in (0,1/4)$ be given. Then for any $x \geq 2$, we have
\begin{align*}
\sum_{d\leq x^{1/4-\varepsilon}}\max_{(c,d)=1} \sup_{\psi \in \Psi_s(\Delta, \log x)} \Big|\sum_{\substack{n \leq x \\ n \equiv c\Mod{d}}} \Lambda(n) \psi(n) -  \frac{dW}{\varphi(dW)}\sum_{\substack{n \leq x \\(n,W)=1\\ n\equiv c\Mod{d}}} \psi(n) \Big|\ll_{s,A,\Delta,\ee} \frac{x}{(\log x)^A},
\end{align*}
where $W = \mathscr{P}((\log x)^C)$ for some constant $C = C(A,s,\Delta,\ee)$.
\end{theorem}

We can increase the level of distribution to $1/3$ if the nilsequence $\psi$ is \emph{fixed} (does not depend on $d$).

\begin{theorem}\label{thm:1/3}
 Let an integer $s\geq 1$, large real numbers $A, \Delta\geq 2$, and a small real number $\varepsilon\in (0,1/3)$ be given. Then for any nilsequence $\psi \in \Psi_s(\Delta, \log x)$ and $x \geq 2$, we have
\begin{align*}
\sum_{d\leq x^{1/3-\varepsilon}}\max_{(c,d)=1} \Big|\sum_{\substack{n \leq x \\ n \equiv c\Mod{d}}} \Lambda(n) \psi(n) -\frac{dW}{\varphi(dW)}\sum_{\substack{n \leq x \\(n,W)=1\\ n\equiv c\Mod{d}}} \psi(n) \Big|\ll_{s,A,\Delta,\ee} \frac{x}{(\log x)^A},
\end{align*}
where $W = \mathscr{P}((\log x)^C)$ for some constant $C = C(A,s,\Delta,\ee)$.
\end{theorem}

We can further increase the level of distribution to $1/2$ if $c$ is fixed (does not depend on $d$) and the absolute value inside the $d$ sum is replaced by a well-factorable weight, defined in the following definition.

\begin{definition}[Well-factorable sequences] We say that a sequence $(\lambda_d)$ of real numbers is \emph{well-factorable} of level $D$ if $(\lambda_d)$ is supported on $d \in [1,D]$ and for any $1\leq R,S\leq D$ with $RS=D$ one can write $\lambda_d=\beta*\gamma(d) := \sum_{d=d_1d_2}\beta(d_1)\gamma(d_2)$ for some sequences $\beta(d), \gamma(d)$ of modulus at most $1$ and such that $\beta$ is supported on $[1,R]$ and $\gamma$ is supported on $[1,S]$.
\end{definition}

Well-factorable weights arise in many sieve problems due to the fact that the linear sieve weights (introduced by Iwaniec) are a bounded linear combination of well-factorable weights; see \cite[Lemma 12.16]{operadecribro}. Bombieri--Friedlander--Iwaniec~\cite{BFI} famously broke the $1/2$ barrier for the level of distribution in the Bombieri--Vinogradov inequality, provided that the absolute value signs are replaced by well-factorable weights and the residue class is fixed for all $q$. In this setting, we can also do better than Theorems~\ref{thm:1/4} and~\ref{thm:1/3}, despite not being able to break the $1/2$ barrier.

\begin{theorem}\label{thm:1/2}
 Let integers $s\geq 1$, $c \neq 0$, large real numbers $A, \Delta \geq 2$, and a small real number $\varepsilon\in (0,1/2)$ be given. Then for any nilsequence $\psi \in \Psi_s(\Delta, \log x)$ and any well-factorable sequence $(\lambda_d)$ of level $x^{1/2-\varepsilon}$ with $x \geq 2$,  we have
\begin{align*}
\Big|\sum_{\substack{d\leq x^{1/2-\varepsilon} \\ (d,c)=1}} \lambda_d \Big(\sum_{\substack{n \leq x\\ n \equiv c\Mod{d}}} \Lambda(n) \psi(n) -  \frac{dW}{\varphi(dW)}\sum_{\substack{n \leq x \\(n,W)=1\\ n\equiv c\Mod{d}}} \psi(n) \Big)\Big| \ll_{s,A,\Delta,\ee,c} \frac{x}{(\log x)^A},
\end{align*}
where $W = \mathscr{P}((\log x)^C)$ for some constant $C = C(A,s,\Delta,\ee)$.
\end{theorem}

\begin{remark}
In Theorems~\ref{thm:1/4},~\ref{thm:1/3}, and~\ref{thm:1/2}, in order to get an arbitrary power of log saving in the error term, it is necessary to perform a ``W-trick" to overcome the fact that functions such as $\Lambda(n)e(\frac{an^s}{q})$ are not equidistributed in all residue classes when $q\leq (\log x)^C$. Since this leads to the choice $W = \mathscr{P}((\log x)^C)$ which is rather large, we are unable to perform the same W-trick as in~\cite{green-tao-linear} which compares the function $n \mapsto \frac{\varphi(W)}{W}\Lambda(Wn+1)$ to $1$. We have adapted an alternative approach which, roughly speaking, compares $\Lambda$ to (a normalized version of) the function $n\mapsto 1_{(n,W)=1}$.
\end{remark}

\begin{remark}\label{rem:mobius}
One can also obtain similar results for the M\"{o}bius function, without the need for the $W$-trick and with no main term. Specifically, with the notations and assumptions of Theorem \ref{thm:1/4}, we have
\begin{align*}
\sum_{d\leq x^{1/4-\varepsilon}}\max_{(c,d)=1} \sup_{\psi \in \Psi_s(\Delta, \log x)} \Big|\sum_{\substack{n \leq x \\ n \equiv c\Mod{d}}} \mu(n) \psi(n)\Big|\ll_{s,A,\Delta,\ee} \frac{x}{(\log x)^A}.
\end{align*}
With the notations and assumptions of Theorem \ref{thm:1/3}, we have
\begin{align*}
\sum_{d\leq x^{1/3-\varepsilon}}\max_{(c,d)=1} \Big|\sum_{\substack{n \leq x \\ n \equiv c\Mod{d}}} \mu(n) \psi(n) \Big|\ll_{s,A,\Delta,\ee} \frac{x}{(\log x)^A}.
\end{align*}
 With the notations and assumptions of Theorem \ref{thm:1/2}, we have
\begin{align*}
\Big|\sum_{\substack{d\leq x^{1/2-\varepsilon} \\ (d,c)=1}} \lambda_d \Big(\sum_{\substack{n \leq x\\ n \equiv c\Mod{d}}} \mu(n) \psi(n)\Big)\Big| \ll_{s,A,\Delta,\ee,c} \frac{x}{(\log x)^A}.
\end{align*}
These statements can be proved with arguments almost identical to Theorems~\ref{thm:1/4},~\ref{thm:1/3} and~\ref{thm:1/2}, using the fact that the M\"{o}bius function obeys an identity analogous to Vaughan's identity for $\Lambda$.
\end{remark}

In fact, our results above also apply to a number of other multiplicative functions, such as the divisor functions $d_k(n)$ (for all values of $k$, including non-integer and complex values) and the indicator function $1_{S}(n)$ of sums of two squares. For these functions, however, one would have to modify the main term sum involving $(n,W)=1$ in Theorems \ref{thm:1/4}, \ref{thm:1/3} and \ref{thm:1/2}. For simplicity, and to reduce repetition in the arguments, we only state Bombieri--Vinogradov type theorems for these functions in the equidistributed (``minor arc'') case where there is no main term; see Theorem \ref{thm_divisor}.

\subsection{Results for polynomial phase twists}\label{subsec:poly}

Since polynomial phases of the form $e(P(n))$ are examples of nilsequences of bounded complexity, Theorems \ref{thm:1/4}, \ref{thm:1/3} and \ref{thm:1/2} immediately imply as special cases Bombieri--Vinogradov type estimates for $\Lambda(n)e(P(n))$. We state these below, since they are of independent interest and since they will be utilized in deriving some of the applications of our results. 

\begin{corollary}\label{cor:1/4}
 Let an integer $s\geq 1$, a large real number $A\geq 2$, and a small real number $\varepsilon\in (0,1/4)$ be given. Then for any $x \geq 2$, we have
\begin{align*}
\sum_{d\leq x^{1/4-\varepsilon}}\max_{(c,d)=1} \sup_{\deg(P)\leq s} \Big|\sum_{\substack{n \leq x \\ n \equiv c\Mod{d}}} \Lambda(n) e(P(n)) - \frac{dW}{\varphi(dW)} \sum_{\substack{n \leq x \\(n,W)=1\\ n\equiv c\Mod{d}}} e(P(n)) \Big|\ll_{s,A,\varepsilon} \frac{x}{(\log x)^A},
\end{align*}
where $W = \mathscr{P}((\log x)^C)$ for some constant $C = C(A,s,\ee)$.
\end{corollary}

\begin{corollary}\label{cor:1/3}
 Let an integer $s\geq 1$, a large real number $A \geq 2$, and a small real number $\varepsilon\in (0,1/3)$ be given. Then for any polynomial $P(x)\in \mathbb{R}[x]$ of degree $s$ and $x \geq 2$, we have
\begin{align*}
\sum_{d\leq x^{1/3-\varepsilon}}\max_{(c,d)=1} \Big|\sum_{\substack{n \leq x \\ n \equiv c\Mod{d}}} \Lambda(n) e(P(n)) -\frac{dW}{\varphi(dW)}\sum_{\substack{n \leq x \\(n,W)=1\\ n\equiv c\Mod{d}}} e(P(n)) \Big|\ll_{s,A,\varepsilon} \frac{x}{(\log x)^A},
\end{align*}
where $W = \mathscr{P}((\log x)^C)$ for some constant $C = C(A,s,\ee)$.
\end{corollary}

\begin{corollary}\label{cor:1/2}
 Let integers $s\geq 1$, $c \neq 0$, a large real number $A\geq 2$, and a small real number $\varepsilon\in (0,1/2)$ be given. Then for any polynomial $P(x)\in \mathbb{R}[x]$ of degree $s$ and any well-factorable sequence $(\lambda_d)$ of level $x^{1/2-\varepsilon}$ with $x \geq 2$,  we have
\begin{align*}
\Big|\sum_{\substack{d\leq x^{1/2-\varepsilon} \\ (d,c)=1}} \lambda_d \Big(\sum_{\substack{n \leq x \\ n \equiv c\Mod{d}}} \Lambda(n) e(P(n)) - \frac{dW}{\varphi(dW)}\sum_{\substack{n \leq x \\(n,W)=1\\ n\equiv c\Mod{d}}} e(P(n)) \Big)\Big| \ll_{s,A,\varepsilon,c} \frac{x}{(\log x)^A},
\end{align*}
where $W = \mathscr{P}((\log x)^C)$ for some constant $C = C(A,s,\ee)$.
\end{corollary}

\begin{remark}
We could also obtain analogous results for bracket polynomial phases, so for example $e(P_1(n)\lfloor P_2(n)\rfloor)$, where $P_1,P_2$ are polynomials. One simply needs the fact that these functions are well-approximable by nilsequences of bounded complexity, a property that was noted above. We leave the details to the interested reader.
\end{remark}

Previous results related to our main theorems are as follows:
\begin{itemize}
    \item For the M\"obius function, it was established in \cite{shao} for $Q< x^{1/2}$ that
    \begin{align}\label{eq22}
    \max_{(c,d)=1} \sup_{\psi \in \Psi_s(\Delta, K)} \Big|\sum_{\substack{n \leq x \\ n \equiv c\Mod{d}}} \mu(n) \psi(n)\Big|\ll_{s,A,\Delta,K} \frac{x\log \log x}{Q(\log(x/Q^2))}
    \end{align}
    for almost all $d\in [Q,2Q]$, in the sense that the number of exceptional $d$ is $\ll_A Q(\log x)^{-A}$. Although this result is applicable for $d\leq x^{1/2-\varepsilon}$, it saves a factor of $(\log \log x)/(\log x)$ at best. The proof relies crucially on almost all numbers having prime factors in various suitable ranges, and hence it does not work for the case of primes, i.e. the von Mangoldt function.
 
    \item 
    For linear exponentials (that is, $s=1$), Theorem \ref{thm:1/3} was proved by Todorova and Tolev \cite{todorova-tolev}, and for quadratic phase functions (which is a special case of the $s=2$ case) by Tolev \cite{tolev-bombieri}.
    \item It was shown by Matom\"aki \cite{matomaki-bombieri}, improving on work of Mikawa \cite{mikawa}, that the well-factorable level of distribution estimate given by Theorem \ref{thm:1/2} holds for  $s=1$.
 \end{itemize}

\subsection{Notation}

We use the usual asymptotic notation $\ll$, $\gg$, $O(\cdot)$, $o(\cdot)$, $\asymp$. Dependence of these symbols on parameters is indicated whenever such a dependence occurs (so, for example, $o_{w;x\to \infty}(1)$ is a quantity depending on $w$ and $x$ and tending to $0$ as $x\to \infty$). 

We use $\Lambda$ to denote the von Mangoldt function, $\mu$ to denote the M\"obius function, $\varphi$ to denote the Euler phi function, $d_k$ to denote the $k$-fold divisor function (with $d(n):=d_2(n)$) and $(a,b)$ to denote the greatest common divisor of $a$ and $b$. We also let $\mathscr{P}(w):=\prod_{p\leq w}p$.

Let $\|\cdot\|_{U^k(\mathbb{Z}/N\Z)}$ stand for the usual Gowers norm over the cyclic group $\mathbb{Z}/N\Z$. Given a function $f:\Z\to \mathbb{C}$ supported on $[N] := \{1,2,\cdots,N\}$, we define its Gowers norm $\|f\|_{U^k[N]}$ over the interval $[N]$ as
\begin{align*}
\|f\|_{U^k[N]}:=\frac{\|f\cdot 1_{[N]}\|_{U^k(\mathbb{Z}/N'\Z)}}{\|1_{[N]}\|_{U^k(\mathbb{Z}/N'\Z)}},    
\end{align*}
where $N'>2N$ (say $N'=2N+1$ for concreteness) and $f\cdot 1_{[N]}$ and $1_{[N]}$ are extended to $\mathbb{Z}/N'\Z$ in the natural way. 

If $A$ is any nonempty, finite set and $f:A\to \mathbb{C}$ is a function, we use the averaging notation
\begin{align*}
\mathbb{E}_{a\in A}f(a):=\frac{1}{|A|}\sum_{a\in A}f(a).    
\end{align*}

For $x \in \R$, we use $\|x\|$ to denote the distance from $x$ to its nearest integer.

%
%\subsection{Acknowledgments}
%
% We thank Ben Green, James Maynard and Terence Tao for helpful comments. We are grateful to the anonymous referee for a careful reading of the paper and for comments that improved the quality of this paper.
%X.S.~was supported by the NSF grant DMS-1802224.
%J.T.~was supported by a Titchmarsh Research Fellowship.

\section{Applications}\label{sec:apps}

We now present several applications of our Bombieri--Vinogradov type theorems to problems related to Diophantine properties of the primes, as well as to additive combinatorics. 

\subsection{Bounded gaps between primes in Bohr sets}

Our first application generalizes the celebrated result of Zhang \cite{zhang}, Maynard \cite{maynard} and Tao (unpublished) on bounded gaps between primes. Subsequent to these works, a number of interesting subsets of the primes have also been shown to exhibit bounded gaps. See \cite{alweiss-luo}, \cite{maynard-compositio} for primes in short intervals, \cite{thorner}, \cite{maynard-compositio} for work on Chebotarev sets, and \cite{baker-zhao}, \cite{smith} for work on primes in Beatty sequences $(\lfloor \alpha n+\beta\rfloor)_{n\geq 1}$. 

As a consequence of our main theorems, we are able to add to this list that the primes lying in a \emph{nil-Bohr set} exhibit bounded gaps. This generalizes the result of Baker and Zhao \cite{baker-zhao} mentioned above, which corresponds to classical Bohr sets of a special form.

Nil-Bohr sets were introduced by Host and Kra in \cite{host-kra-nilbohr} and are a natural generalization of classical Bohr sets to the setting of higher order Fourier analysis. For the convenience of the reader, we first define classical Bohr sets and then nil-Bohr sets.

\begin{definition}[Bohr sets]
Let $r\geq 0$ and $\rho>0$. Then for any real numbers $\alpha_1,\ldots, \alpha_r\in \mathbb{R}$ the set
\begin{align*}
B=\{n\in \mathbb{Z}:\,\, \|\alpha_i n\|<\rho\,\, \forall 1\leq i\leq r\}    
\end{align*}
is called a (classical) \emph{Bohr set} (of rank $r$), where $\|x\|$ denotes the distance from $x$ to its nearest integer.
\end{definition}

\begin{definition}[nil-Bohr sets]
Let $U$ be an open subset of $\mathbb{C}$, and let $\psi:G/\Gamma\to \mathbb{C}$ be a nilsequence defined on some nilmanifold. Then the set
\begin{align*}
B=\{n\in \mathbb{Z}:\,\, \psi(n)\in U\}    
\end{align*}
is called a \emph{nil-Bohr set}. 
\end{definition}

Note that any Bohr set is also a nil-Bohr set. In fact, a Bohr set of any rank $r$ can be represented as a nil-Bohr set, by taking $G/\Gamma = (\R/\Z)^r$ and $\psi, U$ appropriately.
Moreover, any polynomial Bohr set of the form 
\begin{align*}
\{n\in \mathbb{Z}:\,\,\|Q(n)\|< \rho\}    
\end{align*}
is a nil-Bohr set for any polynomial $Q(x)$ and any $\rho\in (0,1/2)$; take $\psi(n)=e(Q(n))$ as the nilsequence in the definition, and take $U$ to be $\{z\in \mathbb{C}\setminus\{0\}:\,\, \arg(z)\in (-2\pi\rho,2\pi \rho)\}$.
Any nil-Bohr set that is ``irrational'' in a suitable sense (see Remark \ref{rem:irrat}) contains infinitely many primes. This is however not a trivial fact; if $Q(x)$ is a polynomial with irrational leading coefficient, then proving the asymptotic
\begin{align}\label{eq23}
|\{p\leq x:\,\, \|Q(p)\|<\rho\}|=(2\rho+o(1))\frac{x}{\log x}    
\end{align}
is related to Weyl's equidistribution criterion, exponential sum estimates in the Waring--Goldbach problem, and, in the case of general nil-Bohr sets, to the Green--Tao result \cite{green-tao-mobius} on M\"obius orthogonality with nilsequences. Also note that nil-Bohr sets (as well as classical Bohr sets) can be rather irregular: in the asymptotic formula \eqref{eq23} it is not possible to specify the $o(1)$ term without imposing a restriction on how fast the denominators of the continued fraction convergents of the leading coefficient of $Q(x)$ grow.

In what follows, we say that a set $S\subset \mathbb{N}$ has \emph{bounded gaps} if there exists $C>0$ such that the inequality $0<s_1-s_2\leq C$ has infinitely many solutions with $s_1,s_2\in S$.

\begin{theorem}[Bounded gaps between primes in nil-Bohr sets]\label{thm:bdd}  Let $Q(x)\in \mathbb{R}[x]$ be a polynomial with at least one irrational coefficient which is not the constant term. Let $\rho\in (0,1/2)$, and form the nil-Bohr set
\begin{align*}
B=\{n\in \mathbb{Z}:\,\, \|Q(n)\|< \rho\}.    
\end{align*}
Then the subset of the primes $\mathbb{P}\cap B$ has bounded gaps.
\end{theorem}

\begin{remark}\label{rem:irrat}
It is only for the sake of simplicity that we restrict to nil-Bohr sets of this form; likely similar arguments could be made to work for any nil-Bohr set $B$ that 
satisfies the following two natural properties: (i) $B$ is \emph{irrational} in the sense that, in the notation of Definition \ref{def:equi}, for some constant $c>0$ and infinitely many $x\in \mathbb{N}$ we have $\psi\in \Psi_s(\Delta,K;x^{-c}, x)$ with $s,\Delta, K$ fixed. (ii) $B$ is \emph{dense} in the sense that $|B\cap [1,x]|\gg x$. We leave the details of this generalization to the interested reader.
\end{remark}

\subsection{Chen primes in Bohr sets}

Our next application involves Chen primes, which are primes $p$ such that $p+2\in P_2$, where $P_2$ is the set of positive integers with at most two prime factors. We write
\begin{align*}
\mathcal{P}_{\textnormal{Chen}}=\{p\in \mathbb{P}:\,\, p+2\in P_2\}.    
\end{align*}
A celebrated result of Chen \cite{chen} shows that $\mathcal{P}_{\textnormal{Chen}}$ is infinite. It is moreover a sparse subset of the primes, since it can be shown to satisfy $|\mathcal{P}_{\textnormal{Chen}}\cap[1,x]|\asymp x/(\log x)^2$. 

It was shown by Matom\"aki \cite{matomaki-bombieri} that the Chen primes are well-distributed in classical Bohr sets, meaning that there exists $\theta>0$ such that
\begin{align*}
\|\alpha p\|<p^{-\theta}    
\end{align*}
has infinitely many solutions in Chen primes $p$ for any fixed irrational $\alpha$.

We generalize this by proving that the Chen primes are well-distributed in more general nil-Bohr sets. 

\begin{theorem}[Chen primes in nil-Bohr sets]\label{thm:chen} Let $Q(x)\in \mathbb{R}[x]$ be a polynomial of degree $s\geq 1$ such that $Q$ has at least one irrational coefficient which is not the constant term. Then for some constant $\theta_s>0$ (independent of $Q$) there are infinitely many solutions to
\begin{align*}
\|Q(p)\|< p^{-\theta_s}, \quad p\in \mathcal{P}_{\textnormal{Chen}}.    
\end{align*}
\end{theorem}

\begin{remark}
As in the case of Theorem \ref{thm:bdd}, it should be possible to generalize our proof method to showing that for any irrational and dense nil-Bohr set $B$ we have infinitely many solutions to $p\in \mathcal{P}_{\textnormal{Chen}}\cap B$. We leave the details to the interested reader.
\end{remark}

\subsection{Linear equations in  primes in arithmetic progressions}

The next application is a generalization of the groundbreaking result proved by Green--Tao \cite{green-tao-linear,green-tao-mobius}  and Green--Tao--Ziegler \cite{gtz} that one can asymptotically count the number of solutions to any linear system of equations in the primes (of finite complexity, which excludes for instance counting twin primes or solutions to the binary Goldbach problem). It is natural to consider this problem for subsets of the primes as well, in particular for primes belonging to an arithmetic progression. We show that in the rather large range $q\leq x^{\theta}$, for suitable $\theta>0$, we still obtain asymptotics for linear equations in primes restricted to a congruence class $a\Mod q$, for almost all choices of $q$.

In what follows, we set
\begin{align}\label{equ100}
\Lambda_{a,q}(n):=\frac{\varphi(q)}{q}\Lambda(qn+a),    
\end{align}
which is a normalized version of the von Mangoldt function restricted to the arithmetic progression $a\Mod q$. Extend $\Lambda_{a,q}(n)$ to integers $n\leq 0$ by setting it to be zero at them. Further, for a system $\Psi(n)=(L_1(\mathbf{n}),\ldots, L_t(\mathbf{n}))$ of affine linear forms in $d$ variables we define its \emph{size} at scale $N$ to be
\begin{align*}
\|\Psi\|_N:=\sum_{i=1}^t \sum_{j=1}^d |L_i(e_j)-L_i(0)|+\sum_{i=1}^t \Big|\frac{L_i(0)}{N}\Big|,
\end{align*}
where $e_1,\ldots, e_d$ are the standard basis vectors in $\mathbb{Z}^d$.

\begin{theorem}[Linear equations in primes in arithmetic progressions to large moduli]\label{thm:app3}
Let $\varepsilon>0$ and $A,t,d,M\geq 1$ be given. Let $x\geq 10$ and $Q\leq x^{1/3-\varepsilon}$. Then for all but $\ll_{\varepsilon,A,t,d,M}Q/(\log x)^{A}$ choices of $1\leq q\leq Q$ the following holds. For every $\ve{a} \in (\Z/q\Z)^d$ and every finite complexity tuple $\Psi=(L_1(\mathbf{n}),\ldots, L_t(\mathbf{n}))$ of non-constant affine-linear forms in $d$ variables of size $\|\Psi\|_{x}\leq M$ we have
\begin{align}\label{eq21}
\sum_{\mathbf{n}\in [1,x]^d}\Lambda(L_1(q\mathbf{n}+\ve{a}))\cdots \Lambda(L_t(q\mathbf{n}+\ve{a}))=\beta_{\infty}\prod_{p}\beta_{p,a,q}+o_{t,d,M}(x^{d}),     
\end{align}
where the product on the right-hand side is convergent, $\beta_{p,a,q}\geq 0$, $0\leq \beta_{\infty}\ll x^{d}$, and the local factors $\beta_{p,a,q}$ are given by
\begin{align*}
\beta_{p,a,q}:=\mathbb{E}_{\ve{n}\in (\mathbb{Z}/p\mathbb{Z})^d}\prod_{i=1}^t \Lambda_{\mathbb{Z}/p\mathbb{Z}}(L_i(q\ve{n}+\ve{a})),    
\end{align*}
and $\Lambda_{\mathbb{Z}/p\mathbb{Z}}(b)=\frac{p}{p-1}1_{(n,p)=1}$, and $\beta_{\infty}=\textnormal{vol}_d([1,x]^d\cap \Psi^{-1}(\mathbb{R}_{+}^t))$.
\end{theorem}

We also obtain a similar theorem (without the main term) with the M\"obius function in place of the von Mangoldt function; see Proposition \ref{app3-mobius}. Note that the condition $Q \leq x^{1/3-\ee}$ actually corresponds to a $1/4$ level of distribution, since  the primes being counted are of size $\approx Qx$.

\begin{remark}
Theorem \ref{thm:app3} can be viewed as generalizing a result of Bienvenu~\cite{bienvenu} on linear equations in primes in the Siegel--Walfisz regime, corresponding to $q\ll_A (\log x)^{A}$ in \eqref{eq21}. As in the Siegel--Walfisz theorem, there are no exceptional moduli $q$ in \cite[Theorem 1.3]{bienvenu}. The same applies to our Theorem~\ref{thm:app3} as well, since if we are considering moduli of size $\leq Q=(\log x)^{A}$, the theorem gives $\ll_{A} (\log x)^{-2A}$ exceptional moduli, and this quantity is clearly less than $1$ for $x$ large. 
\end{remark}

In the course of proving Theorem~\ref{thm:app3}, we obtain Gowers uniformity of primes in almost all arithmetic progressions.

\begin{theorem}[Gowers uniformity of primes in arithmetic progressions to large moduli]
\label{thm:app3'}
Let $\ee > 0$ and $A,k \geq 1$ be given. Let $x \geq 10$, $Q \leq x^{1/3-\ee}$, and $w \geq 1$. Then for all but $\ll_{\ee,A,k} Q/(\log x)^A$ choices of $1 \leq q \leq Q$ with $\mathscr{P}(w) \mid q$ the following holds. For every $(a,q)=1$, the function $\Lambda_{a,q}$ defined by \eqref{equ100}
satisfies the  Gowers norm bound
\[ \|\Lambda_{a,q}-1\|_{U^k[x]}=o_{w\to \infty}(1)+o_{w;x\to \infty}(1). \]
\end{theorem}

We conclude with an immediate corollary of Theorem~\ref{thm:app3} to long arithmetic progressions in the primes.

\begin{corollary}[Green--Tao theorem for primes in arithmetic progressions to large moduli] Let $x\geq 10$, and let $k\geq 1$, $A\geq 1$ and $\varepsilon>0$ be given. Also let $Q\leq x^{1/4-\varepsilon}$. Then for all but $\ll_{\varepsilon, A,k}Q/(\log x)^{A}$ moduli $q\leq Q$, each of the sets $\mathbb{P}\cap (q\mathbb{Z}+a)$ with $(a,q)=1$ contains a nontrivial $k$-term arithmetic progression all of whose elements are $\leq x$. 
\end{corollary}

\subsection{Future work}

In ongoing work with P.-Y. Bienvenu, we develop a transference principle for arbitrary systems of linear equations (not necessarily translation-invariant) in the spirit of~\cite{green-tao-APs, green-tao-linear}, generalizing the transference principle of \cite{matomaki-shao}. This new transference principle requires as a key input that the sparse set we want to transfer is ``well--distributed'' in nil-Bohr sets, which using our Bombieri--Vinogradov theorems (Theorems \ref{thm:1/4}, \ref{thm:1/3}, \ref{thm:1/2}) we can verify for many subsets of the primes of interest. In particular, we will be able to generalize the Green--Tao linear equations in primes result to Chen primes. 

\section{Quantitative equidistribution of polynomial sequences}

To understand the main idea behind the proof of our main results (Theorems~\ref{thm:1/4},~\ref{thm:1/3} and~\ref{thm:1/2}), it is not essential to know the precise definitions around nilsequences,  as long as one is willing to accept certain results about nilsequences as black boxes, many of which can be found in~\cite{green-tao-nilmanifolds}.  The readers are thus encouraged to consider the special case when the nilsequence $\psi(n)$ is a polynomial phase function (Corollaries~\ref{cor:1/4},~\ref{cor:1/3} and~\ref{cor:1/2}).
The purpose of this section is to collect definitions and results about equidistribution of polynomial sequences that are used in our proofs. These can be thought of as generalizations of major and minor arc analysis to treat exponential sums, and of Weyl's inequality.

\begin{definition}[Equidistribution]
Let $G/\Gamma$ be a nilmanifold equipped with a Mal'cev basis $\mathcal{X}$. Let $N \geq 1$ and $\delta > 0$.
\begin{enumerate}
\item A finite sequence $(g(n)\Gamma)_{1 \leq n \leq N}$ is said to be \emph{$\delta$-equidistributed}  if we have
\[ \left| \frac{1}{N}\sum_{n=1}^N F(g(n)\Gamma) - \int_{G/\Gamma} F \right| \leq \delta \|F\|_{\operatorname{Lip}} \]
for all Lipschitz functions $F \colon G/\Gamma \rightarrow \C$. Here the integral is with respect to the unique Haar measure on $G/\Gamma$.
\item A finite sequence $(g(n)\Gamma)_{1 \leq n \leq N}$ is said to be \emph{totally $\delta$-equidistributed}  if we have
\[ \left| \frac{1}{|P|}\sum_{n\in P} F(g(n)\Gamma) - \int_{G/\Gamma} F \right| \leq \delta \|F\|_{\operatorname{Lip}} \]
for all Lipschitz functions $F \colon G/\Gamma \rightarrow \C$ and all arithmetic progressions $P \subset [N]$ of length at least $\delta N$.
\end{enumerate}
\end{definition}

\begin{definition}[Horizontal characters]
Let $G/\Gamma$ be an $m$-dimensional nilmanifold equipped with a Mal'cev basis $\mathcal{X}$. A \emph{horizontal character} is a homomorphism $\eta \colon G \rightarrow \R/\Z$ which annihilates $\Gamma$. It is called \emph{nontrivial} if it is not identically zero. Any horizontal character can be written as $\eta(x) = k \cdot \psi(x)$ for some $k \in \Z^m$, where $\psi \colon G \rightarrow \R^m$ is the Mal'cev coordinate map of $G$ (see~\cite[Definition 2.1]{green-tao-nilmanifolds}). The \emph{modulus} of $\eta$ is defined by $|\eta| := |k|$.
\end{definition}

If $G/\Gamma$ is equipped with a filtration of degree $s$, $\eta$ is a horizontal character on $G/\Gamma$, and $g$ is a polynomial sequence on $G$, then $\eta\circ g: \Z \rightarrow \R/\Z$ is a genuine polynomial of degree at most $s$. This follows from Mal'cev coordinates of polynomial sequences; see~\cite[Lemma 6.7]{green-tao-nilmanifolds}, for example.

\begin{definition}[Smoothness norms]\label{def:smoothness-norm}
Let $g \colon \Z \rightarrow \R/\Z$ be a (genuine) polynomial of degree $s$, written uniquely as
\[ g(n) = \alpha_0 + \alpha_1\binom{n}{1} + \cdots + \alpha_s \binom{n}{s}. \]
For any $N \geq 1$ we define the \emph{smoothness norm}
\[ \|g\|_{C^{\infty}(N)} := \sup_{1 \leq j \leq s} N^j \|\alpha_j\|_{\R/\Z}. \]
\end{definition}

One can easily relate the smoothness norm with the traditional coefficients of a polynomial (see~\cite[Lemma 3.2]{green-tao-mobius}).

\begin{lemma}\label{lem:smoothness-norm}
Let $g \colon \Z \rightarrow \R/\Z$ be a polynomial $g(n) = \beta_0 + \beta_1n + \cdots + \beta_sn^s$, and let $N \geq 1$. There exists a positive integer $q = O_s(1)$ such that
\[ \|q\beta_j\|_{\R/\Z} \ll_s N^{-j} \|g\|_{C^{\infty}(N)} \]
for each $1 \leq j \leq s$. In the other direction, we have
\[ \|g\|_{C^{\infty}(N)} \ll_s \sup_{1 \leq j \leq s} N^j \|\beta_j\|_{\R/\Z}. \]
\end{lemma}

We now relate equidistribution of a polynomial sequence with smoothness norms.

\begin{theorem}[Quantitative Leibman theorem]\label{thm:quant-leibman}
Let $s, \Delta \geq 1$, $0 < \delta < 1/2$ and $N \geq 1$. Let $G/\Gamma$ be a nilmanifold of dimension $\Delta$, equipped with a filtration $G_{\bullet}$ of degree $s$ and a $1/\delta$-rational Mal'cev basis. Let $g \colon \Z \rightarrow G$ be a polynomial sequence adapted to $G_{\bullet}$. If $(g(n)\Gamma)_{1\leq n \leq N}$ is not totally $\delta$-equidistributed, then there is a nontrivial horizontal character $\eta \colon G \rightarrow \R/\Z$ with $|\eta| \ll \delta^{-O_{s,\Delta}(1)}$ such that
\[ \|\eta \circ g\|_{C^{\infty}(N)} \ll \delta^{-O_{s,\Delta}(1)}. \]
\end{theorem}

\begin{proof}
This is proved in~\cite[Theorem 2.9]{green-tao-nilmanifolds} under the marginally stronger hypothesis that $(g(n))_{1 \leq n \leq N}$ is not $\delta$-equidistributed.  Our version can be deduced from~\cite[Theorem 2.9]{green-tao-nilmanifolds} using the same arguments as those from~\cite[Theorem 5.2]{frantzikinakis-host}, as follows. Let $F \colon G/\Gamma \rightarrow \C$ be a Lipschitz function and let $P \subset [N]$ be an arithmetic progression of length $|P| \geq \delta N$, such that
\[ \left| \frac{1}{|P|} \sum_{n \in P} F(g(n)\Gamma) - \int_{G/\Gamma} F \right| > \delta \|F\|_{\operatorname{Lip}}. \]
Let $q$ be the common difference of $P$ (so that $q \leq \delta^{-1}$), and write $P = \{qm+a \colon 1 \leq m \leq |P|\}$ for some $|a| \leq q$. Consider the polynomial sequence $g': \Z \rightarrow G$ defined by
\[ g'(m) = g(qm+a). \]
It follows that $(g'(m))_{1 \leq m \leq |P|}$ is not $\delta$-equidistributed in $G/\Gamma$, and hence, by~\cite[Theorem 2.9]{green-tao-nilmanifolds}, there exists a nontrivial horizontal character $\eta : G\rightarrow \R/\Z$ with $|\eta| \ll \delta^{-O_{s,\Delta}(1)}$ such that
\[ \|\eta \circ g'\|_{C^{\infty}(|P|)} \ll \delta^{-O_{s,\Delta}(1)}. \]
If we write
\[ \eta \circ g(n) = \sum_{i=0}^s \alpha_i n^i, \]
then
\[ \eta \circ g'(m) = \sum_{i=0}^s \alpha_i (qm+a)^i = \sum_{j=0}^s \left(\sum_{i\geq j} \alpha_i \binom{i}{j} a^{i-j}\right) q^jm^j \]
In view of the bound for the smoothness norm of $\eta\circ g'$ and Lemma~\ref{lem:smoothness-norm}, it follows that
\[ \left\| q^j \sum_{i \geq j} \alpha_i \binom{i}{j} a^{i-j} \right\|_{\R/\Z} \ll \delta^{-O_{s,\Delta}(1)} |P|^{-j} \ll \delta^{-O_{s,\Delta}(1)} N^{-j} \]
for each $1 \leq j \leq s$. A simple induction argument then shows that
\[ \| (aq)^ss! \alpha_j \|_{\R/\Z} \ll \delta^{-O_{s,\Delta}(1)} N^{-j} \]
for each $1 \leq j \leq s$.
It follows that (again using Lemma~\ref{lem:smoothness-norm})
\[ \|(aq)^ss!\eta \circ g\|_{C^{\infty}(N)} \ll \delta^{-O_{s,\Delta}(1)}. \]
This completes the proof since $\|(aq)^ss!\eta\| = (aq)^ss! \|\eta\| \ll \delta^{-O_{s,\Delta}(1)}$.
\end{proof}

We have the following converse to the quantitative Leibman theorem (see~\cite[Lemma 5.3]{frantzikinakis-host}).

\begin{lemma}\label{lem:converse-leibman}
Let $s, \Delta \geq 1$. Let $G/\Gamma$ be a nilmanifold of dimension $\Delta$, equipped with a filtration $G_{\bullet}$ of degree $s$. Let $g : \Z \rightarrow G$ be a polynomial sequence adapted to $G_{\bullet}$. Suppose that $N \geq CD$ for some sufficiently large constant $C = C(s,\Delta)$. If $\|\eta\circ g\|_{C^{\infty}(N)} \leq D$ for some nontrivial horizontal character $\eta : G \rightarrow \R/\Z$ with $\|\eta\| \leq D$, then $(g(n))_{1 \leq n \leq N}$ is not totally $cD^{-2}$-equidistributed in $G/\Gamma$, where $c = c(s,\Delta)> 0$ is a positive constant.
\end{lemma}

Finally, we will need the following factorization theorem for polynomial sequences~\cite[Theorem 1.19]{green-tao-nilmanifolds}, which decomposes an arbitrary polynomial sequence into a smooth part, an equidistributed part, and a rational part.

\begin{theorem}[Factorization theorem]\label{thm:factor}
Let $s$ be a positive integer and let $\Delta, M_0 \geq 2$.  Let $G/\Gamma$ be a nilmanifold of dimension at most $\Delta$, equipped with a filtration $G_{\bullet}$ of degree at most $s$ and an $M_0$-rational Mal'cev basis $\mathcal{X}$. Let $g : \Z \rightarrow G$ be  a polynomial sequence adapted to $G_{\bullet}$.

For any $A >0$, we may find $M \in [M_0, M_0^{O_{A,s,\Delta}(1)}]$,
a sub-nilmanifold $G'/\Gamma'$ of $G/\Gamma$,a Mal'cev basis $\mathcal{X}'$ for $G'/\Gamma'$ in which each element is an $M$-rational combination of the elements of $\mathcal{X}$, and a decomposition $g = \epsilon g'\gamma$ into polynomial sequences $\epsilon, g',\gamma$ with the following properties:
\begin{enumerate}
\item $\epsilon : \Z \rightarrow G$ is $(M,N)$-smooth in the sense that $d_{\mathcal{X}}(\epsilon(n), \operatorname{id}) \leq M$ and $d_{\mathcal{X}}(\epsilon(n), \epsilon(n-1)) \leq M/N$ for each $1 \leq n \leq N$;
\item $g' \colon \Z \rightarrow G'$ takes values in $G'$, and $\{g'(n)\Gamma'\}_{1 \leq n \leq N}$ is totally $M^{-A}$-equidistributed in $G'/\Gamma'$;
\item $\gamma : \Z \rightarrow G$ is $M$-rational in the sense that for each $n \in\Z$, $\gamma(n)^r \in \Gamma$ for some $1 \leq r \leq M$. Moreover, $\gamma$ is periodic with period at most $M$.
\end{enumerate}

\end{theorem}

\section{Bombieri--Vinogradov theorems for equidistributed nilsequences}\label{sec:equidistribution}

In this section we reduce our problem to studying those nilsequences $\psi$ that are ``equidistributed". In the case when $\psi(n) = e(\alpha n^s)$, this corresponds to $\alpha$ lying in ``minor arcs". 

\begin{definition}[Equidistributed nilsequences]\label{def:equi}
Recall the definition of $\Psi_s(\Delta,K)$ in Definition~\ref{def:nil}. For $\eta \in (0,1)$ and $x \geq 2$, we define $\Psi_s(\Delta, K; \eta, x)$ to be the class of those nilsequences $\psi \in \Psi_s(\Delta, K)$ of the form $\psi(n) = \varphi(g(n)\Gamma)$ that obey the following additional conditions:
\begin{enumerate}
\item the finite sequence $(g(n)\Gamma)_{1 \leq n \leq 10x}$ is totally $\eta$-equidistributed in $G/\Gamma$ (defined in~\cite[Definition 1.2]{green-tao-nilmanifolds});
\item the Lipschitz function $\varphi$ satisfies $\int_{G/\Gamma} \varphi = 0$ (Here the integral is with respect to the unique Haar measure on $G/\Gamma$).
\end{enumerate}
We will loosely call those nilsequences in $\Psi_s(\Delta, K; \eta, x)$ $\eta$-\emph{equidistributed}.
\end{definition}

\subsection{Statements of results for equidistributed nilsequences}

The following theorems show that our main theorems hold for $\eta$-equidistributed nilsequences when $0<\eta\leq (\log x)^{-O_A(1)}$. Moreover, the error term for such nilsequences is even power-saving if $\eta=x^{-\delta}$ for some constant $\delta>0$.

\begin{theorem}\label{equidist-1/4}
Let an integer $s\geq 1$, a large real number $\Delta \geq 2$, and a small real number $\varepsilon\in (0,1/4)$ be given. There exists a constant $\kappa = \kappa(s,\Delta,\ee) > 0$, such that for $x \geq 2$, $\eta>0$ we have
\[ \sum_{d \leq x^{1/4-\ee}} \max_{(c,d)=1} \sup_{\psi \in \Psi_s(\Delta, \eta^{-\kappa};\eta,x/d)} \Big| \sum_{\substack{n\leq x \\ n \equiv c\Mod{d}}} \Lambda(n) \psi((n-c)/d) \Big| \ll \eta^{\kappa} x (\log x)^2. \]
\end{theorem}

\begin{theorem}\label{equidist-1/3}
Let an integer $s\geq 1$, a large real number $\Delta\geq 2$, and a small real number $\varepsilon\in (0,1/3)$ be given. There exists a constant $\kappa = \kappa(s,\Delta,\ee) > 0$, such that for any $x \geq 2$, $\eta>0$ and any nilsequence $\psi \in \Psi_s(\Delta, \eta^{-\kappa};\eta,x)$  we have
\[
\sum_{d\leq x^{1/3-\varepsilon}}\max_{(c,d)=1} \Big|\sum_{\substack{n \leq x \\ n \equiv c\Mod{d}}} \Lambda(n) \psi(n)\Big|\ll\eta^{\kappa} x (\log x)^2.
\]
\end{theorem}

\begin{theorem}\label{equidist-1/2}
Let integers $s\geq 1$, $c \neq 0$, a large real number $\Delta\geq 2$, and a small real number $\varepsilon\in (0,1/2)$ be given. There exists a constant $\kappa = \kappa(s,\Delta,\varepsilon) > 0$, such that for any well-factorable sequence $(\lambda_d)$ of level $x^{1/2-\varepsilon}$ with $x \geq 2$ and any nilsequence $\psi \in \Psi_s(\Delta, \eta^{-\kappa};\eta,x)$ with $\eta > 0$, we have 
\[
\Big|\sum_{\substack{d\leq x^{1/2-\varepsilon} \\ (d,c)=1}} \lambda_d \sum_{\substack{n \leq x \\ n \equiv c\Mod{d}}} \Lambda(n) \psi(n) \Big| \ll \eta^{\kappa} x (\log x)^2.
\]
\end{theorem}

\begin{remark}\label{rem:mobius2}
Theorems \ref{equidist-1/4}, \ref{equidist-1/3}, \ref{equidist-1/2} hold equally well with $\Lambda$ replaced by $\mu$, since the proofs rely on type I and II estimates and $\mu$ satisfies a similar Vaughan identity as $\Lambda$ does. 
\end{remark}

As mentioned earlier, our proof methods also apply to other arithmetic functions than $\Lambda$ and $\mu$, in particular to the functions $d_k$ and $1_S$ (and other sequences that satisfy an identity of Heath-Brown's type). 

\begin{theorem}\label{thm_divisor}
Theorems \ref{equidist-1/4}, \ref{equidist-1/3} and \ref{equidist-1/2} continue to hold with $\Lambda$ replaced by the $k$-fold divisor function $d_k$ with $k\in \mathbb{C}$ fixed, or $\Lambda$ replaced by $1_{S}$, where $S$ is the set of natural numbers expressible as the sum of two squares. 
\end{theorem}

In our deductions of Theorems~\ref{thm:1/4},~\ref{thm:1/3}, and~\ref{thm:1/2} from Propositions~\ref{equidist-1/4},~\ref{equidist-1/3}, and~\ref{equidist-1/2}, we also need the analogues of Propositions~\ref{equidist-1/4},~\ref{equidist-1/3}, and~\ref{equidist-1/2} with $\Lambda$ replaced by the function $n \mapsto 1_{(n,W)=1}$. Proposition~\ref{prop:sieve1} below will be used in the deduction of Theorem~\ref{thm:1/4}, and Proposition~\ref{prop:sieve2} below will be used in the deductions of Theorems~\ref{thm:1/3} and~\ref{thm:1/2}.

\begin{proposition}\label{prop:sieve1}
Let an integer $s\geq 1$, large real numbers $A, \Delta \geq 2$, and a small real number $\varepsilon\in (0,1/2)$ be given. There exists a constant $\kappa = \kappa(s,\Delta,\ee) > 0$, such that for $x \geq 2$, $\eta>0$ we have
\[ \sum_{d \leq x^{1/2-\ee}} \max_{(c,d)=1} \sup_{\psi \in \Psi_s(\Delta, \eta^{-\kappa};\eta,x/d)} \Big| \sum_{\substack{n\leq x \\ n \equiv c\Mod{d}\\ (n,W)=1}} \psi((n-c)/d) \Big| \ll \eta^{\kappa} x \log x + O_A(x(\log x)^{-A}) ,
\]
where $W = \mathscr{P}(w)$ for some $w \leq \exp(\sqrt{\log x})$.
\end{proposition}

\begin{proposition}\label{prop:sieve2}
Let an integer $s\geq 1$, large real numbers $A, \Delta\geq 2$, and a small real number $\varepsilon\in (0,1/2)$ be given. There exists a constant $\kappa = \kappa(s,\Delta,\ee) > 0$, such that for any $x \geq 2$, $\eta>0$ and any nilsequence $\psi \in \Psi_s(\Delta, \eta^{-\kappa};\eta,x)$  we have
\[
\sum_{d\leq x^{1/2-\varepsilon}}\max_{(c,d)=1} \Big|\sum_{\substack{n \leq x \\ n \equiv c\Mod{d}\\ (n,W)=1}} \psi(n)\Big|\ll\eta^{\kappa} x \log x + O_A(x(\log x)^{-A}),
\]
where $W = \mathscr{P}(w)$ for some $w \leq \exp(\sqrt{\log x})$.
\end{proposition}

The proofs of Theorems~\ref{equidist-1/4}--\ref{thm_divisor} and of Propositions~\ref{prop:sieve1} and~\ref{prop:sieve2} will be given in Section~\ref{sec:combine}, using the type I and type II estimates established in Sections~\ref{sec:typeI} and~\ref{sec:typeII}.

\subsection{Proof that Theorem \ref{equidist-1/4} implies Theorem \ref{thm:1/4}}\label{sec:deduction1}

We now show that these equidistributed cases for $\Lambda$ imply the general cases of our main theorems.  The proof of the analogous results for the M\"{o}bius function as mentioned in Remark~\ref{rem:mobius}  is completely similar, using Remark \ref{rem:mobius2} (and in fact easier due to not needing the $W$-trick). 

We start with the deduction of Theorem~\ref{thm:1/4} from Theorem~\ref{equidist-1/4}.
We may assume that $x$ is sufficiently large in terms of $A,s,\Delta,\ee$.
Let $1\leq Q\leq x^{1/4-\varepsilon}$. For $q \in [Q,2Q]$, let
\begin{align*}
f_q(n):=\Lambda(n)-\frac{qW}{\varphi(qW)}1_{(n,W)=1}.   
\end{align*}
Let $\mathcal{B}$ be the set of $q\in [Q,2Q]$ such that
\begin{align*}
\max_{(a,q)=1} \sup_{\psi\in \Psi_{s}(\Delta,\log x)} \Big|\sum_{\substack{n\leq x\\n\equiv a \Mod{q}}}f_q(n)\psi(n)\Big| \geq \frac{x}{q (\log x)^{2A}}. 
\end{align*}
We may assume that $|\mathcal{B}|\geq Q(\log x)^{-2A}$ (for some choice of $Q$), since otherwise the claim of Theorem \ref{thm:1/4} follows from the triangle inequality. For each $q \in \mathcal{B}$, we may pick some $(a,q)=1$ and some $\psi \in \Psi_s(\Delta, \log x)$ such that
\begin{align}\label{eq41}
\Big|\sum_{\substack{n\leq x\\n\equiv a \Mod q}}f_q(n)\psi((n-a)/q)\Big|\geq \frac{x}{q (\log x)^{2A}}   ,  
\end{align}
where we used the simple observation that if $\psi \in \Psi_s(\Delta,\log x)$ then the dilation $\psi'(n)=\psi(qn+a)$ is also in $\Psi_s(\Delta,\log x)$.

\begin{lemma}\label{lem:factor}
Let $\mathcal{B}$ and $f_q$ be defined as above.
For $q \in \mathcal{B}$ and any $B \geq 2$, we may find 
\begin{itemize}
	\item $M \in [(\log x)^{3A}, (\log x)^{O_{A,B,s,\Delta}(1)}]$;
    \item $y \in [x (\log x)^{-O_{A,B,s,\Delta}(1)}, x]$;
    \item $q' = tq$ for some positive integer $t \leq M$;
        \item a residue class $a'\Mod{q'}$ with $0 \leq a' < q'$ and $a'\equiv a\Mod{q}$,
\end{itemize}
such that at least one of the following two conclusions hold:
\begin{enumerate}
\item For some $\psi' \in \Psi_s(\Delta, M^{O_{s,\Delta}(1)}; M^{-B}, y/q')$ we have
\[  \Big|\sum_{\substack{n\leq y\\n\equiv a'\Mod {q'}}}f_q(n)\psi'((n-a')/q')\Big|\gg \frac{x}{q'} (\log x)^{-O_{A,B,s,\Delta}(1)}. \]
\item We have
\[ \Big|\sum_{\substack{n \leq y \\ n \equiv a'\Mod {q'}}} f_q(n) \Big| \gg \frac{x}{q'} (\log x)^{-O_{A,B,s,\Delta}(1)}. \]
\end{enumerate}
\end{lemma}

\begin{proof}
This will be deduced from the factorization theorem for polynomial sequences (Theorem~\ref{thm:factor}), in a similar way as in~\cite[Lemma 2.4]{shao}. Roughly speaking, the nilsequence $n \mapsto \psi((n-a)/q)$ becomes equidistributed after passing to a sub-progression of the form $\{n \in I \colon n \equiv a'\Mod{q'}\}$. On this sub-progression, we can write $\psi = \psi' + c_0$ for some equidistributed $\psi'$ as in conclusion (1), where $c_0$ is the average of $\psi$ on the sub-progression. 

We now turn to the details.
Write $\psi(n) = \varphi(g(n)\Gamma)$ for some Lipschitz function $\varphi$ on $G/\Gamma$ and some polynomial sequence $g$. Let $D$ be a constant sufficiently large in terms of $B, s,\Delta$. By Theorem~\ref{thm:factor} applied with $M_0 = (\log x)^{3A}$, we may find $M \in [(\log x)^{3A}, (\log x)^{O_{A,D,s,\Delta}(1)}]$, a sub-nilmanifold $G'/\Gamma'$ of $G/\Gamma$, a Mal'cev basis $\mathcal{X}'$ for $G'/\Gamma'$ in which each element is an $M$-rational combination of the elements of $\mathcal{X}$, and a decomposition $g = \epsilon g'\gamma$ into polynomial sequences $\epsilon, g', \gamma$ with the following properties:
\begin{enumerate}
\item $\epsilon$ is $(M, x/q)$-smooth;
\item $g'$ takes values in $G'$, and $\{g'(n)\}_{1 \leq n \leq x/q}$ is totally $M^{-D}$-equidistributed in $G'/\Gamma'$;
\item $\gamma$ is $M$-rational and is periodic with period $t \leq M$.
\end{enumerate}
After a change of variables $n = qm+a$, we may write~\eqref{eq41} as
\[ \Big| \sum_{m \leq x/q} f_q(qm+a) \varphi(g(m)\Gamma) \Big| \geq \frac{ x}{q} (\log x)^{-2A}.   \]
Dividing $[0, x/q]$ into $O(M^2)$ intervals of equal length and then further dividing them into residue classes modulo $t$, we may find an interval $I \subset [0, x]$ with $|I| \asymp x/M^2$ and some residue class $b\Mod t$, such that
\begin{equation}\label{eq:factor1} 
\Big| \sum_{\substack{ m \equiv b\Mod t\\ qm+a \in I}} f_q(qm+a) \varphi(\epsilon(m) g'(m) \gamma(m) \Gamma) \Big| \gg \frac{|I|}{qt}(\log x)^{-2A}.
\end{equation}
Pick any $m_0$ counted in the sum (i.e. $m_0 \equiv b\Mod t$ and $qm_0+a \in I$). Note that 
\begin{align*}  
& \Big| \sum_{\substack{ m \equiv b\Mod t\\ qm+a \in I}} f_q(qm+a) \left(\varphi(\epsilon(m) g'(m) \gamma(m) \Gamma) - \varphi(\epsilon(m_0) g'(m) \gamma(m) \Gamma)\right) \Big|  \\
\ll & (\log x) \sum_{\substack{ m \equiv b\Mod t\\ qm+a \in I}} \left|\varphi(\epsilon(m) g'(m) \gamma(m) \Gamma) - \varphi(\epsilon(m_0) g'(m) \gamma(m) \Gamma)\right| \\
\ll & (\log x) \sum_{\substack{ m \equiv b\Mod t\\ qm+a \in I}} d_{\mathcal{X}} (\epsilon(m)g'(m)\gamma(m), \epsilon(m_0)g'(m)\gamma(m)) \\
= & (\log x) \sum_{\substack{ m \equiv b\Mod t\\ qm+a \in I}} d_{\mathcal{X}}(\epsilon(m), \epsilon(m_0)),
\end{align*}
where we used the right invariance of $d_{\mathcal{X}}$. By the smoothness property of of $\epsilon$, each term $d_{\mathcal{X}}(\epsilon(m), \epsilon(m_0))$ is $\ll M^{-1} \ll (\log x)^{-3A}$. Hence the expression above is $\ll \tfrac{|I|}{qt} (\log x)^{-3A+1}$,
which is negligible compared with~\eqref{eq:factor1}, so we can replace $\epsilon(m)$ in~\eqref{eq:factor1} by $\epsilon(m_0)$. Moreover, we may clearly replace $\gamma(m)$ in~\eqref{eq:factor1} by $\gamma(m_0)$, since $\gamma$ has period $t$. Hence, after a change of variables $m = t\ell + b$,~\eqref{eq:factor1} becomes
\[ \Big| \sum_{\ell\colon qt\ell+qb+a \in I} f_q(qt\ell+qb+a) \varphi(\epsilon(m_0) g'(t\ell+b) \gamma(m_0)\Gamma) \Big| \gg \frac{|I|}{qt} (\log x)^{-2A}. \]
By writing the interval $I$ as a difference of two intervals $[1,y_2]\setminus [1,y_1]$, we have for either $y=y_1$ or $y=y_2$,
\[ \Big| \sum_{\ell\colon qt\ell+qb+a \leq y} f_q(qt\ell+qb+a) \varphi(\epsilon(m_0) g'(t\ell+b) \gamma(m_0)\Gamma) \Big|  \gg \frac{x}{qt}(\log x)^{-O_{A,B,s,\Delta}(1)}. \]
Note that for this to hold, we necessarily have $y \geq x(\log x)^{-O_{A,B,s,\Delta}(1)}$.
Now let $\wt{g}$ be the polynomial sequence defined by
\[ \wt{g}(\ell) = \gamma(m_0)^{-1} g'(t\ell +b) \gamma(m_0), \]
taking values in $\wt{G} = \gamma(m_0)^{-1}G'\gamma(m_0)$, and let $\varphi'$ be the Lipschitz function on $\wt{G}/\wt{\Gamma}$ (where $\wt{\Gamma} = \wt{G} \cap \Gamma$) defined by
\[ \varphi'(x) = \varphi(\epsilon(m_0) \gamma(m_0) x). \]
Then we have
\[ \Big| \sum_{\ell\colon qt\ell +qb+a \leq y} f_q(qt\ell + qb+a) \varphi'(\wt{g}(\ell)\wt{\Gamma}) \Big| \gg \frac{x}{qt} (\log x)^{-O_{A,B,s,\Delta}(1)}. \]
Setting $q' = qt$ and $a' = qb+a$, the inequality above can be rewritten as
\[ \Big| \sum_{\substack{n \leq y \\ n \equiv a'\Mod{q'}}} f_q(n) \varphi'(\wt{g}((n-a')/q')\wt{\Gamma}) \Big| \gg \frac{x}{q'} (\log x)^{-O_{A,B,s,\Delta}(1)}. \]
Let $c_0 = \int\varphi'$ and $\wt{\varphi} = \varphi' - c_0$, so that $\int\wt{\varphi} = 0$. Then either conclusion (2) holds, in which case we are done, or else we have
\[ \Big| \sum_{\substack{n \leq y \\ n \equiv a'\Mod{q'}}} f_q(n) \wt{\varphi}(\wt{g}((n-a')/q')\wt{\Gamma})  \Big| \gg \frac{x}{q'} (\log x)^{-O_{A,B,s,\Delta}(1)}. \]
Using some standard ``quantitative nil-linear algebra" arguments (see the claim at the end of~\cite[Section 2]{green-tao-mobius}), one can verify the following properties:
\begin{itemize}
\item $\wt{G}/\wt{\Gamma}$ is a nilmanifold of dimension at most $\Delta$, equipped with a filtration $\wt{\mathcal{X}}$ of degree at most $s$ and a $M^{O_{s,\Delta}(1)}$-rational Mal'cev basis;
\item $\|\wt{\varphi}\|_{\operatorname{Lip}(\wt{\mathcal{X}})} \leq M^{O_{s,\Delta}(1)} \leq (\log x)^{O_{A,D,s,\Delta}(1)}$;
\item $\wt{g}$ is a polynomial sequence on $\wt{G}$ adapted to $\wt{\mathcal{X}}$, and $\{\wt{g}(\ell)\}_{1 \leq \ell \leq y/q'}$ is totally $M^{-cD+O(1)}$-equidistributed in $\wt{G}/\wt{\Gamma}$ for some small constant $c = c(s,\Delta)>0$.
\end{itemize}
By taking $D$ large enough in terms of $B,s,\Delta$, we may ensure that $M^{-cD+O(1)} \leq M^{-B}$. Hence conclusion (1) follows by taking $\psi'(\ell) = \wt{\varphi}(\wt{g}(\ell)\wt{\Gamma}) / \|\wt{\varphi}\|_{\operatorname{Lip}(\wt{\mathcal{X}})}$.
\end{proof}

We now apply Lemma~\ref{lem:factor}, with $B$ a constant sufficiently large in terms of $A,s,\Delta,\ee$. 
Note that the quantities $M, y, t$ produced in Lemma~\ref{lem:factor} all depend on $q$. However, each of them can be chosen from at most $(\log x)^{O_{A,s,\Delta,\ee}}$ possibilities: For $t$ this is clear since $t \leq M$; For $y$, we may require in the proof that it is of the form $y = \lfloor kx/M^2\rfloor$ for some positive integer $k \leq M^2$; For $M$, we may require that it is of the form $M = (\log x)^k$ for a positive integer $k = O_{A,s,\Delta,\ee}(1)$. Hence, after restricting to a subset $\mathcal{B}'$ of $\mathcal{B}$ of size $|\mathcal{B}'| \gg Q(\log x)^{-O_{A,s,\Delta,\ee}(1)}$, we may find
\begin{itemize}
	\item $M \in [(\log x)^{3A}, (\log x)^{O_{A,s,\Delta,\ee}(1)}]$,
    \item $y \in [x(\log x)^{-O_{A,s,\Delta,\ee}(1)}, x]$;
    \item $q' = \ell q$ for some positive integer $\ell \leq M$,
  \end{itemize}
such that at least one of the following two conclusions hold:
\begin{itemize}
\item For each $q \in \mathcal{B}'$ we have
\begin{equation}\label{eq42} 
\max_{\substack{a'\Mod{q'} \\ (a',q)=1}} \sup_{\psi' \in \Psi_s(\Delta, M^{O_{s,\Delta}(1)}; M^{-B}; y/q')} \Big|\sum_{\substack{n\leq y\\n\equiv a'\Mod {q'}}}f_q(n)\psi'((n-a')/q')\Big|\gg \frac{x}{q'} (\log x)^{-O_{A,s,\Delta,\ee}(1)}. 
\end{equation}
\item For each $q \in \mathcal{B}'$ we have
\begin{equation}\label{eq43}
\max_{\substack{a'\Mod{q'} \\ (a',q)=1}}  \Big|\sum_{\substack{n\leq y\\n\equiv a'\Mod {q'}}}f_q(n)\Big|\gg \frac{x}{q'} (\log x)^{-O_{A,s,\Delta,\ee}(1)}. 
\end{equation}
\end{itemize}

{\textbf{First assume that~\eqref{eq43} holds.}} We will show that this is a contradiction to the classical Bombieri--Vinogradov inequality.
By the definition of $f_q(n)$, we have
\begin{equation}\label{eq:sumfq}
\sum_{\substack{n \leq y \\ n \equiv a'\Mod{q'}}} f_q(n) = \sum_{\substack{n \leq y \\ n \equiv a'\Mod{q'}}} \Lambda(n) - \frac{qW}{\varphi(qW)}\sum_{\substack{n \leq y \\ (n,W) = 1 \\ n \equiv a'\Mod{q'}}} 1.
\end{equation}

\subsubsection*{Case $(a',q')>1$:}
First consider the case when $(a',q') > 1$. Since $a' \equiv a\Mod{q}$ and $(a,q)=1$, we must have $(a',q') = (a',\ell) \leq \ell \leq (\log x)^{O_{A,s,\Delta,\ee}(1)}$. Hence we can ensure that $(a',q') \mid W^{\infty}$ by choosing $C$ in the definition of $W$ large enough in terms of $A,s,\Delta,\ee$, and thus the second sum on the right hand side of~\eqref{eq:sumfq} is empty, so~\eqref{eq:sumfq} is $O(\log x)$ in this case, contradicting~\eqref{eq43}.

\subsubsection*{Case $(a',q')=1$:} Now consider the case when $(a',q')=1$. Let $D$ be a constant sufficiently large in terms of $A,s,\Delta,\ee$. By the classical Bombieri--Vinogradov theorem, we have
\[ \sum_{\substack{n \leq y \\ n \equiv a'\Mod{q'}}} \Lambda(n) = \frac{y}{\varphi(q')} \left(1 + O( (\log x)^{-D}) \right) \]
for almost all $q \in \mathcal{B}'$. Let $\mathcal{B}''$ be the set of $q \in \mathcal{B}'$ satisfying the estimate above, so that $|\mathcal{B}''| \gg Q(\log x)^{-O_{A,s,\Delta,\ee}(1)}$. By the fundamental lemma of sieve theory~\cite[Lemma 6.3]{iwaniec-kowalski}, the second term on the right hand side of~\eqref{eq:sumfq} is
\[ \frac{qW}{\varphi(qW)}\sum_{\substack{n \leq y \\ (n,W) = 1 \\ n \equiv a'\Mod{q'}}} 1  =  \left(1 + O( (\log x)^{-D}) \right) \frac{qW}{\varphi(qW)} \cdot \frac{y}{q'} \prod_{p \mid W, p\nmid q'} \left(1-p^{-1}\right). \]
Since $qW/\varphi(qW) = q'W/\varphi(q'W)$ by the fact that $\ell\mid W^{\infty}$, the main term above can be computed to be
\[  \frac{qW}{\varphi(qW)} \cdot \frac{y}{q'} \prod_{p \mid W, p\nmid q'} \left(1-p^{-1}\right) =  \frac{q'W}{\varphi(q'W)} \cdot \frac{y}{q'} \cdot\frac{\varphi(q'W)}{W\varphi(q')} = \frac{y}{\varphi(q')}. \]
Combining the above, we see that~\eqref{eq:sumfq} is $\ll \tfrac{y}{\varphi(q')} (\log x)^{-D}$, contradicting~\eqref{eq43} once $D$ is large enough.

\vspace{0.5cm}

{\textbf{Now assume that~\eqref{eq42} holds.}} We will show that this is a contradiction to either Theorem~\ref{equidist-1/4} or Proposition~\ref{prop:sieve1}. By the definition of $f_q(n)$, either
\[
\max_{\substack{a'\Mod{q'} \\ (a',q)=1}} \sup_{\psi'\in \Psi_s(\Delta, M^{O_{s,\Delta}(1)};M^{-B},y/q')}\Big|\sum_{\substack{n\leq y\\n\equiv a'\Mod {q'}}}\Lambda(n)\psi'((n-a')/q')\Big|\gg \frac{x}{q'} (\log x)^{-O_{A,s,\Delta,\ee}(1)},
\]
or
\[
\max_{\substack{a'\Mod{q'} \\ (a',q)=1}} \sup_{\psi'\in \Psi_s(\Delta, M^{O_{s,\Delta}(1)};M^{-B},y/q')}\Big|\sum_{\substack{n\leq y\\n\equiv a'\Mod {q'} \\ (n,W)=1}} \psi'((n-a')/q')\Big|\gg \frac{x}{q'} (\log x)^{-O_{A,s,\Delta,\ee}(1)} 
\]
holds for $q \in \mathcal{B}'$. In both cases we must have $(a',q')=1$, since otherwise the sum over $y$ is $\ll \log x$. 
Since $y \geq x (\log x)^{-O_{A,s,\Delta,\ee}(1)}$, we have
\[ q' = \ell q \leq x^{1/4-\ee} (\log x)^{O_{A,s,\Delta,\ee}(1)} \leq y^{1/4-\ee/2}. \]
It follows that either
\begin{equation}\label{eq25a} 
\sum_{q' \leq y^{1/4-\ee/2}} \max_{(a',q')=1} \sup_{\psi'\in \Psi_s(\Delta, M^{O_{s,\Delta}(1)};M^{-B},y/q')}\Big|\sum_{\substack{n\leq y\\n\equiv a'\Mod {q'}}}\Lambda(n)\psi'\left(\frac{n-a'}{q'}\right)\Big| \gg x (\log x)^{-O_{A,s,\Delta,\ee}(1)}, \end{equation}
or
\begin{equation}\label{eq25b}
\sum_{q' \leq y^{1/4-\ee/2}} \max_{(a',q')=1} \sup_{\psi'\in \Psi_s(\Delta, M^{O_{s,\Delta}(1)};M^{-B},y/q')}\Big|\sum_{\substack{n\leq y\\n\equiv a'\Mod {q'}\\ (n,W)=1}}\psi'\left(\frac{n-a'}{q'}\right)\Big| \gg x (\log x)^{-O_{A,s,\Delta,\ee}(1)}. \end{equation}
On the other hand, Theorem~\ref{equidist-1/4} implies that the left-hand side of~\eqref{eq25a} is $\ll M^{-\kappa B} x(\log x)^2 \ll x (\log x)^{-\kappa B+2}$ for some $\kappa = \kappa(s,\Delta,\ee) > 0$. This is a contradiction to~\eqref{eq25a} once $B$ is large enough. Similarly, Proposition~\ref{prop:sieve1} implies that the left-hand side of~\eqref{eq25b} is $\ll M^{-\kappa B}x (\log x)^2 + O_D(x(\log x)^{-D})$ for any $D \geq 2$. This is a contradiction to~\eqref{eq25b} once $B$ and $D$ are large enough.

\subsection{Proof that Theorems~\ref{equidist-1/3} and~\ref{equidist-1/2} imply Theorems~\ref{thm:1/3} and~\ref{thm:1/2}}\label{sec:deduction2}

In the deductions of Theorems~\ref{thm:1/3} and~\ref{thm:1/2}, we have a fixed nilsequence $\psi$ and we will apply the factorization theorem to reduce to the case when it is equidistributed at scale $x$. In comparison, we had nilsequences $\psi_q$ varying with $q$ in Section~\ref{sec:deduction1} and we applied the factorization theorem to reduce to the case when each $\psi_q$ is equidistributed at scale $x/q$. Apart from this difference, the rest of the arguments are very much the same, so we will be brief with the arguments in this section.

We focus on the deduction of Theorem~\ref{thm:1/3} from Theorem~\ref{equidist-1/3}, since the deduction of Theorem~\ref{thm:1/2} from Theorem~\ref{equidist-1/2} is analogous.
We may assume that $x$ is sufficiently large in terms of $A,s,\Delta,\ee$.
Let $1\leq Q\leq x^{1/3-\varepsilon}$. For $q \in [Q,2Q]$, let $f_q$ be defined as before.
Let $\mathcal{B}$ be the set of $q\in [Q,2Q]$ such that
\begin{align*}
\max_{(a,q)=1} \Big|\sum_{\substack{n\leq x\\n\equiv a \Mod{q}}}f_q(n)\psi(n)\Big| \geq \frac{x}{q (\log x)^{2A}}. 
\end{align*}
We may assume that $|\mathcal{B}|\geq Q(\log x)^{-2A}$ (for some choice of $Q$), since otherwise the claim of Theorem~\ref{thm:1/3} follows from the triangle inequality.

\begin{lemma}\label{lem:factor2}
Let $\mathcal{B}$ and $f_q$ be defined as above.
For $q \in \mathcal{B}$ and any $B \geq 2$, we may find 
\begin{itemize}
	\item $M \in [(\log x)^{3A}, (\log x)^{O_{A,B,s,\Delta}(1)}]$;
    \item $y \in [x (\log x)^{-O_{A,B,s,\Delta}(1)}, x]$;
    \item $q' = tq$ for some positive integer $t \leq M$;
        \item a residue class $a'\Mod{q'}$ with $0 \leq a' < q'$ and $a'\equiv a\Mod{q}$,
\end{itemize}
such that at least one of the following two conclusions hold:
\begin{enumerate}
\item For some $\psi' \in \Psi_s(\Delta, M^{O_{s,\Delta}(1)}; M^{-B}, y)$ independent of $q$ we have
\[  \Big|\sum_{\substack{n\leq y\\n\equiv a'\Mod {q'}}}f_q(n)\psi'(n)\Big|\gg \frac{x}{q'} (\log x)^{-O_{A,B,s,\Delta}(1)}. \]
\item We have
\[ \Big|\sum_{\substack{n \leq y \\ n \equiv a'\Mod {q'}}} f_q(n) \Big| \gg \frac{x}{q'} (\log x)^{-O_{A,B,s,\Delta}(1)}. \]
\end{enumerate}
\end{lemma}

We omit the proof of Lemma~\ref{lem:factor2}, since its deduction from the  factorization theorem for polynomial sequences (Theorem~\ref{thm:factor}) is completely analogous to that of Lemma~\ref{lem:factor}.

We now apply Lemma~\ref{lem:factor2}, with $B$ a constant sufficiently large in terms of $A,s,\Delta,\ee$. 
After restricting to a subset $\mathcal{B}'$ of $\mathcal{B}$ of size $|\mathcal{B}'| \gg Q(\log x)^{-O_{A,s,\Delta,\ee}(1)}$, we may find
\begin{itemize}
	\item $M \in [(\log x)^{3A}, (\log x)^{O_{A,s,\Delta,\ee}(1)}]$,
    \item $y \in [x(\log x)^{-O_{A,s,\Delta,\ee}(1)}, x]$;
    \item $q' = \ell q$ for some positive integer $\ell \leq M$,
  \end{itemize}
(with all of $M,y,\ell$ independent of $q$), such that at least one of the following two conclusions hold:
\begin{itemize}
\item For each $q \in \mathcal{B}'$ we have
\begin{equation}\label{eq42'} 
\max_{\substack{a'\Mod{q'} \\ (a',q)=1}}  \Big|\sum_{\substack{n\leq y\\n\equiv a'\Mod {q'}}}f_q(n)\psi'(n)\Big|\gg \frac{x}{q'} (\log x)^{-O_{A,s,\Delta,\ee}(1)}. 
\end{equation}
\item For each $q \in \mathcal{B}'$ we have
\begin{equation}\label{eq43'}
\max_{\substack{a'\Mod{q'} \\ (a',q)=1}}  \Big|\sum_{\substack{n\leq y\\n\equiv a'\Mod {q'}}}f_q(n)\Big|\gg \frac{x}{q'} (\log x)^{-O_{A,s,\Delta,\ee}(1)}. 
\end{equation}
\end{itemize}

\eqref{eq43'} is analyzed in the same way as~\eqref{eq43}; it leads to a contradiction to the classical Bombieri--Vinogradov inequality.  \eqref{eq42'} is analyzed in the same way as~\eqref{eq42}; it leads to a contradiction to either Theorem~\ref{equidist-1/3} or Proposition~\ref{prop:sieve2}.

\section{Type I estimates}\label{sec:typeI}

Our remaining task for proving the main theorems is proving the results for equidistributed nilsequences from Section \ref{sec:equidistribution}. By applying Vaughan's identity, our sum of interest
\[ \sum_{\substack{n \leq x \\ n \equiv c\pmod d}} \Lambda(n) \psi((n-c)/d) \]
can be decomposed in terms of type I sums of the form
\[ \sum_{\substack{mn\leq x \\ M\leq m \leq 2M \\ mn \equiv c\Mod{d}}} a_m  \psi((mn-c)/d) \ \ \text{or} \ \ \sum_{\substack{mn\leq x \\ M\leq m \leq 2M \\ mn \equiv c\Mod{d}}} a_m (\log n) \psi((mn-c)/d) \]
for some $M \leq x^{1/3}$, and type II sums of the form
\[ \sum_{\substack{mn\leq x \\ M\leq m \leq 2M \\ mn \equiv c\Mod{d}}} a_m  b_n \psi((mn-c)/d) \]
for some $x^{1/3} \leq M \leq x^{2/3}$.
We begin by analyzing the type I sums.  In our argument we will repeatedly use the following recurrence result, which is a standard consequence of Weyl's inequality (see~\cite[Lemma 4.5]{green-tao-nilmanifolds}).

\begin{lemma}\label{lem:recur}
Let $0 < \delta < 1/2$ and $0 < \ee < \delta/4$. Let $g(n)$ be a polynomial of degree $s$. Let $N \geq \delta^{-C}$ for some sufficiently large constant $C = C(s)$, and suppose that for at least $\delta N$ values of the integers $1 \leq n \leq N$ we have $\|g(n)\|_{\R/\Z} \leq \ee$. Then there exists a positive integer $q \ll \delta^{-O_s(1)}$ such that
\[ \|q\alpha_i\|_{\R/\Z} \ll \ee \delta^{-O_s(1)} N^{-i} \]
for each $0 \leq i \leq s$, where $\alpha_i$ is the coefficient of $n^i$ in $g(n)$.
\end{lemma}

Note that the assumption $N \geq \delta^{-C}$ is necessary, even though this is not explicitly mentioned in~\cite[Lemma 4.5]{green-tao-nilmanifolds}. However, when $g(n)$ is a monomial $g(n) = \alpha_s n^s$, the conclusion of Lemma~\ref{lem:recur} holds for any $N \geq 1$. Note also that~\cite[Lemma 4.5]{green-tao-nilmanifolds} only gives the conclusion for $1 \leq i \leq s$, but the bounds for $\|q\alpha_i\|$ combined with our assumption that $\|g(n)\| \leq \ee$ for some $1 \leq n \leq N$ gives the conclusion for $i=0$.

Throughout the rest of the paper, we will use the notation
\begin{align}\label{eq49}
\|a\|_p:=\Big(\frac{1}{M}\sum_{m}|a_m|^p\Big)^{1/p},    
\end{align}
where $M\geq 1$ is the smallest element of the support of $\{a_m\}$. The following type I estimate suffices for Theorems~\ref{equidist-1/4}.

\begin{proposition}[type I estimate with supremum]\label{typeI}
Let $x \geq 2$ and $\varepsilon>0$. Let $1\leq M \leq x^{1/2}$ and $1\leq D\leq x^{1/2-\varepsilon}$.  Let $s \geq 1$, $\Delta \geq 2$, and $0 < \delta < 1/2$. Then  there exists a large constant $C = C(s,\Delta,\varepsilon)$ such that 
\[ \sum_{D\leq d\leq 2D} \max_{(c,d)=1} \sup_{\psi \in \Psi_s(\Delta, \delta^{-1}; \delta^C, x/D)} \Big| \sum_{\substack{mn\leq x \\ M\leq m \leq 2M \\ mn \equiv c\Mod{d}}} a_m  \psi((mn-c)/d) \Big| \ll \delta x \|a\|_2 \]
for any sequence $\{a_m\}$.
\end{proposition}

By applying the triangle inequality and Cauchy--Schwarz, Proposition \ref{typeI} reduces to the following slightly more general bound (replacing $\delta$ by $\delta^2$ there). Note that, in~\eqref{eq53} below, the maximum is taken over all residue classes $c\pmod{d}$, not just those coprime with $d$.

\begin{proposition}
Let $x \geq 2$ and $\varepsilon>0$. Let $1\leq M \leq x^{1/2}$ and $1\leq D\leq x^{1/2-\varepsilon}$.  Let $s \geq 1$, $\Delta \geq 2$, and $0 < \delta < 1/2$. Then  there exists a large constant $C = C(s,\Delta,\varepsilon)$ such that 
\begin{align}\label{eq53}
\sum_{D\leq d\leq 2D} \max_{c} \sup_{\psi \in \Psi_s(\Delta, \delta^{-1}; \delta^C, x/D)} \sum_{\substack{M\leq m\leq 2M\\(m,d)=1}} \Big| \sum_{\substack{n\leq x/m \\ n \equiv cm^{-1}\Mod{d}}} \psi((mn-c)/d) \Big| \ll \delta x.
\end{align}
\end{proposition}

\begin{proof}
For ease of notation we will suppress the dependence of implied constants on $s,\Delta,\ee$.
 We may assume that $\delta>x^{-1/C}$, since otherwise $\Psi_s(\Delta,\delta^{-1};\delta^C,x/D)$ is empty.

Let $T$ be the set of $(d,m)\in [D,2D]\times [M,2M]$ for which the inner sum over $n$ in \eqref{eq53} is at least $\delta x/(DM)$. It suffices to show that
\[ |T| \ll \delta DM. \]
For each $d\in [D,2D]$, let $\psi_d\in \Psi_{s}(\Delta,\delta^{-1};\delta^{C},x/D)$ be such that the supremum in \eqref{eq53} is attained by $\psi_d$.  Also let $c_d \in [0,d)$ be such that the maximum over $c$ in \eqref{eq53} is attained by this value. Write $\psi_d(n) = \varphi_d(g_d(n)\Gamma_d)$ for some Lipschitz function $\varphi_d$ on a nilmanifold $G_d/\Gamma_d$ and some polynomial sequence $g_d$, with $\{g_d(n)\}_{1 \leq n \leq x/D}$ totally $\delta^{C}$-equidistributed. For $(d,m) \in T$, consider the polynomial sequence $h_{d,m}$ defined by
\[ h_{d,m}(n) := g_d(mn + b_{d,m}), \]
where $b_{d,m}$ is the unique number satisfying $1 \leq b_{d,m} \leq m$ and $db_{d,m} \equiv -c_d \Mod{m}$. Then by a change of variables we have
\[ \sum_{\substack{n \leq x/m \\ n \equiv c_{d}m^{-1}\Mod{d}}} \psi_d((mn-c_d)/d) = \sum_{n \leq x/(md)} \varphi_d(h_{d,m}(n)\Gamma_d) + O(1). \]
If $(d,m) \in T$, then the sums above have size at least $\delta x/(2DM)$, and thus the sequence  $(h_{d,m}(n)\Gamma_d)_{1 \leq n \leq x/(DM)}$ fails to be totally $\delta^2$-equidistributed. Hence by Theorem~\ref{thm:quant-leibman}, there is a nontrivial horizontal character $\chi_{d,m} \colon G_d \to \R/\Z$ with $\|\chi_{d,m}\| \ll \delta^{-O(1)}$ such that
\[ \|\chi_{d,m} \circ h_{d,m}\|_{C^{\infty}(x/DM)} \ll \delta^{-O(1)}. \]
Since the number of such characters is $\delta^{-O(1)}$, there is a subset $T' \subset T$ with $|T'| \gg \delta^{O(1)}|T|$ and a nontrivial horizontal character $\chi$ (not depending on $d,m$) with $\|\chi\| \ll \delta^{-O(1)}$, such that
\[ \| \chi \circ h_{d,m}\|_{C^{\infty}(x/DM)} \ll \delta^{-O(1)} \]
for all $(d,m) \in T'$. Suppose, for the purpose of contradiction, that $|T| \gg \delta^3 DM$, so that $|T'| \gg \delta^{O(1)}DM$. Then there are at least $\delta^{O(1)}D$ values of $d$ such that $(d,m) \in T'$ for at least $\delta^{O(1)}M$ values of $m$. For each such $d$, we will show that $\{g_d(n)\}_{1 \leq n \leq x/D}$ is not totally $\delta^C$-equidistributed, a contradiction.

Fix any such    $d$ for the rest of the proof. For $(d,m) \in T'$, we may explicitly write
\[ \chi \circ h_{d,m}(n) = \sum_{i=0}^s \beta_{d,m,i} n^i \]
for some coefficients $\beta_{d,m,i} \in \R$. 
Then by Lemma~\ref{lem:smoothness-norm}, there is a positive integer $q = O(1)$ such that
\begin{equation}\label{eq:beta-diophantine}
 \|q \beta_{d,m,i}\| \ll (x/DM)^{-i}\|\chi \circ h_{d,m}\|_{C^{\infty}(x/DM)} \ll (x/DM)^{-i} \delta^{-O(1)}. 
 \end{equation}
Using the definition of $h_{d,m}$, we see that
\[
\chi \circ h_{d,m}(n) = \chi\circ g_d(mn + b_{d,m}). \]
Thus if we write
\[ \chi \circ g_d(n) = \sum_{i=0}^s \alpha_{d,i} n^i \]
for some coefficients $\alpha_{d,i} \in \R$, then by the binomial formula we have the relations
\begin{equation}\label{eq:alpha-beta} 
\alpha_{d,i}m^i  = \beta_{d,m,i} - m^i\sum_{i < j \leq s} \binom{j}{i}\alpha_{d,j} b_{d,m}^{j-i}. 
\end{equation}
At this point we simply have a Diophantine problem to solve: we know that the $\beta_{d,m,i}$'s lie in ``major arcs" (of appropriate width) for many $m$, and we would like to deduce that the $\alpha_{d,i}$'s also lie in major arcs (of appropriate width). We claim that there is a positive integer $k \ll \delta^{-O(1)}$, such that
\begin{equation}\label{eq:alpha-diophantine} 
\|kq \alpha_{d,i}\| \ll (x/D)^{-i} \delta^{-O(1)}
\end{equation}
for each $1 \leq i \leq s$. To prove the claim we proceed by induction. For $i=s$ we have $\alpha_{d,s}m^s = \beta_{d,m,s}$. From~\eqref{eq:beta-diophantine} we see that 
\[ \|q\alpha_{d,s}m^s\| \ll (x/DM)^{-s} \delta^{-O(1)} \]
for at least $\delta^{O(1)}M$ values of $m$. By our assumption that $\delta > x^{-1/C}$ for some sufficiently large $C$, we have $(x/DM)^{-s} \leq x^{-\ee} \leq \delta^{O(1)}$. Hence we may apply Lemma~\ref{lem:recur} to the polynomial $m\mapsto q\alpha_{d,s}m^s$ to conclude that
\[ \|kq \alpha_{d,s}\| \ll (x/D)^{-s} \delta^{-O(1)} \]
for some positive integer $k \ll \delta^{-O(1)}$, completing the proof of the base case $i=s$. Now take $1 \leq i < s$, and assume that~\eqref{eq:alpha-diophantine} with some $k \ll \delta^{-O(1)}$ has already been proven for larger values of $i$.
By~\eqref{eq:beta-diophantine},~\eqref{eq:alpha-beta} and the induction hypothesis, we see that
\[ \|kq \alpha_{d,i}m^i\| \ll \|kq \beta_{d,m,i}\| +  \sum_{i < j \leq s} M^j \|kq \alpha_{d,j}\| \ll (x/DM)^{-i} \delta^{-O(1)}  \]
holds for at least $\delta^{O(1)}M$ values of $m$. By Lemma~\ref{lem:recur} applied to the polynomial $m \mapsto kq\alpha_{d,i}m^i$, there exists a positive integer $k' \ll \delta^{-O(1)}$ such that
\begin{align*}
 \|k'kq \alpha_{d,i}\| \ll (x/D)^{-i} \delta^{-O(1)}. 
\end{align*}
Replacing $k$ by $k'k$ completes the induction step, and the finishes the proof of the claim~\eqref{eq:alpha-diophantine}.\\

Finally, by Lemma~\ref{lem:smoothness-norm} we have
\[ \|kq\chi \circ g_d\|_{C^{\infty}(x/D)} \ll \delta^{-O(1)}. \]
Since $kq\chi$ is nontrivial and $\|kq\chi\| \ll \delta^{-O(1)}$, Lemma~\ref{lem:converse-leibman} implies that $\{g_d(n)\}_{1 \leq n \leq x/D}$ is not totally $\delta^{O(1)}$-equidistributed, a contradiction if  $C$ is chosen large enough.
\end{proof}

 The following type I estimate suffices for Theorems~\ref{equidist-1/3} and~\ref{equidist-1/2}.

\begin{proposition}[type I estimate]\label{typeI2}
Let $x \geq 2$ and $\ee > 0$. Let $1 \leq M \leq x^{1/2}$ and $1 \leq D \leq x^{1/2-\ee}$. Let $s \geq 1$, $\Delta \geq 2$, and $0 < \delta < 1/2$. 
Let $\psi \in \Psi_s(\Delta, \delta^{-1}; \delta^C, x)$ for some sufficiently large constant $C = C(s,\Delta,\ee)$. Then
\[ \sum_{D \leq d \leq 2D} \max_{(c,d)=1} \Big| \sum_{\substack{mn \leq x \\ M \leq m \leq 2M \\ mn \equiv c\Mod{d}}} a_m \psi(mn) \Big| \ll \delta x \|a\|_2 \]
for any sequence $\{a_m\}$.
\end{proposition}

As with Proposition \ref{typeI}, by applying the triangle inequality and Cauchy--Schwarz, we reduce Proposition \ref{typeI2} to the following slightly more general statement, with no coprimality condition on $c\pmod d$.

\begin{proposition}
Let $x \geq 2$ and $\ee > 0$. Let $1 \leq M \leq x^{1/2}$ and  $1 \leq D \leq x^{1/2-\ee}$. Let $s \geq 1$, $\Delta \geq 2$, and $0 < \delta < 1/2$. 
Let $\psi \in \Psi_s(\Delta, \delta^{-1}; \delta^C, x)$ for some sufficiently large constant $C = C(s,\Delta,\ee)$. Then
\begin{align}
\label{eq50}\sum_{D \leq d \leq 2D} \max_{c}  \sum_{\substack{M \leq m \leq 2M\\(m,d)=1}}\Big|\sum_{\substack{mn \leq x \\  n \equiv cm^{-1}\Mod{d}}} \psi(mn) \Big| \ll \delta x.
\end{align}
\end{proposition}

\begin{proof}
Again, for ease of notation we will suppress the dependence of implied constants on $s,\Delta,\ee$. 
 As before,
 we may assume that $\delta>x^{-1/C}$. Let $c_d\in [0,d)$ be such that the inner sum in \eqref{eq50} is maximal with $c=c_d$. Write $\psi(n) = \varphi(g(n)\Gamma)$ for some Lipschitz function $\varphi$ on a nilmanifold $G/\Gamma$ and some polynomial sequence $g$, with $\{g(n)\}_{1 \leq n \leq x}$ totally $\delta^{C}$-equidistributed. This is the same sum $\Sigma$ that appears in the proof of Proposition~\ref{typeI} if $\psi_d(n) = \varphi(g_d(n))$ with $g_d(n) = g(dn+c_d)$. 

Suppose that the desired bound \eqref{eq50} fails. Then, following the proof of Proposition~\ref{typeI} up to the claim~\eqref{eq:alpha-diophantine}, we find a nontrivial horizontal character $\chi$ with $\|\chi\| \ll \delta^{-O(1)}$, such that if we write
\[ \chi\circ g_d(n) = \sum_{i=0}^s \alpha_{d,i} n^i, \]
then there is a positive integer $q \ll \delta^{-O(1)}$ with
\begin{equation}\label{eq:typeI3}
\|q\alpha_{d,i}\| \ll (x/D)^{-i} \delta^{-O(1)} 
\end{equation}
for each $1 \leq i \leq s$ and at least $\delta^{O(1)}D$ values of $D \leq d \leq 2D$. Now, if we write
\[ \chi \circ g(n) = \sum_{i=0}^s \beta_i n^i, \]
then from the relations $g_d(n) = g(dn+c_d)$ we see that
\begin{equation}\label{eq:typeI5} 
\alpha_{d,i} = d^i \sum_{i \leq j \leq s} \binom{j}{i} \beta_j c_d^{j-i}. 
\end{equation}
We are then left with the Diophantine problem of deducing from~\eqref{eq:typeI3} that there exists a positive integer $k \ll \delta^{-O(1)}$, such that
\begin{equation}\label{eq:typeI4} 
\|kq\beta_i\| \ll x^{-i} \delta^{-O(1)} 
\end{equation}
for each $1 \leq i \leq s$. We prove~\eqref{eq:typeI4} by induction. Since $\alpha_{d,s} = d^s\beta_s$, we see that
\[ \|q\beta_s d^s\| \ll (x/D)^{-i}\delta^{-O(1)} \]
for at least $\delta^{O(1)}D$ values of $d$. It follows from Lemma~\ref{lem:recur} that there is a positive integer $k \ll \delta^{-O(1)}$ such that~\eqref{eq:typeI4} holds for $i=s$.

Now take $1 \leq i < s$, and assume that~\eqref{eq:typeI4} with some $k \ll \delta^{-O(1)}$ has already been proven for larger values of $i$. By~\eqref{eq:typeI3},~\eqref{eq:typeI5} and the induction hypothesis, we see that
\[ \|kqd^i \beta_i\| \ll \|kq\alpha_{d,i}\| + \sum_{i < j \leq s} D^j \|kq\beta_j\| \ll (x/D)^{-i} \delta^{-O(1)} \]
holds for at least $\delta^{O(1)}D$ values of $d$. By Lemma~\ref{lem:recur} again, there exists a positive integer $k' \ll \delta^{-O(1)}$ such that
\[ \|k'kq \beta_i\| \ll x^{-i} \delta^{-O(1)}. \]
Replacing $k$ by $k'k$ completes the induction step.\\

Finally, by Lemmas~\ref{lem:smoothness-norm} and~\ref{lem:converse-leibman},~\eqref{eq:typeI4} implies that $\{g(n)\}_{1 \leq n \leq x}$ is not totally $\delta^{O(1)}$-equidistributed, a contradiction if $C$ is chosen large enough.
\end{proof}

\section{Type II estimates}\label{sec:typeII}

We proceed to the type II estimates, which are more delicate than the type I estimates. In particular, each of Theorems \ref{thm:1/4}, \ref{thm:1/3} and \ref{thm:1/2} requires a different estimate. Despite this, the general structure of the proofs is still similar: 
\begin{enumerate}
\item We start with appropriate applications of Cauchy--Schwarz to eliminate the sequences $\{a_m\}, \{b_n\}$. 
\item Assuming that the desired estimate fails, we then use the quantitative Leibman theorem (Theorem~\ref{thm:quant-leibman}) to obtain recurrence properties of certain polynomials.
\item At this point we face a purely elementary Diophantine problem of deducing from the recurrence properties that the coefficients lie in appropriate major arcs.
\end{enumerate}
The readers are again encouraged to consider the special case when the nilequences are polynomial phase functions, in which case step (2) above is essentially no longer needed.
Propositions~\ref{type II-1/4},~\ref{type II-1/3}, and~\ref{typeII-lambda} below are the required type II estimates for Theorems~\ref{equidist-1/4},~\ref{equidist-1/3}, and~\ref{equidist-1/2}, respectively.

 We first need a lemma that generalizes Lemma~\ref{lem:recur} to polynomials in more than one variables.

\begin{lemma}\label{le_multileibman}
 Let $0<\delta<1/2$ be a parameter, and let $k,d\geq 1$, $0 < \ee < \delta/4^k$. Let $P(x_1,\ldots,x_k)$ be a polynomial of total degree $d$ with real coefficients, and let $N_1, \ldots, N_k \geq \delta^{-C}$ for some sufficiently large constant $C = C(d,k)$. Suppose that for proportion $\geq \delta$ of the integer tuples $(x_1,\ldots,x_k) \in [N_1] \times \cdots \times [N_k]$ we have 
 \begin{align*}
\|P(x_1,\ldots,x_k)\|\leq \ee.
 \end{align*} 
 Then there exists an integer $1\leq q\ll \delta^{-O_{d,k}(1)}$ such that 
 \begin{align*}
 \|q \alpha_{(i_1,\ldots,i_k)}\| \ll \ee\delta^{-O_{d,k}(1)} N_1^{-i_1} \cdots N_k^{-i_k},
 \end{align*}
for each $i_1,\ldots,i_k \geq 0$,
 where $\alpha_{(i_1,\ldots, i_k)}$ the coefficient of $x_1^{i_1}\cdots x_k^{i_k}$ in $P$.
 \end{lemma}

\begin{proof}
We perform an induction on $k$. The $k=1$ case follows immediately from Lemma~\ref{lem:recur}. Suppose that the case $k$ has been proved and that we are considering case $k+1$. Consider those values of $1 \leq y \leq N_{k+1}$ with the property that $\|P(x_1,\ldots, x_k, y)\|\leq \ee$ for at least $\delta/2$-proportion of $(x_1,\ldots, x_k)\in [N_1]\times \cdots \times [N_k]$. The proportion of such $y\in [N_{k+1}]$ is at least $\delta/2$. For each such $y$, consider the $k$-variable polynomial $Q_y$ given by
\[ Q_y(x_1,\ldots, x_k)= P(x_1,\ldots,x_k, y). \]
The coefficient of $x_1^{i_1}\cdots x_k^{i_k}$ in $Q_y$ is given by
\[ \sum_{i_{k+1} \geq 0} \alpha_{(i_1,\ldots,i_{k+1})}y^{i_{k+1}} . \]
Hence, by the induction hypothesis applied to $Q_y$, there exists a positive integer $q'= q'(y) \ll \delta^{-O_{d,k}(1)}$ such that 
\begin{equation}\label{eq:multileibman}
\left\|q' \sum_{i_{k+1}\geq 0} \alpha_{(i_1,\ldots,i_{k+1})}y^{i_{k+1}}\right\|\ll \ee\delta^{-O_{d,k}(1)} N_1^{-i_1}\cdots N_k^{-i_k} 
\end{equation}
for each $i_1,\ldots,i_k \geq 0$.
By the pigeonhole principle, there exists a value of $q' \ll \delta^{-O_{d,k}(1)}$ such that the above inequality holds for $\delta' \geq \delta^{O_{d,k}(1)}$ proportion of $y \in [N_{k+1}]$.  For fixed $i_1,\ldots,i_k$, the left-hand side in~\eqref{eq:multileibman} is a polynomial in $y$. Since $N_1, \ldots, N_k \geq \delta^{-C(d,k)}$, we may ensure that the upper bound in~\eqref{eq:multileibman} is less than $\delta'/4^k$ by choosing $C(d,k)$ large enough. 
Hence we may apply Lemma~\ref{lem:recur} to obtain a positive integer $q = q(i_1,\ldots,i_{k})\ll \delta^{-O_{d,k}(1)}$, such that
\begin{align*}
\|qq' \alpha_{(i_1,\ldots,i_{k+1})}\|\ll \ee\delta^{-O_{d,k}(1)} N_1^{-i_1}\cdots N_{k+1}^{-i_{k+1}}
\end{align*}
for each $i_{k+1} \geq 0$.
Finally, we can make this $q$ independent of $(i_1,\ldots,i_{k})$ by taking it to be the product of all the $q(i_1,\ldots,i_{k})$.
\end{proof}

Recall that, in the single variable case (Lemma~\ref{lem:recur}), the assumption $N \geq \delta^{-C}$ can be removed if $g$ is a monomial. This phenomenon extends to the multi-variable case as follows.

\begin{lemma}\label{le_multileibman2}
 Let $0<\delta<1/2$ be a parameter, and let $k,d\geq 1$, $0 < \ee < \delta/4^k$.  Let $N_1, \ldots, N_k$ be positive integers. For a sufficiently large constant $C = C(d,k)$, let $J = \{1 \leq j \leq k: N_j \geq \delta^{-C}\}$.  Let $P(x_1,\ldots,x_k)$ be a polynomial of total degree $d$, with $\alpha_{(i_1,\ldots, i_k)}$ the coefficient of $x_1^{i_1}\cdots x_k^{i_k}$ in $P$. Let $a_1,\ldots, a_k \geq 0$. Assume that $\alpha_{(i_1,\ldots,i_k)} = 0$ whenever $i_j = a_j$ for each $j \in J$ and $(i_1,\ldots,i_k) \neq (a_1,\ldots,a_k)$.
 
Suppose that for proportion $\geq \delta$ of the integer tuples $(x_1,\ldots,x_k) \in [N_1] \times \cdots \times [N_k]$ we have 
 \begin{align*}
\|P(x_1,\ldots,x_k)\|\leq \ee.
 \end{align*} 
 Then there exists an integer $1\leq q\ll \delta^{-O_{d,k,C}(1)}$ such that 
 \begin{align*}
 \|q \alpha_{(a_1,\ldots,a_k)}\| \ll \ee\delta^{-O_{d,k}(1)} N_1^{-a_1} \cdots N_k^{-a_k}.
 \end{align*}
 \end{lemma}

In short, we can still obtain Diophantine information on the coefficient $\alpha_{(a_1,\ldots,a_k)}$, provided that $x_1^{a_1}\cdots x_k^{a_k}$ is the only term in $P$  whose $x_j$-degree is equal to $a_j$ for each $j \in J$. 

\begin{proof}
If $J = \emptyset$, then $P$ is the monomial $P(x_1,\ldots,x_k) = \alpha_{(a_1,\ldots,a_k)}x_1^{a_1} \cdots x_k^{a_k}$, and the conclusion holds trivially.

Now suppose that $J$ is nonempty.
By the pigeonhole principle, there exists $x_j'$ for each $j \notin J$, such that for proportion $\geq \delta$ of the integer tuples $(x_j')_{j \in J} \in \prod_{j \in J} [N_j]$ we have $\|P(x_1',\ldots,x_k')\| \leq \ee$. By fixing these choices of $x_j'$ for $j \notin J$, we may view $P$ as a polynomial in the  variables $(x_j)_{j \in J}$ and apply Lemma~\ref{le_multileibman}. Our assumptions imply that  the coefficient of $\prod_{j\in J} x_j^{a_j}$ is $\alpha_{(a_1,\ldots,a_k)} \prod_{j \notin J} (x_j')^{a_j}$. Hence there exists a positive integer $q \ll \delta^{-O_{d,k}(1)}$ such that
\[ \left\|q \alpha_{(a_1,\ldots,a_k)} \prod_{j \notin J} (x_j')^{a_j} \right\| \ll \epsilon \delta^{-O_{d,k}(1)} \prod_{j \in J} N_j^{-a_j}. \]
The conclusion follows after replacing $q$ by $q \prod_{j \notin J} x_j^{a_j}$, since $x_j \leq N_j \leq \delta^{-C}$ for $j \notin J$.
\end{proof}

We are now ready to state and prove our type II estimates. Recall our notation $\|a\|_p$ from \eqref{eq49}.

\begin{proposition}[type II estimate with supremum]\label{type II-1/4}
Let $x \geq 2$ and $M \in [x^{1/4}, x^{3/4}]$. Let $1\leq D\leq M^{1/2-\ee}$ for some small $\ee > 0$. Let $s \geq 1$, $\Delta \geq 2$, and $0 < \delta < (\log x)^{-1}$. Then for a large constant $C = C(s,\Delta, \ee)$ we have
\[ \sum_{D\leq d\leq 2D} \max_{(c,d)=1} \sup_{\psi \in \Psi_s(\Delta, \delta^{-1}; \delta^C, x/D)} \Big| \sum_{\substack{x\leq mn \leq 2x \\ M\leq m\leq 2M \\ mn \equiv c\Mod{d}}} a_m b_n \psi((mn-c)/d) \Big| \ll \delta x \|a\|_2 \|b\|_4, \]
for any sequences $\{a_m\}, \{b_n\}$.
\end{proposition}

\begin{proof}
As in the proof of the type I estimates, we may assume that $\delta>x^{-1/C}$, since otherwise $\Psi_s(\Delta, \delta^{-1}; \delta^C, x/D)$ is empty. We need to bound the sum 
\[ \Sigma := \sum_{D\leq d\leq 2D}  \sum_{\substack{x\leq mn\leq 2x \\ M\leq m\leq 2M \\ mn \equiv c_d\Mod{d}}} a_m b_n \psi_d((mn-c_d)/d) \]
for any choice of $1 \leq c_d \leq d$ with $(c_d,d)=1$ and $\psi_d \in \Psi_s(\Delta,\delta^{-1}; \delta^C, x/D)$. This can be rewritten as
\[ \Sigma = \sum_{\substack{x\leq mn\leq 2x \\ M\leq m\leq 2M}} a_m b_n F(mn), \]
where the function $F$ is defined by
\[ F(k) := \sum_{\substack{D\leq d\leq 2D \\  k \equiv c_d\Mod{d}}} \psi_d((k-c_d)/d) \]
for $x \leq k \leq 2x$.
By the Cauchy--Schwarz inequality, we deduce
\[ \Sigma^2 \ll M \|a\|_2^2 \sum_{N/2\leq n_1,n_2\leq 2N}|b_{n_1}||b_{n_2}| \Big| \sum_{\substack{M\leq m\leq 2M \\ x/n_1\leq m\leq 2x/n_1 \\ x/n_2\leq m\leq 2x/n_2}} F(mn_1) \overline{F(mn_2)} \Big|, \]
where $N = x/M$. By Cauchy--Schwarz, this is further
\begin{align*}
\ll MN\|a\|_2^2 \|b\|_4^2 \Big(\sum_{N/2\leq n_1,n_2\leq 2N} \Big| \sum_{\substack{M\leq m\leq 2M \\ x/n_1\leq m\leq 2x/n_1  \\ x/n_2\leq m\leq 2x/n_2}} F(mn_1) \overline{F(mn_2)} \Big|^2 \Big)^{1/2} .  
\end{align*}
Using the trivial bound $O(M)$ for the inner sum over $m$, we get
\[ \Sigma^2 \ll M^{1/2} x \|a\|_2^2 \|b\|_4^2 \Big(\sum_{N/2\leq n_1,n_2\leq 2N} \Big| \sum_{\substack{M\leq m\leq 2M \\ x/n_1\leq m\leq 2x/n_1  \\ x/n_2\leq m\leq 2x/n_2}} F(mn_1) \overline{F(mn_2)} \Big| \Big)^{1/2}. \]
If the desired bound $\Sigma \ll \delta x \|a\|_2 \|b\|_4$ fails, then
\begin{equation}\label{eq:shao}
\sum_{N/2\leq n_1,n_2\leq 2N} \Big| \sum_{\substack{M\leq m\leq 2M \\ x/n_1\leq m\leq 2x/n_1  \\ x/n_2\leq m\leq 2x/n_2}} F(mn_1) \overline{F(mn_2)} \Big|  \geq \delta^4 N^2M. 
\end{equation}
We are now in a position to apply~\cite[Lemma 3.3]{shao}. For it to be applicable, we need the following two conditions:
\[ D^2 \leq \delta^C M, \ \ N^{-c} < \delta^4 < (\log D)^{-1} \] 
for some sufficiently large constant $C = C(s,\Delta)$ and sufficiently small constant $c = c(s,\Delta) > 0$. Both conditions are satisfied by our assumption on $\delta$. Hence~\cite[Lemma 3.3]{shao} implies that $\psi_d \notin \Psi_s(\Delta, \delta^{-1}; \delta^{O_{s,\Delta}(1)}, x/D)$ for some $d$, a contradiction.
\end{proof}

\begin{remark}
We briefly mention here how the sum~\eqref{eq:shao} is treated in~\cite[Lemma 3.3]{shao}, and where the key condition $D^2 \leq \delta^CM$ comes from. After expanding out using the definition of $F$, the sum over $m$ can be written as
\[ \sum_m F(mn_1) \overline{F(mn_2)} = \sum_{d_1, d_2} \sum_{m: mn_i \equiv c_{d_i}\Mod{d_i}} \psi_{d_1}((mn_1-c_{d_1})/d_1) \overline{\psi_{d_2}((mn_2-c_{d_2})/d_2)}. \]
The inner sum above is of length $\asymp M/D^2$ (assuming that $(d_1,d_2)=1$), and can be expressed in the form (after a change of variables)
\[ \sum_{m \in J(d_1,d_2,n_1,n_2)} \psi_{d_1,d_2,n_1,n_2}(m) \]
for an appropriate interval $J(d_1,d_2,n_1,n_2)$ and nilsequence $\psi_{d_1,d_2,n_1,n_2}$. One can apply the quantitative Leibman theorem to obtain recurrence property for the underlying polynomial sequences involved in $\psi_{d_1,d_2,n_1,n_2}$, for many $(d_1,d_2,n_1,n_2)$, and then proceed to solve the Diophantine problem that shows up. Since the length of this sum is $M/D^2$, these recurrence properties are only meaningful when $M/D^2 \geq \delta^{-C}$ for some large constant $C$. We also remark here that the original statement in~\cite[Lemma 3.3]{shao} has an error: the assumption $10Q^2 \leq L$ there should be $Q^2 \leq \delta^CL$ for some sufficiently large constant $C$ in terms of $s,d$. This is needed in the proof in order for Lemma~\ref{lem:recur} to be applicable.
\end{remark}

We shall then state and prove a type II estimate with well-factorable weights.

\begin{proposition}[Well-factorable type II estimate]\label{typeII-lambda}
Let $x \geq 2$, let $c \neq 0$ be a fixed integer, and let $M \in [x^{1/2}, x^{3/4}]$. Let $\ee > 0$ be a small constant, and let $(\lambda_d)$ be a well-factorable sequence of level $x^{1/2-\ee}$.  Let $s \geq 1$, $\Delta \geq 2$, $0 < \delta < 1/2$. Then for a large constant $C = C(s,\Delta,\ee)$ and $\psi \in \Psi_s(\Delta,\delta^{-1};\delta^C,x)$, we have
\begin{equation}\label{eq-wtIIb} 
\Big|\sum_{\substack{d \leq x^{1/2-\ee} \\ (d,c)=1}}  \lambda_d \sum_{\substack{x\leq mn\leq 2x \\ M\leq m\leq 2M \\ mn \equiv c\Mod{d}}} a_m b_n \psi(mn)\Big|  \ll \delta x \|a\|_2 \|b\|_4 (\log x)^3,
\end{equation}
for any sequences $\{a_m\}, \{b_n\}$.
\end{proposition}

 By the well-factorability of $(\lambda_d)_{d \leq x^{1/2-\varepsilon}}$, there exist $(\gamma_d), (\theta_d)$ with $|\gamma_d|, |\theta_d|\leq 1$ such that $\lambda_d=\sum_{d=d_1d_2}\gamma_{d_1}\theta_{d_2}$, and with $(\gamma_d)$ and $(\theta_d)$ being supported on $[1,U]$ and $[1,V]$, respectively, where 
\begin{align}\label{eq12}
U=x^{1-\varepsilon}/M,\quad V=Mx^{-\frac{1}{2}}.   
\end{align}
Therefore, Proposition \ref{typeII-lambda} is an immediate consequence of the following more general statement.

\begin{proposition}\label{typeII-lambda-general}
Let $x \geq 2$, let $c' \neq 0$ be a fixed integer, and let $M \in [x^{1/2}, x^{3/4}]$. Let $\ee > 0$ be a small constant, and let $U,V\geq 1$ satisfy $U\leq x^{1-\varepsilon}/M$, $UV\leq x^{1/2-\varepsilon}$, and $UV^2\leq Mx^{-\varepsilon}$.  Let $s \geq 1$, $\Delta \geq 2$, $0 < \delta < 1/2$. Then for a large constant $C = C(s,\Delta,\ee)$ and $\psi \in \Psi_s(\Delta,\delta^{-1};\delta^C,x)$, we have
\begin{equation*}
\sum_{U\leq u\leq 2U}\sum_{\substack{V\leq v\leq 2V\\(v,c'u)=1}}  \max_{(c,u)=1}\Big|\sum_{\substack{x\leq mn\leq 2x \\ M\leq m\leq 2M \\ mn \equiv c\Mod{u}\\mn\equiv c'\Mod{v}}} a_m b_n \psi(mn)\Big|  \ll \delta x \|a\|_2 \|b\|_4 \log x,
\end{equation*}
for any sequences $\{a_m\}, \{b_n\}$.
\end{proposition}

\begin{proof}
As in the proof of the type I estimates, we suppress the dependence of the implied constants on $s,\Delta,\ee$ in the proof, and we may assume that $\delta > x^{-1/C}$. We may also assume that $U\geq x^{\varepsilon/2}$, since otherwise $UV^2\leq Mx^{-\varepsilon}$ implies $UV\leq M^{1/2}x^{-\varepsilon/4},$ in which case the claim follows from Proposition \ref{type II-1/4}.

By replacing absolute values with an arbitrary unimodular sequence, the left-hand side of~\eqref{eq-wtIIb} can be written as
\begin{align}\label{eq3}
\sum_{U\leq u\leq 2U}\sum_{\substack{V\leq v\leq 2V\\(v,c'u)=1}}\theta_{u,v}\sum_{\substack{x\leq mn\leq 2x\\M\leq m\leq 2M\\mn\equiv c_u\Mod{u}\\mn\equiv c'\Mod{v}}}a_m b_n \psi(mn) :=\Sigma  
\end{align}
for some $|\theta_{u,v}|\leq 1$ and some $c_u$ coprime to $u$. By assumption, we have
\begin{align}\label{eq13}
UV^2\leq Mx^{-\varepsilon},\quad x^{\varepsilon/2}\leq U\leq x^{1-\varepsilon}/M. 
\end{align}
 By the triangle inequality, we have
\begin{align*}
\Sigma\ll  \sum_{U\leq u\leq 2U}\sum_{M\leq m\leq 2M}|a_m| \Big|\sum_{\substack{V\leq v\leq 2V\\(v,c'u)=1}}\theta_{u,v}\sum_{\substack{x/m\leq n\leq 2x/m\\mn\equiv c_u\Mod{u}\\mn\equiv c'\Mod{v}}}b_n \psi(mn)\Big|.  
\end{align*}
Let $N:=x/M$. By applying the Cauchy--Schwarz inequality and expanding, we get
\begin{align*}
\Sigma^2\ll  MU \|a\|_2^2 \sum_{U\leq u\leq 2U} \sum_{\substack{V\leq v_1,v_2\leq 2V\\(v_1,c'u)=(v_2,c'u)=1}}\theta_{u,v_1}\overline{\theta_{u,v_2}}\sum_{N/2\leq n_1,n_2\leq 2N}b_{n_1}\overline{b_{n_2}}\sum_{\substack{x\leq mn_1,mn_2\leq 2x\\M\leq m\leq 2M\\mn_i\equiv c_u \Mod{u}\\ mn_i\equiv c'\Mod{v_i}}} \psi(mn_1) \overline{\psi(mn_2)}. 
\end{align*}
Using the triangle inequality, the trivial inequality $|b_{n_1}b_{n_2}| \ll |b_{n_1}|^2 + |b_{n_2}|^2$, and symmetry, we have
\begin{align*}
\Sigma^2 \ll  MU \|a\|_2^2 \sum_{U\leq u\leq 2U} \sum_{\substack{ V\leq v_1,v_2\leq 2V\\(v_1,c'u)=(v_2,c'u)=1}}\sum_{\substack{N/2\leq n_1,n_2\leq 2N \\ n_1 \equiv n_2\Mod{u(v_1,v_2)}}}|b_{n_1}|^2\Big|\sum_{\substack{x \leq mn_1,mn_2\leq 2x \\M\leq m\leq 2M\\mn_i\equiv c_u \Mod{u}\\ mn_i\equiv c'\Mod{v_i}}} \psi(mn_1) \overline{\psi(mn_2)} \Big|,
\end{align*}
where we are free to add the constraint $n_1\equiv n_2 \Mod{u(v_1,v_2)}$ since otherwise the inner sum over $m$ is empty. 
Note that the inner sum over $m$ is bounded trivially by $O(M/u[v_1,v_2]) = O(M(v_1,v_2)/UV^2)$, since $u[v_1,v_2] \ll UV^2\leq  Mx^{-\ee}$ by~\eqref{eq13}.
Separating out the terms with $n_1=n_2$ and with $(v_1,v_2) > \delta^{-2}$, and applying the Cauchy--Schwarz inequality, we get
\[ \Sigma^2 \ll MU \|a\|_2^2 \left(\Sigma_0 +\Sigma_0' + \Sigma_1^{1/2} \Sigma_2^{1/2}\right), \]
where
\begin{align*} 
\Sigma_0 &= \sum_{U\leq u\leq 2U} \sum_{\substack{ V\leq v_1,v_2\leq 2V\\(v_1,c'u)=(v_2,c'u)=1}}\sum_{N/2\leq n\leq 2N} |b_n|^2 \sum_{\substack{x \leq mn\leq 2x \\M\leq m\leq 2M\\mn\equiv c_u \Mod{u}\\ mn\equiv c'\Mod{v_i}}} |\psi(mn)|^2, \\
\Sigma_0' &=  \sum_{U\leq u\leq 2U} \sum_{\substack{ V\leq v_1,v_2\leq 2V\\(v_1,c'u)=(v_2,c'u)=1\\ (v_1,v_2)>\delta^{-2}}}\sum_{\substack{N/2\leq n_1,n_2\leq 2N \\ n_1 \equiv n_2\Mod{u(v_1,v_2)}\\n_1 \neq n_2}}|b_{n_1}|^2\Big|\sum_{\substack{x \leq mn_1,mn_2\leq 2x \\M\leq m\leq 2M\\mn_i\equiv c_u \Mod{u}\\ mn_i\equiv c'\Mod{v_i}}} \psi(mn_1) \overline{\psi(mn_2)} \Big|
\\
\Sigma_1 &=  \sum_{U\leq u\leq 2U} \sum_{\substack{ V\leq v_1,v_2\leq 2V\\(v_1,c'u)=(v_2,c'u)=1\\ (v_1,v_2) \leq \delta^{-2}}}\sum_{\substack{N/2\leq n_1,n_2\leq 2N\\ n_1\equiv n_2\Mod{u(v_1,v_2)}\\ n_1 \neq n_2}} |b_{n_1} |^4, 
\\
 \Sigma_2 &=
  \sum_{U\leq u\leq 2U} \sum_{\substack{ V\leq v_1,v_2\leq 2V\\(v_1,c'u)=(v_2,c'u)=1\\ (v_1,v_2) \leq \delta^{-2}}}\sum_{\substack{N/2\leq n_1,n_2\leq 2N\\ n_1\equiv n_2\Mod{u(v_1,v_2)}\\ n_1 \neq n_2}}\Big|\sum_{\substack{x \leq mn_1,mn_2\leq 2x \\M\leq m\leq 2M\\mn_i\equiv c_u \Mod{u}\\ mn_i\equiv c'\Mod{v_i}}} \psi(mn_1) \overline{\psi(mn_2)} \Big|^2.
  \end{align*}
For $\Sigma_0$, using the trivial bound for the inner sum over $m$ we get
\[ \Sigma_0 \ll \frac{M}{UV^2} \sum_{U\leq u\leq 2U} \sum_{\substack{ V\leq v_1,v_2\leq 2V\\(v_1,c'u)=(v_2,c'u)=1}} (v_1,v_2) \sum_{N/2\leq n\leq 2N} |b_n|^2.  \]
Since
\[ \sum_{V \leq v_1,v_2 \leq 2V} (v_1,v_2) \leq \sum_{v_0 \leq 2V} \sum_{\substack{V \leq v_1,v_2 \leq 2V \\ v_0 \mid v_1, v_0 \mid v_2}} v_0 \ll \sum_{v_0 \leq 2V} \frac{V^2}{v_0} \ll V^2 \log x,  \]
we have
\[ \Sigma_0 \ll MN \|b\|_2^2 \log x = x \|b\|_2^2 \log x. \]
For $\Sigma_0'$,  using the trivial bound again for the inner sum over $m$, we see that the contribution to $\Sigma_0'$ from those terms with $(v_1,v_2) = v_0$ for a given $v_0 \leq 2V$ is
\[ \ll  U \left(\frac{V}{v_0}\right)^2 \left(\sum_{n_1} |b_{n_1}|^2\right) \frac{N}{Uv_0} \cdot \frac{Mv_0}{UV^2} = \frac{Nx}{Uv_0^2} \|b\|_2^2.  \]
Summing this over all $v_0 \geq\delta^{-2}$ gives
\[ \Sigma_0' \ll \delta^2 \frac{Nx}{U} \|b\|_2^2. \]
For $\Sigma_1$ we clearly have
\[ \Sigma_1 \ll UV^2 \left(\sum_n |b_n|^4\right) \frac{N}{U} = V^2N^2 \|b\|_4^4. \]
Thus it suffices to show that
\begin{align}\label{eq4} 
\Sigma_2 \ll \frac{\delta^4 x^2}{U^2V^2}. 
\end{align}
Indeed, once~\eqref{eq4} is established, we can combine it with our bounds for $\Sigma_0, \Sigma_0', \Sigma_1$ to get
\begin{align*} 
\Sigma^2 &\ll  MU \|a\|_2^2 \left(x \|b\|_2^2\log x + \delta^2\frac{Nx}{U} \|b\|_2^2  + VN\|b\|_4^2 \cdot \frac{\delta^2x}{UV}\right) \\
&\ll (MUx + \delta^2x^2) (\log x) \|a\|_2^2 \|b\|_4^2.  
\end{align*}
Since $MUx \leq x^{2-\ee} \leq \delta^2x^2$ by~\eqref{eq13} and our assumption on $\delta$, this gives the desired bound for $\Sigma$.

Now we turn to bounding $\Sigma_2$. 
Let $\Gamma$ be the set of $(n_1,n_2,u,v_1,v_2)$ with
\begin{align*} &U\leq u\leq 2U, \ \ V\leq v_1,v_2\leq 2V, \ \ N/2\leq n_1,n_2\leq 2N, \ \ n_1 \neq n_2, \\ 
&(u,c)=(v_1,c')=(v_2,c')=1,\ \ (v_1,v_2) \leq \delta^{-2}, \ \ n_1\equiv n_2 \Mod{u(v_1,v_2)} 
\end{align*}
such that
\begin{equation}\label{eq5-wtII}
\Big|\sum_{\substack{x\leq mn_1,mn_2\leq 2x\\M\leq m\leq 2M\\mn_i\equiv c_u \Mod{u}\\ mn_i \equiv c'\Mod{v_i}}} \psi(mn_1) \overline{\psi(mn_2)} \Big| \geq \frac{\delta^2 M}{UV^2}. 
\end{equation}
Suppose, for the purpose of contradiction, that~\eqref{eq4} fails. Then we must have
\[
|\Gamma| \geq \delta^{O(1)} V^2N^2.
\]
 For $1\leq \ell\leq \delta^{-2}$, let $\Gamma_{\ell}$ be the set of $(n_1,n_2,u,v_1,v_2)\in \Gamma$ such that $(v_1,v_2)=\ell$. By the pigeonhole principle we have 
 \begin{equation}\label{eq-Gamma-wtII}
 |\Gamma_{\ell}|\gg \delta^{O(1)}V^2N^2
 \end{equation}
  for some $\ell$. Fix such an $\ell$.

Let $\psi = \varphi \circ g$. For $(n_1,n_2,u,v_1,v_2) \in \Gamma_{\ell}$, note that the congruence condition in the sum in~\eqref{eq5-wtII} is equivalent to $m \equiv B\Mod{uv_1v_2/\ell}$ for some $B = B(n_1,n_2,u,v_1,v_2) \in [1, uv_1v_2/\ell]$. Let $h = h_{n_1,n_2,u,v_1,v_2}\colon \Z \to G \times G$ be the polynomial sequence defined by
\[ h(n) := (g(uv_1v_2/\ell\cdot n_1n+n_1B),  g(uv_1v_2/\ell\cdot n_2n + n_2B)), \]
and $\varphi \otimes \overline{\varphi} \colon G \times G \rightarrow \C$ be the function defined by
\[ \varphi \otimes \overline{\varphi}(x_1, x_2) := \varphi(x_1) \overline{\varphi(x_2)} \]
for $x_1,x_2 \in G$. After a change of variables,~\eqref{eq5-wtII} implies that
\[ \Big|\sum_{n \in I} (\varphi \otimes \overline{\varphi}) \circ h(n)\Big|  \gg \frac{\delta^2 M}{UV^2}, \]
where $I=I_{u,v_1,v_2,n_1,n_2}$ is an interval contained in $[1, 2
\ell M/UV^2]$ and $|I| \asymp \ell M/UV^2$. Since $\ell \leq \delta^{-2}$, we have
\[ \frac{\delta^2 M}{UV^2} \gg \delta^{O(1)} |I|. \]
Hence the polynomial sequence $(h(n))_{n \leq 2\ell M/UV^2}$ fails to be totally $\delta^{O(1)}$-equidistributed, and hence by Theorem~\ref{thm:quant-leibman}, there is a nontrivial horizontal character $\chi = \chi(n_1,n_2,u,v_1,v_2)$ on $G \times G$ with $\|\chi\| \ll \delta^{-O(1)}$, such that
\begin{equation}\label{eq6-wtII} 
\|\chi \circ h\|_{C^{\infty}(2\ell M/UV^2)} \ll \delta^{-O(1)}.
\end{equation}
After pigeonholing (which replaces $\Gamma_{\ell}$ by a subset of it, but the lower bound in \eqref{eq-Gamma-wtII} remains valid with potentially larger implied constants), we may assume that $\chi$ is independent of $n_1,n_2,u,v_1,v_2$. We can write $\chi = (\chi_1,\chi_2)$, where $\chi_1,\chi_2$ are horizontal characters on $G$ with $\|\chi_i\| \ll \delta^{-O(1)}$, at least one of which, say $\chi_1$, is nontrivial, so that $\chi(x_1,x_2) = \chi_1(x_1)+\chi_2(x_2)$ for $x_1,x_2 \in G$. If we write
\[ \chi_1\circ g(n) = \sum_{j=0}^{s}\alpha_j n^j, \ \ \chi_2 \circ g(n) = \sum_{j=0}^{s}\beta_j n^j, \]
then we can write
\[ 
\begin{split}
\chi \circ h(n) &= \sum_{j=0}^s \alpha_j (uv_1v_2/\ell\cdot n_1n+n_1B)^j + \sum_{j=0}^s \beta_j (uv_1v_2/\ell\cdot n_2n+n_2B)^j = \sum_{j=0}^s \gamma_j n^j \\
\end{split}
\]
for some coefficients  $\gamma_j = \gamma_j(n_1,n_2,u,v_1,v_2)$. In particular,  
\begin{align*}
\gamma_s(n_1,n_2,u,v_1,v_2) = \alpha_s (u v_1v_2 n_1/\ell)^s + \beta_s (uv_1v_2n_2/\ell)^s .   
\end{align*}
It follows from~\eqref{eq6-wtII} that
\begin{align}\label{eq28}
\|\gamma_j(n_1,n_2,u,v_1,v_2)\|\ll \frac{\delta^{-O(1)}}{|I|^j}    
\end{align}
for each $1 \leq j \leq s$ and $(n_1,n_2,u,v_1,v_2) \in \Gamma_{\ell}$. We are now left with the problem of deducing the Diophantine properties of the $\alpha_j$'s and $\beta_j$'s from those of the $\gamma_j$'s, which will eventually contradict the equidistribution assumption on $g$.

Consider the $5$-variable polynomial
\[ Q(n_1,z,u, v_1', v_2')= \gamma_s(n_1,n_1+u\ell z,u, \ell v_1', \ell v_2'). \]
By \eqref{eq28}, we have
\[
\|Q(n_1,z,u,v_1',v_2')\|\ll \frac{\delta^{-O(1)}}{|I|^s} \ll \delta^{-O(1)}\left(\frac{UV^2}{M}\right)^s
\]
for $\gg \delta^{O(1)}V^2N^2$ values of $(n_1,z,u,v_1',v_2')$, with $n_1 \asymp N$, $z\asymp N/(\ell U)$, $u \asymp U$, $v_1',v_2' \asymp  V/\ell$. Note that
\begin{itemize}
\item the sizes of $n_1,z$ are both $\geq \delta^{O(1)}x^{\ee/2} \geq \delta^{-O(1)}$ by~\eqref{eq13};
\item the term $(\alpha_s+\beta_s)\ell^s \cdot (uv_1'v_2'n_1)^s$ is the only one appearing in $Q(n_1,z,u,v_1',v_2')$ whose degrees in the $n_1, z$ variables are $s, 0$, respectively;
\item the term $\beta_s \ell^{2s} \cdot u^{2s}(v_1'v_2'z)^s$ is  the only one appearing in $Q(n_1,z,u,v_1',v_2')$ whose degrees in the $n_1, z$ variables are $0, s$, respectively.
\end{itemize}

Hence Lemma~\ref{le_multileibman2} is applicable to $Q$ and implies that there exists $1\leq q\ll \delta^{-O(1)}$ such that
\begin{align*}
\|q \ell^s (\alpha_s+\beta_s)\|&\ll \delta^{-O(1)}\left(\frac{UV^2}{M}\right)^s \frac{1}{(UV^2N)^s}
\ll \delta^{-O(1)}x^{-s},\\
\|q \ell^{2s} \beta_s\|&\ll \delta^{-O(1)}\left(\frac{UV^2}{M}\right)^s \frac{1}{(UV^2N)^s} \ll \delta^{-O(1)}x^{-s}.\\
\end{align*}
Hence $q_s = q\ell^{2s} \ll \delta^{-O(1)}$ has the property that
\[ \|q_s\alpha_s\|\ll \delta^{-O(1)}x^{-j},\quad \|q_s\beta_s\|\ll \delta^{-O(1)}x^{-j}.  \]
We then show by induction that for $1\leq j\leq s$ and some $1\leq q_j\ll \delta^{-O(1)}$ with $q_{j+1} \mid q_j$, we have
\begin{align}\label{eq27}
\|q_j\alpha_j\|\ll \delta^{-O(1)}x^{-j},\quad \|q_j\beta_j\|\ll \delta^{-O(1)}x^{-j}.    
\end{align}
The case $j=s$ has been proved, so we may assume that all the cases $s\geq j'>j$ have been proved and that we are considering case $j$. The coefficient $\gamma_j(n_1,n_2,u,v_1,v_2)$ can be explicitly written as
\begin{align*}
\gamma_j = & \alpha_j (uv_1v_2n_1/\ell)^j + \beta_j(uv_1v_2n_2/\ell)^j + \\
& \sum_{j < j' \leq s} \binom{j'}{j} \left(\alpha_{j'} (uv_1v_2n_1/\ell)^j (n_1B)^{j'-j} + \beta_{j'}(uv_1v_2n_2/\ell)^j (n_2B)^{j'-j} \right) \\
\end{align*}
Since \eqref{eq27} holds for all indices $j'>j$, one easily verifies that
\[ \left\| q_{j'}\alpha_{j'} (uv_1v_2n_1/\ell)^j (n_1B)^{j'-j}\right\| \ll \frac{\delta^{-O(1)}}{ |I|^{j'}}, \ \ \left\| q_{j'}\beta_{j'} (uv_1v_2n_2/\ell)^j (n_2B)^{j'-j}\right\| \ll \frac{\delta^{-O(1)}}{|I|^{j'}}.  \]
Defining $q_{j+1}=\prod_{j<j'\leq s}q_{j'}$, we see that the coefficient $\gamma_j(n_1,n_2,u,v_1,v_2)$ may be written in the form
\begin{align*}
\left\|q_{j+1} \gamma_j\right\| = \left\|q_{j+1}\alpha_j (u v_1v_2 n_1/\ell)^j + q_{j+1}\beta_j (uv_1v_2n_2/\ell)^j\right\| +O(\delta^{-O(1)}|I|^{-j-1}).   
\end{align*}
Thus by~\eqref{eq28} we have
\begin{align*}
\|q_{j+1} \alpha_j (u v_1v_2 n_1/\ell)^j + q_{j+1} \beta_j (uv_1v_2n_2/\ell)^j\|\ll \frac{\delta^{-O(1)}}{|I|^j}    
\end{align*}
for $(n_1,n_2,u,v_1,v_2) \in \Gamma_{\ell}$. Precisely as in the case $j=s$ (with $s$ in that argument now replaced by $j$), this implies \eqref{eq27} for $j$. Now, for $q = q_1\ll \delta^{-O(1)}$ we have for every $1\leq j\leq s$  that
\begin{align*}
\|q\alpha_j\|\ll \delta^{-O(1)}x^{-j},\quad \|q\beta_j\|\ll \delta^{-O(1)}x^{-j}.    
\end{align*}
Finally, by Lemma~\ref{lem:smoothness-norm} we have 
\[ \|q\chi_1\circ g\|_{C^{\infty}(x)}\ll \delta^{-O(1)}. \]
Since $q\chi_1$ is nontrivial and $\|q\chi_1\| \ll \delta^{-O(1)}$, Lemma~\ref{lem:converse-leibman} implies that $\{g(n)\}_{1 \leq n \leq x}$ is not totally $\delta^{O(1)}$-equidistributed, a contradiction if $C$ is chosen large enough.
\end{proof}

\begin{proposition}[type II estimate with a fixed nilsequence]\label{type II-1/3}
Let $x \geq 2$, let $\ee > 0$ be a small constant, let $M \in [x^{1/2}, x^{2/3+\varepsilon/2}]$.  Let $s \geq 1$, $\Delta \geq 2$, $0 < \delta < 1/2$. Then for a large constant $C = C(s,\Delta,\ee)$ and $\psi \in \Psi_s(\Delta,\delta^{-1};\delta^C,x)$, we have
\begin{equation}\label{eq-wtII} 
\sum_{d \leq x^{1/3-\ee}}  \max_{(c,d)=1}\Big|\sum_{\substack{x\leq mn\leq 2x \\ M\leq m\leq 2M \\ mn \equiv c\Mod{d}}} a_m b_n \psi(mn)\Big|  \ll \delta x \|a\|_2 \|b\|_4 (\log x)^3,
\end{equation}
for any sequences $\{a_m\}, \{b_n\}$.
\end{proposition}

\begin{proof}
This follows directly from  Proposition~\ref{typeII-lambda-general} (with $\varepsilon$ replaced by $\varepsilon/3$ there), taking $U\leq x^{1/3-\varepsilon}$, $V=1$ there.
\end{proof}

\section{Combining the type I and type II estimates}\label{sec:combine}

The type I and II estimates in the previous sections have been tailored so that we will be able to shortly conclude the proofs of Theorems \ref{equidist-1/4}, \ref{equidist-1/3}, \ref{equidist-1/2} by appealing to Vaughan's identity, stated here for convenience.

\begin{lemma}[Vaughan's identity]\label{le_vaughan} Let $x\geq 2$ and $x^{2/3}< n\leq x$. Then we can write 
\begin{align*}
\Lambda(n)=\sum_{j\leq J}(\alpha_j*\delta_j(n)+\beta_j*\gamma_j(n)),
\end{align*}
where 
\begin{enumerate}[(i)]
    \item $J\ll (\log x)^{10}$;
    
    \item $|\alpha_j(n)|, |\beta_j(n)|, |\gamma_j(n)|\ll d_3(n)(\log x)$ for each $j\leq J$, and $\alpha_j, \beta_j,\gamma_j$ are each supported inside dyadic intervals;
    
    \item $\supp(\alpha_j)\subset [1,2x^{1/3}]$ and $\supp(\beta_j), \supp(\gamma_j)\subset [x^{1/3},x^{2/3}]$ for each $j\leq J$.
    
    \item $\delta_j(n)=1$ or $\delta_j(n)=\log n$.
\end{enumerate}
\end{lemma}

\begin{proof}
This follows from the formulation of Vaughan's identity in \cite[Proposition 13.4]{iwaniec-kowalski}, taking $y=z=x^{1/3}$ there and splitting the variables into dyadic intervals. 
\end{proof}

\begin{proof}[Proof of Theorems \ref{equidist-1/4}, \ref{equidist-1/3}, \ref{equidist-1/2}] The bound $\ll \eta^{\kappa}x(\log x)^2$ in the statements of the theorems follows trivially from the Brun--Titchmarsh inequality if $\eta^{\kappa}>1/\log x$, so we may assume that $\eta\leq (\log x)^{-1/\kappa}$ for small enough $\kappa>0$. Then it suffices to prove a bound of the form $\ll \eta^{\kappa} x(\log x)^{O(1)}$ for the sums in Theorems \ref{equidist-1/4}, \ref{equidist-1/3}, \ref{equidist-1/2}.

Let $\varepsilon>0$ be small. Note that by the triangle inequality and crude estimation, it suffices to prove each proposition with the $n$ sum ranging over any dyadic interval $[x',2x']\subset [x^{1-\varepsilon/10},x]$ instead of $[1,x]$. Now, by Vaughan's identity, for $n\in [x',2x']$ we can write $\Lambda(n)$ as a sum of $\ll(\log x')^{10}$
type I convolutions $\alpha*\delta(n)$ with $\delta\in \{1,\log\}$ and type II convolutions $\beta*\gamma(n)$, where $\alpha,\beta,\gamma$ satisfy condition (2) of Lemma~\ref{le_vaughan} and  $\supp(\alpha)\subset [1,(x')^{1/3+\varepsilon}]$, $\supp(\beta), \supp(\gamma)\subset [(x')^{1/3},(x')^{2/3+\varepsilon}]$. By the triangle inequality, it suffices to prove each of Theorems \ref{equidist-1/4}, \ref{equidist-1/3}, \ref{equidist-1/2} with $\Lambda(n)$ replaced by one of these type I or type II convolutions. 

For the type I convolutions, we can readily apply Proposition~\ref{typeI} in the case of Theorem~\ref{equidist-1/4} and Proposition~\ref{typeI2} in the case of Propositions~\ref{equidist-1/3}, \ref{equidist-1/2} (possibly applying partial summation first to get rid of the $\log n$ factor in $\alpha*\log(n)$). For the type II convolutions, in turn, we can appeal to Proposition~\ref{type II-1/4} in the case of Theorem~\ref{equidist-1/4}, Proposition~\ref{type II-1/3} in the case of Theorem~\ref{equidist-1/3}, and Proposition~\ref{typeII-lambda} in the case of Theorem~\ref{equidist-1/2} (taking the sequence $a_m$ in these propositions to be the one of $\beta_m,\gamma_m$ whose support is contained in $[(x')^{1/2},x']$). This completes the proof.
\end{proof}

In view of the reduction in Section \ref{sec:equidistribution} to the equidistributed case, our main Theorems \ref{thm:1/4}, \ref{thm:1/3} and~\ref{thm:1/2} follow from Theorems \ref{equidist-1/4}, \ref{equidist-1/3}, \ref{equidist-1/2}, together with Propositions \ref{prop:sieve1}, \ref{prop:sieve2}. Hence, to complete the proofs of our main theorems, it now suffices to prove these two propositions.

\begin{proof}[Proof of Propositions \ref{prop:sieve1} and \ref{prop:sieve2}]

We prove Proposition~\ref{prop:sieve1} by replacing the condition $(n,W)=1$ by a standard sieve weight, and then using our type I estimate (Proposition~\ref{typeI}). The proof of Proposition~\ref{prop:sieve2} is completely analogous, using Proposition~\ref{typeI2} instead of Proposition~\ref{typeI}.

Let $D = x^{0.01}$, and let $(\lambda_{\ell}^{\pm})_{\ell \leq D}$ be upper and lower linear sieve weights, defined in \cite[Section 6.2]{iwaniec-kowalski}. Note that if $\psi \in \Psi_s(\Delta, \eta^{-\kappa}; \eta, x/d)$ then its real and imaginary parts lie in $\Psi_s(\Delta, \eta^{-\kappa}; \eta, x/d)$ as well. Thus, by splitting $\psi$ into its real and imaginary parts we may restrict attention to real-valued $\psi$.

Apply the upper bound sieve to the non-negative function $\psi+1$ to obtain
\[
\begin{split}
\sum_{\substack{n\leq x\\n\equiv c\Mod {d} \\ (n,W)=1}} (\psi((n-c)/d) + 1) & \leq  \sum_{\substack{n \leq x\\ n\equiv c\Mod{d}}} \left(\sum_{\ell \mid (n,W)} \lambda_\ell^+\right) (\psi((n-c)/d) + 1)\\
& = \sum_{\ell \mid W} \lambda_\ell^+ \sum_{\substack{n \leq x \\ n \equiv c\Mod{d} \\ \ell\mid n}} [\psi((n-c)/d) + 1].
\end{split}
\]
Apply the lower bound sieve to the constant function $1$ to obtain
\[ \sum_{\substack{n\leq x\\n\equiv c\Mod {d} \\ (n,W)=1}} 1 \geq  \sum_{\ell\mid W} \lambda_\ell^- \sum_{\substack{n \leq x \\ n \equiv c\Mod{d} \\ \ell\mid n}} 1. \]
Using the fundamental lemma of sieve theory~\cite[Lemma 6.3]{iwaniec-kowalski} (with $g(\ell)=\frac{1}{\ell}1_{(\ell,d)=1}$ and $s=(\log D)/(\log w)$ there), we get
\[ \sum_{\ell\mid W} \lambda_\ell^{\pm} \sum_{\substack{n \leq x \\ n \equiv c\Mod{d} \\ \ell\mid n}} 1 = \frac{x}{d} \prod_{\substack{p \mid W \\ p\nmid d}} \left(1 - \frac{1}{p}\right) + O_A\left(\frac{x}{d(\log x)^{2A}}\right). \]
 It follows that
\[ \sum_{\substack{n\leq x\\n\equiv c\Mod {d} \\ (n,W)=1}} \psi((n-c)/d) \leq \sum_{\ell \mid W} \lambda_\ell^+ \sum_{\substack{n \leq x \\ n \equiv c\Mod{d} \\ \ell\mid n}} \psi((n-c)/d) + O_A\left(\frac{x}{d(\log x)^{2A}}\right).
\]
 By applying the inequality above with $\psi$ replaced by $-\psi$, we deduce that
\[ \Big|\sum_{\substack{n\leq x\\n\equiv c\Mod {d} \\ (n,W)=1}} \psi((n-c)/d)\Big| \leq \Big|\sum_{\ell \mid W} \lambda_\ell^+ \sum_{\substack{n \leq x \\ n \equiv c\Mod{d} \\ \ell\mid n}} \psi((n-c)/d)\Big| + O_A\left(\frac{x}{d(\log x)^{2A}}\right).
\]
Thus it suffices to show that
\[
\sum_{d \leq x^{1/2-\ee}} \max_{(c,d)=1} \sup_{\psi \in \Psi_s(\Delta, \eta^{-\kappa};\eta,x/d)} \Big| \sum_{\ell \mid W} \lambda_\ell^+ \sum_{\substack{n \leq x \\ n \equiv c\Mod{d} \\ \ell\mid n}} \psi((n-c)/d)\Big| \ll \eta^{\kappa} x \log x. 
\]
 By a change of variables $n=\ell m$, we can rewrite the expression inside the absolute value sign in the form
\[ \sum_{\substack{\ell m \leq x \\ \ell m \equiv c\Mod{d}}} a_\ell \psi((\ell m-c)/d), \]
where $a_{\ell} = \lambda_{\ell}^+ 1_{\ell\mid W}$. Since $(\lambda_\ell^+)$ (and thus $(a_\ell)$) is supported on $\ell \leq D = x^{0.01}$, the desired estimate follows from our type I estimate (Proposition~\ref{typeI}) by  taking $\kappa = C^{-1}$, where $C = C(s,\Delta,\ee)$ is the constant from Proposition~\ref{typeI}.
\end{proof}

The proof of Theorem \ref{thm_divisor} proceeds similarly to the proofs of Theorems \ref{equidist-1/4}, \ref{equidist-1/3}, \ref{equidist-1/2}, except that we need an analogue of Heath-Brown's identity for the multiplicative functions $d_k(n)$ and $1_{S}(n)$. This follows from the following general lemma on decompositions of multiplicative functions.

\begin{lemma}[A Heath-Brown type decomposition for multiplicative functions]\label{le_HB} Let $\kappa,Q,n_0,$ $A\geq 1$ and $\varepsilon>0$ be fixed. Let  $f:\mathbb{N}\to \mathbb{C}$ be a multiplicative function satisfying $|f(n)|\leq d_{\kappa}(n)$ and such that the sequence $(f(p))_{p\geq n_0}$ is periodic with period $Q\in \mathbb{N}$. Then for $x\geq 2$ and  $1\leq n\leq x$ we can write
\begin{align*}
f(n)=\sum_{j\leq J}\alpha_{1,j}*\cdots *\alpha_{R,j}(n)+O(E(n)),   
\end{align*}
where
\begin{enumerate}[(i)]
\item $J\ll (\log x)^{O(1)}$ and $R\ll 1$;

\item All the $\alpha_{i,j}(n)$ are $\ll  d(n)^{O(1)}$ in absolute value;

\item  All the $\alpha_{i,j}(n)$ are supported on intervals of the form $I_{i,j}=[y,y']$ with $y<y'\leq 2y$, and if $y>x^{\varepsilon}$, then we have $\alpha_{i,j}(n)=\chi(n)1_{[y,y']}(n)$ or $\alpha_{i}(n)=\chi(n)(\log n)1_{[y,y']}(n)$, where $\chi$ is a Dirichlet character (depending on $i,j$) whose modulus divides $Q$;

\item We have $$E(n)\ll ((\log x)^{O(1)}(1_{\exists p>w:\, p^2\mid n}+1_{\exists m>x^{\varepsilon}:\,m\mid n,\, m\,\textnormal{is}\, w-\textnormal{smooth}})+(\log x)^{-A})d_{\kappa}(n),$$
where $w=x^{1/(\lceil 3/\varepsilon\rceil (\log \log x))}$.
\end{enumerate}
\end{lemma}

\begin{proof}
This is a minor modification of arguments in \cite[Section 5]{matomaki-teravainen}. For the sake of completeness, we provide the details.

We may assume that $x$ is large enough. In what follows, all constants in $O(\cdot)$ notation may depend on the fixed quantities $\kappa, Q, n_0,\varepsilon,A$. Let $C=\lceil 3/\varepsilon\rceil$, $K=C\log \log x$, and $w=x^{1/K}$. Then we may factorize
\begin{align*}
f(n)=\sum_{\substack{0\leq k\leq K\\0\leq l\leq C}}\frac{1}{k!\ell!}\sum_{\substack{n=mp_1\cdots p_kq_1\cdots q_{\ell}\\w<p_i\leq x^{1/C}\\q_i>x^{1/C}\\p'\mid m\implies p'\leq w}}f(p_1)\cdots f(p_k)f(q_1)\cdots f(q_{\ell})f(m)+O(d_{\kappa}(n)1_{\exists p>w:\,\, p^2\mid n}),    
\end{align*}
where the error term arises from those $n$ that have are divisible by the square of some prime $>w$. Next, by the orthogonality of characters, for any prime $q>\max\{n_0,Q\}$ we can write 
\begin{align}\label{e26}
f(q)=\frac{1}{\varphi(Q)}\sum_{\chi\pmod{Q}}\sum_{1\leq b\leq Q}a_b\overline{\chi(b)}\chi(p), 
\end{align}
where $a_b\in \mathbb{C}$ is such that $f(p)=a_b$ for all primes $p\equiv b\pmod Q$, $p\geq n_0$. Applying \eqref{e26} to each of $f(q_i)$, we can decompose $f(n)$ as a linear combination (with $O(1)$ coefficients) of $O(1)$ sums of the form 
\begin{align*}
\sum_{\substack{0\leq k\leq K\\0\leq l\leq C}}\frac{1}{k!\ell!}\sum_{\substack{n=mp_1\cdots p_kq_1\cdots q_{\ell}\\w<p_i\leq x^{1/C}\\q_i>x^{1/C}\\p'\mid m\implies p'\leq w}}f(p_1)\cdots f(p_k)\chi_1(q_1)\cdots \chi_{\ell}(q_{\ell})f(m)+O(d_{\kappa}(n)1_{\exists p>w:\,\, p^2\mid n})
\end{align*}
where $\chi_i$ are characters to modulus $Q$. We split each $q_i$ into short intervals of the form $[x(1-1/(\log x)^{A+1})^{j+1},x(1-1/(\log x)^{A})^j]$, where $j\geq 0$. Then  $f(n)$ becomes a linear combination of $O(1)$ sums of the form
\begin{align*}
\sum_{I_1,\ldots, I_{\ell}\in \mathcal{I}}\sum_{\substack{0\leq k\leq K\\0\leq l\leq C}}\frac{1}{k!\ell!}\sum_{\substack{n=mp_1\cdots p_kq_1\cdots q_{\ell}\\w<p_i\leq x^{1/C}\\q_i\in I_i\\p'\mid m\implies p'\leq w}}f(p_1)\cdots f(p_k)\chi_1(q_1)\cdots \chi_{\ell}(q_{\ell})f(m)+O(d_{\kappa}(n)1_{\exists p>w:\,\, p^2\mid n}),
\end{align*}
where $\mathcal{I}$ is the collection of intervals of the form $I=[x(1-1/(\log x)^{A+1})^{j+1},x(1-1/(\log x)^{A+1})^j]$, with $j\geq 0$ and $I\subset [x^{1/C},x]$.

We then add the von Mangoldt weight to each of the $q_i$ variables by writing, for $n\in I=[y,y(1+(\log x)^{-A-1})]$,
\begin{align*}
1_{\mathbb{P}}(n)&=\frac{\Lambda(n)}{\log n}+O(1_{n\,\, \textnormal{not squarefree}})=\frac{\Lambda(n)}{\log y}+O((\log x)^{-A})+O(1_{n\,\, \textnormal{not squarefree}}).
\end{align*}
Estimating the contribution of the two error terms trivially by the triangle inequality, the result is that $f(n)$ can be written as a linear combination (with $O(1)$ coefficients) of $O((\log x)^{(A+2)C})$ sums of the form 
\begin{align*}
 \sum_{\substack{0\leq k\leq K\\0\leq l\leq C}}\frac{1}{k!\ell!}\sum_{\substack{n=mp_1\cdots p_km_1\cdots m_{\ell}\\w<p_i\leq x^{1/C}\\q_i\in I_i\\p'\mid m\implies p'\leq w\\m\leq x^{\varepsilon}}}f(p_1)\cdots f(p_k)f(m)\prod_{j=1}^{\ell}\Lambda(m_j)\chi_j(m_j),   
\end{align*}
plus an error term
\begin{align*}
E(n)\ll d_{\kappa}(n)((\log x)^{O(1)}1_{\exists p>w:\,\, p^2\mid n}+(\log x)^{O(1)}1_{\exists m>x^{\varepsilon}:\,m\mid n,\, m\,\textnormal{is}\, w-\textnormal{smooth}})+(\log x)^{-A}),
\end{align*}
where the second summand in the upper bound for $E(n)$ arose from the restriction to $m\leq x^{\varepsilon}$. For the summands with $k\ll 1$, we could now use Heath--Brown's identity (\cite[Proposition 13.3]{iwaniec-kowalski} with $K=\lceil 1/\varepsilon\rceil$ there) to each of $\Lambda(m_i)$  to obtain a decomposition of the desired form. To make the argument work for large $k$, it suffices to show that
\begin{align*}
g(n):=\sum_{0\leq k\leq K}\frac{1}{k!}\sum_{\substack{n=p_1\cdots p_k\\w<p_i\leq x^{1/C}}}f(p_1)\cdots f(p_k)
\end{align*}
can be written as a linear combination of $O((\log x)^{O(1)})$ convolutions of the form $\beta_{1,i}*\cdots *\beta_{M,i}$ with $M\ll 1$ and $|\beta_{i,j}|\ll (\log x)^{O(1)}$ and with each $\beta_{i,j}$ supported in $[x^{1/C},x^{3/C}]$. But this follows from a simple grouping of $p_1\cdots p_k$ into subproducts of size $\in [x^{1/C},x^{3/C}]$, detailed in \cite[Section 5]{matomaki-teravainen}.
\end{proof}

We are now in a position to also conclude the proof of Theorem \ref{thm_divisor} that involves the functions $d_k$ and $1_S$.

\begin{proof}[Proof of Theorem \ref{thm_divisor}] Let $\varepsilon>0$ be small. As previously, it suffices to prove a dyadic version of Theorem \ref{thm_divisor}, where we sum over $n\in [x',2x']\subset [x^{1-\varepsilon/10},x]$ instead of $n\leq x$. Let $f(n)=d_k(n)$ with $k\in \mathbb{C}$ fixed, or $f(n)=1_S(n)$. By Lemma \ref{le_HB} (with $Q=4$ and $\kappa=|k|+1$), we can decompose 
$f$ as a sum of $\ll (\log x)^{O(1)}$ convolutions each of which is of the form $\alpha_{1}*\cdots *\alpha_R(n)$ with $\alpha_i$ as in Lemma \ref{le_HB}, plus an error term $E(n)$ that is of the form given in the lemma. 

We will show that each of the convolutions $\alpha_1*\cdots *\alpha_R(n)$ can be written as either a type I convolution $\alpha*\delta(n)$ or a type II convolution $\beta*\gamma(n)$, where
\begin{enumerate}[(i)]
    \item  $|\alpha(n)|\ll (\log x)^{O(1)}d(n)^{O(1)}$ and $\supp(\alpha)\subset [1,C(x')^{1/3}]$ for some $C\ll 1$ and $\delta(n)=\chi(n)$ or $\delta(n)=\chi(n)\log n$ with $\chi$ of modulus dividing $4$.
    
    \item $|\beta(n)|, |\gamma(n)|\ll (\log x)^{O(1)}d(n)^{O(1)}$  and $\supp(\beta)\subset [(x')^{1/3},(x')^{2/3}]$.;
\end{enumerate}

Let $\alpha_i$ be supported inside $[y_i,2y_i]$. Since $n$ is supported on $[x',2x']$, we must have $y_1\cdots y_R\asymp x'$. If $y_i \geq (x')^{2/3}$ for some $i$, then we can form a type I convolution by taking $\delta=\alpha_i$, $\alpha=\alpha_1*\cdots*\alpha_{i-1}*\alpha_{i+1}*\cdots*\alpha_R$ (or $\alpha(n)=1_{n=1}$ if $R=1$), and now we have written $\alpha_1*\cdots*\alpha_R$ in the form (1) above. If instead $y_i \in [(x')^{1/3},(x')^{2/3}]$ for some $i$, then we form a type II convolution of the form (2) above by defining $\beta=\alpha_i, \gamma=\alpha_1*\cdots*\alpha_{i-1}*\alpha_{i+1}*\cdots*\alpha_R$. If all $y_i \leq (x')^{1/3}$, then choose $j$ to be the  smallest index with $y_1\cdots y_j\geq (x')^{1/3}$.  We also must have $y_1\cdots y_j\leq (x')^{2/3}$ since $y_j \leq (x')^{1/3}$. Therefore, taking $\beta = \alpha_1*\cdots*\alpha_j$ and $\gamma=\alpha_{j+1}*\cdots *\alpha_{R}$  we have formed the desired type II convolution.

Now that each $\alpha_1*\cdots *\alpha_R$ has been written as a type I or II convolution satisfying (1) or (2) above, we can apply the same argument that was used to conclude the proofs of Theorems \ref{equidist-1/4}, \ref{equidist-1/3}, \ref{equidist-1/2} (noting that the $\chi(n)$ weight is harmless in our type I Bombieri--Vinogradov estimates by splitting $n$ into progressions $\pmod 4$) to see that the analogue of each of those theorems holds with $\alpha_1*\cdots *\alpha_R$ in place of $\Lambda(n)$.

To conclude the proof of Theorem \ref{thm_divisor}, it now suffices to show that
\begin{align}\label{eq30}
\sum_{d\leq x^{1/2-\varepsilon}}\max_{(c,d)=1}\sum_{\substack{n\leq x\\ n\equiv c\Mod d}}|E(n)|\ll_{B} x/(\log x)^{B},    
\end{align}

To prove \eqref{eq30}, we first note that by Cauchy--Schwarz and Shiu's bound \cite{shiu}, for any $d\leq x^{1/2-\varepsilon}$ and $c$ coprime to $d$ we have
\begin{align*}
\sum_{\substack{n\leq x\\ n\equiv c\pmod d,\\\exists w<p\leq x^{1/2}:\,\,p^2\mid n}}d_{|k|+1}(n)&\ll \left(\sum_{\substack{n\leq x\\ n\equiv c\pmod d}}d_{|k|+1}(n)^2\right)^{1/2}\left(\sum_{\substack{n\leq x\\ n\equiv c\pmod d,\\\exists w<p\leq x^{1/2}:\,\,p^2\mid n}}1\right)^{1/2}\\
&\ll (\log x)^{(|k|+1)^2/2}\left(\frac{x}{dw^{1/2}}+\frac{x^{3/4}}{d^{1/2}}\right).
\end{align*}
Hence, summing over $d\leq x^{1/2-\varepsilon}$, the term $(\log x)^{O(1)}1_{\exists p>w:\,\, p^2\mid n}d_{|k|+1}(n)$ in the definition of $E(n)$ gives an admissible contribution on the left of \eqref{eq30}. We also have by Shiu's bound, for any $d\leq x^{1/2-\varepsilon}$ and $c$ coprime to $d$, that
\begin{align*}
\sum_{\substack{n\leq x\\n\equiv c\pmod d}}d_{|k|+1}(n)\ll  \frac{x(\log x)^{|k|}}{d},   
\end{align*}
so the contribution of the term $(\log x)^{-A}d_{|k|+1}(n)$ in the definition of $E(n)$ is also acceptable, provided that $A$ is chosen large in terms of $B$. 

Lastly, we deal with the term $(\log x)^{O(1)}1_{\exists m>x^{\varepsilon}:\,m\mid n,\, m\,\textnormal{is}\, w-\textnormal{smooth}}\cdot d_{|k|+1}(n)$ in the definition of $E(n)$. Let $\rho_{w}(n)$ denote the indicator function of $w$-smooth numbers. It suffices to show that
\begin{align}\label{eq32}
\sum_{d\leq x^{1/2-\varepsilon}}\max_{(c,d)=1}\sum_{\substack{n_1n_2\leq x\\n_1n_2\equiv c\pmod d\\n_1> x^{\varepsilon}}}d_{|k|+1}(n_1)\rho_w(n_1)d_{|k|+1}(n_2)\ll_Bx/(\log x)^{B}.    
\end{align}
By a bilinear Bombieri--Vinogradov estimate \cite[Theorem 17.4]{iwaniec-kowalski} (with $\Delta=1$ there), for any $x^{\varepsilon}\leq N\leq x$ we have
\begin{align*}
&\sum_{d\leq x^{1/2-\varepsilon}}\max_{(c,d)=1}\sum_{\substack{n\leq x\\n\equiv c\Mod d}}\sum_{\substack{n=n_1n_2\\N\leq n_1\leq 2N}}\rho_{w}(n_1)d_{|k|+1}(n_1)d_{|k|+1}(n_2)\\
&\ll \left(\sum_{N\leq n_1\leq 2N}\rho_w(n_1)^2d_{|k|+1}(n_1)^2\right)^{1/2}\left(\sum_{n_2\leq x/N}d_{|k|+1}(n_2)^2\right)^{1/2}.
\end{align*}
Now, applying Cauchy--Schwarz to separate $d_{|k|+1}(n_1)^2$ and $\rho_{w}(n_1)^2$ and 
recalling that $w=x^{1/(\lceil 3/\varepsilon \rceil(\varepsilon\log \log x))}$ and applying a standard upper bound for smooth numbers \cite[Theorem 1.1]{hildebrand-tenenbaum}, this is bounded by $\ll x(\log x)^{O(1)}\exp(-(\log \log x)(\log \log \log x))\ll_{B}x/(\log x)^{B}$.

Summing dyadically over $N$, we conclude that  \eqref{eq32} holds, so \eqref{eq30} follows. 
\end{proof}

\section{Sieve lemmas}

The rest of the paper is devoted to the proofs of our applications (Theorems \ref{thm:bdd}, \ref{thm:chen} and \ref{thm:app3}). We first formulate weighted versions of the sieves of Maynard and Chen, which will be needed subsequently. 

\begin{definition}
We say that a $k$-tuple $(h_1,\ldots, h_k)$ of integers is \emph{admissible} if for every prime $p$ there exists $a\in \mathbb{Z}$ such that $p\nmid a+h_i$ for all $1\leq i\leq k$.
\end{definition}

\begin{proposition}[Maynard's sieve]\label{prop_maynard} For any $\theta\in (0,1)$, $C_0\geq 1$ and $k\in \mathbb{N}$, there exist $C=C(\theta)$ and $\sigma=\sigma(\theta,k)$ such that the following holds.

Let $(\omega_n)_{n\leq x} $ be any sequence of nonnegative real numbers with $x\geq x_0$, and let $(L_1,\ldots, L_k)$ be an admissible $k$-tuple of linear forms with $L_i(n)=a_in+b_i$ and $1\leq a_i,|b_i|\leq (\log x)^{1/100}$. Suppose that $(\omega_n)$ satisfies the following hypotheses:
\begin{enumerate}[(i)]
    \item  (Prime number theorem) For all $1\leq i\leq k$ and some $\delta > 0$ we have 
\begin{align*}
\frac{1}{k}\sum_{i=1}^k\frac{\varphi(a_i)}{a_i}\sum_{\substack{n\leq x\\L_i(n)\in \mathbb{P}}}\omega_n \geq \frac{\delta}{\log x}\sum_{n\leq x}\omega_n;
\end{align*}

\item  (Well-distribution in arithmetic progressions) We have
\begin{align*}
\sum_{r\leq x^{\theta}}\,\,\max_{c\Mod r}\Big|\sum_{\substack{n\leq x\\n\equiv c\Mod r}}\omega_n-\frac{1}{r}\sum_{n\leq x}\omega_n\Big|\ll \frac{\sum_{n\leq x}\omega_n}{(\log x)^{101k^2}}.    
\end{align*}

\item (Bombieri--Vinogradov) For all $1\leq i\leq k$  we have
\begin{align*}
\sum_{r\leq x^{\theta}}\max_{(L_i(c),r)=1}\Big|\sum_{\substack{n\leq x\\n\equiv c\Mod r\\L_i(n)\in \mathbb{P}}}\omega_n-\frac{1}{\varphi_{L_i}(r)}\sum_{\substack{n\leq x\\L_i(n)\in \mathbb{P}}}\omega_n\Big|\ll \frac{\sum_{n\leq x}\omega_n}{(\log x)^{101k^2}},    
\end{align*}
where for $L=\ell n+b$ we define $\varphi_{L}(n):=\varphi(|\ell|n)/\varphi(|\ell|)$.

\item (Brun--Titchmarsh) We have 
\begin{align*}
\max_{c\Mod r}\sum_{\substack{n\leq x\\n\equiv c\Mod r}}\omega_n\leq \frac{C_0}{r}\sum_{n\leq x}\omega_n, \end{align*}
for all $1\leq r\leq x^{\theta}$.
\end{enumerate}

Then, for $x\gg_{k,\theta,C_0}1$ we have
\begin{align*}
\sum_{\substack{n\leq x\\ |[L_1(n),\ldots, L_k(n)]\cap \mathbb{P}|\geq C^{-1}\delta \log k\\ p\mid \prod_{i=1}^k L_i(n)\Longrightarrow p>x^{\sigma}}}\omega_n\gg_{k,\theta,\delta} \frac{\mathfrak{S}(L_1,\ldots, L_k)}{(\log x)^k}\sum_{n\leq x}\omega_n,    
\end{align*}
where the singular series $\mathfrak{S}(L_1,\ldots, L_k)$ is given by 
\begin{align*}
\mathfrak{S}(L_1,\ldots, L_k):=\prod_{p}\Big(1-\frac{1}{p}\Big)^{-k}\Big(1-\frac{|\{n\in \mathbb{Z}/p\mathbb{Z}:\,\, L_1(n)\cdots L_k(n)\equiv 0\pmod p\}|}{p}\Big)>0.   
\end{align*}
\end{proposition}

\begin{proof}
This is \cite[Theorem 6.2]{matomaki-shao}, which adds weights to the corresponding statement in~\cite{maynard-compositio} (in \cite[Theorem 6.2]{matomaki-shao} the result is stated for dyadic sums, but this clearly makes no difference in the proof).
\end{proof}

\begin{proposition}[Chen's sieve]\label{prop_chen} Let $(\omega_n)_{n\leq x} $ be any sequence of nonnegative real numbers, and let $x$ be large enough. Let $\varepsilon>0$ be a small enough absolute constant. Suppose that $(\omega_n)$ satisfies the following hypotheses:
\begin{enumerate}[(i)]
\item  (Bombieri--Vinogradov with well-factorable weights) We have
\begin{align*}
\Big|\sum_{\substack{r\leq x^{1/2-\varepsilon}\\(r,2)=1}}\lambda_r\Big(\sum_{\substack{n\leq x\\n+2\equiv 0\Mod r\\n\in \mathbb{P}}}\omega_n-\frac{1}{\varphi(r)}\sum_{n\leq x}\frac{\omega_n}{\log(n+1)}\Big)\Big|\ll \frac{\sum_{n\leq x}\omega_n}{(\log x)^{10}}    
\end{align*}
for any $\lambda_r$ that is either well-factorable of level $x^{1/2-\varepsilon}$ or a convolution of the form $\lambda=1_{p\in [P,P')}*\lambda'$ with $\lambda'$ well-factorable of level $x^{1/2-\varepsilon}/P$ and $2P\geq P'\geq P\in [x^{1/10},x^{1/3-\varepsilon}]$. 

\item (Bombieri--Vinogradov for almost primes with well-factorable weights) For $j\in \{1,2\}$ we have
\begin{align*}
\Big|\sum_{\substack{r\leq x^{1/2-\varepsilon}\\(r,2)=1}}\lambda_r\Big(\sum_{\substack{n\leq x\\n\equiv 0\Mod r\\n+2\in B_j}}\omega_n-\frac{1}{\varphi(r)}\sum_{\substack{n\leq x\\n+2\in B_j}}\omega_n\Big)\Big|\ll \frac{\sum_{n\leq x}\omega_n}{(\log x)^{10}},    
\end{align*}
where $\lambda_r$ is any well-factorable sequence of level $x^{1/2-\varepsilon}$ and 
\begin{align*}
B_1&=\{p_1p_2p_3:\,\, x^{1/10}\leq p_1\leq x^{1/3-\varepsilon},\,\, x^{1/3-\varepsilon}\leq p_2\leq (2x/p_1)^{1/2},\,\, p_3\geq x^{1/10}\},\\
B_2&=\{p_1p_2p_3:\,\, x^{1/3-\varepsilon}\leq p_1\leq p_2\leq (2x/p_1)^{1/2},\,\, p_3\geq x^{1/10}\}.
\end{align*}

\item ($(\omega_n)$ is not concentrated on almost primes): For $j\in \{1,2\}$ we have
\begin{align*}
\sum_{\substack{n\leq x\\n\in B_j}}\omega_n \leq (1+\varepsilon)|B_j\cap [1,x]|\cdot \frac{1}{x}\sum_{n\leq x}\omega_n.   
\end{align*}
\end{enumerate}
Then we have
\begin{align*}
\sum_{\substack{n\leq x\\ n\in \mathbb{P}\\n+2\in P_2\\p\mid n+2\Longrightarrow p\geq x^{1/10}}}\omega_n\geq \frac{\delta_0}{(\log x)^2}\sum_{n\leq x}\omega_n-O(x^{0.9}\max_{n}\omega_n),    
\end{align*}
for some absolute constant $\delta_0>0$.
\end{proposition}

\begin{proof}
This is essentially \cite[Theorem 6.4]{matomaki-shao}, which adds weights to Chen's sieve argument. Note that since the sieve weights appearing in the proof of the weighted Chen's sieve in \cite[Appendix A]{matomaki-shao} are supported on squarefree numbers, the quantity $\mu(r)^2\lambda_r$ in \cite[Hypothesis 6.3]{matomaki-shao} can be replaced with just $\lambda_r$. Note also that the $1+\varepsilon$ factor in Hypothesis (iii) was $1+o(1)$ there, but inspecting the proof in \cite[Appendix A]{matomaki-shao} (in particular Subsection A.5), this relaxation makes no difference.
\end{proof}

\section{Bounded gaps between primes in nil-Bohr sets}

In this section we shall prove Theorem \ref{thm:bdd}. We begin with Weyl's exponential sum estimate and an analogous estimate over the primes, formulated in a way that is convenient for us. 

\begin{lemma}[Weyl sum bounds over the integers and primes]\label{le_weyl}
Let $s\geq 1$ be fixed. There exists $C_s>0$ such that the following holds.  Let $Q(x)=\alpha_sx^s+\cdots +\alpha_1 x+\alpha_0$ with $\alpha_0,\ldots, \alpha_s\in \mathbb{R}$. Let $W\geq 1$ be fixed. Lastly, let $x\geq 2$, $\delta\in (0,1)$ and $1\leq h\leq W$ be such that
\begin{align*}
\left|\sum_{n\leq x}e(Q(n))\right|>\delta x\quad \textnormal{or}\quad \left|\sum_{\substack{n\leq x\\Wn+h\in \mathbb{P}}}e(Q(n))\right|>\delta \frac{W}{\varphi(W)}\frac{x}{\log x}.  
\end{align*}
Then there exists $1\leq \ell\ll \delta^{-C_s}$ such that $\|\ell \alpha_j\|\ll \delta^{-C_s}(\log x)^{C_s}/x^j$ for all $1\leq j\leq s$.
\end{lemma}

\begin{proof}
In the case of the Weyl sum over the integers, this follows from \cite[Proposition 9.3]{green-tao-nilmanifolds}. In the case of the exponential sum over the primes, in turn, we may combine Theorem~\ref{equidist-1/4} (taking $\psi(n)=Q((n-h)/W)$ there and considering only the $d=W$ term in the sum) with the quantitative Leibman theorem (Theorem \ref{thm:quant-leibman}) and Lemma \ref{lem:smoothness-norm} to obtain the result. 
\end{proof}

We then state a lemma that provides a convenient minorant function for $1_{\|x\|< \rho}$.

\begin{lemma}[Vinogradov]\label{le_vinogradov} For any $\rho\in (0,1/4)$ and $\eta\in (0,\rho/2)$, there exists a function $g=g_{\rho,\eta}$ such that we have the minorant property $0\leq g(x)\leq 1_{\|x\|< \rho}$ and such that:
\begin{enumerate}[(i)]
    \item $g(x)$ is $1$-periodic and $g(x)=0$ for $x\in [-1/2,1/2]\setminus (-\rho,\rho)$.
    
\item  We have a Fourier expansion
\begin{align*}
g(x)=2(1-\eta)\rho+\sum_{|j|>0}c_je(jx),    
\end{align*}
where $|c_j|\leq 10\rho$.

\item  We have
\begin{align*}
\sum_{|j|>K}|c_j|\leq \frac{10\eta^{-1}}{K}    
\end{align*}
for any $K\geq 1$.
\end{enumerate}
\end{lemma}

\begin{proof}
This follows from \cite[Lemma 12]{vinogradov}.
\end{proof}

Let $k\in \mathbb{N}$ be any fixed natural number. For the rest of this section, we fix the choice $\eta=1/(10000k^2)$ in the construction of the minorant function in Lemma \ref{le_vinogradov}, that is, we work with the function
\begin{align*}
g=g_{\rho,1/(10000k^2)}.    
\end{align*}
  
Before embarking on the main proof, we show that if $(h_1,\ldots, h_k)$ is a suitably chosen tuple, then the sequence $g(Q(n+h_1))\cdots g(Q(n+h_k))$ has nearly the expected mean value $(2\rho)^k$ in the following quantitative sense.

\begin{lemma}[Correlations of $g(Q(n))$]\label{le_maynard2}
Let $k\geq 1$ and $\rho\in (0,1/4)$ be fixed. Let $g=g_{\rho,1/(10000k^2)}$ be the function constructed above. Also let $Q(y)\in \mathbb{R}[y]$ be a fixed non-constant polynomial with irrational leading coefficient, and let $\mathcal{F}:\mathbb{N}\to \mathbb{N}$ be any fixed function. Then there exists $H\geq \mathcal{F}(k)$ (depending on all the above quantities) and a $k$-tuple $(h_1,\ldots, h_k)\in [1,H]^k$ of distinct integers such that
\begin{align}\label{equ1}
\sum_{n\leq x}g(Q(n+h_1))\cdots g(Q(n+h_k))\geq 0.99(2\rho)^{k}x    
\end{align}
for $x\geq x_0$ large enough.
\end{lemma} 

\begin{proof}
We may assume that $\mathcal{F}$ grows fast enough. Write $Q(y)=\alpha_sy^s+\cdots+\alpha_1 y+\alpha_0$ with $\alpha_s\in \mathbb{R}\setminus \mathbb{Q}$. By Dirichlet's approximation theorem, we can find infinitely many pairs $(a',q')$ of  coprime integers such that $|\alpha_s-a'/q'|\leq 1/(q')^2$. Choose a pair for which $q'\geq \mathcal{F}(k+\lceil \rho^{-1}\rceil)$. We take $H=(q')^2$. Note that we have the hierarchy $x\gg H\gg k+\rho^{-1}$ with each parameter large enough in terms of the ones to the right of it.

From the union bound it follows that if we can show that each of (i) and (ii) and is individually satisfied by proportion $>1/2$ of tuples $(h_1,\ldots, h_k)\in [1,H]^k$ with $h_i$ distinct, then there exists a choice of $(h_1,\ldots, h_k)$ that obeys each of the two properties.

To show that proportion $>1/2$ of tuples $(h_1,\ldots, h_k)$ satisfy (i), by the pigeonhole principle it certainly suffices to show that 
\begin{align}\label{eq18}
 \sum_{1\leq h_1,\ldots, h_k\leq H}\sum_{n\leq x}g(Q(n+h_1))\cdots g(Q(n+h_k))\geq 0.999(2\rho)^k x\cdot H^k.    
\end{align}
By writing out the Fourier series for $g(x)$ from Lemma~\ref{le_vinogradov}, and truncating it at height  $K=(\log H)^{10}$, the left-hand side of \eqref{eq18} becomes
\begin{align}\label{equ2}\begin{split}
&(2\rho(1-\frac{1}{10000k^2}))^kx\cdot H^k+\sum_{\substack{|j_1|,\ldots, |j_k|\leq K\\(j_1,\ldots, j_k)\neq (0,\ldots,0)}}c_{j_1}\cdots c_{j_k}\sum_{n\leq x}\prod_{i=1}^k \sum_{h\leq H}e(j_iQ(h+n))\\
&+O_{k}(H^k\frac{x}{(\log H)^{10}}).
\end{split}
\end{align}
To bound the sum over $j_1,\ldots, j_k$ in \eqref{equ2}, we note that if $j_m\neq 0$ then by Lemma \ref{le_weyl}, and our choice $H=(q')^{1/2}$ where $|\alpha_s-a'/q'|\leq 1/(q')^2$,  we have
\begin{align}\label{equ52}
\left|\sum_{h\leq H}e(j_mQ(h+n))\right|\ll H^{1-\gamma_s}
\end{align}
for some constant $\gamma_s>0$. Since $j_m$ cannot all be zero in the first sum in \eqref{equ2}, we see that the whole expression \eqref{equ2} is 
\begin{align*}
\geq (2\rho(1-\frac{1}{10000k^2}))^kx\cdot H^k+O(xH^{k-\gamma_s}(\log H)^{10k})-O_{k}(\frac{xH^k}{(\log H)^{10}})\geq 0.999(2\rho)^kxH^k,
\end{align*}
recalling that $H$ is large enough in terms of $k$ and $\rho$. This gives the desired inequality.
\end{proof}

We will need the following lemma on simultaneous equidistribution of minor arc polynomial phases.

\begin{lemma}[Equidistribution of minor arc polynomial phases]\label{le_equidist_poly}
Let $s\geq 1$, $k\geq 1$ and $D\geq 1$ be fixed. Then there exists $c_s>0$ such that the following holds. 

Let $Q(y)=\alpha_sy^s+\cdots +\alpha_1 y+\alpha_0$ with $\alpha_0,\alpha_1,\ldots, \alpha_s\in \mathbb{R}$. Let $1\leq a\leq q$ be such that $(a,q)=1$ and $|\alpha_s-a/q|\leq 1/q^{2}$. Let $x=q^2$. Let $|j_i|,|h_i|\leq x^{1/(10s)}$ be any integers such that the $h_i$ are distinct and $j_1Q(y+h_1)+\cdots +j_kQ(y+h_k)$ is not constant. Then, for $x$ large enough, 
\begin{enumerate}
\item the polynomial sequence $(j_1Q(n+h_1)+\cdots +j_kQ(n+h_k) + \Z)_{n\leq x}$ on the nilmanifold $\mathbb{R}/\Z$ is totally $x^{-c_s}$-equidistributed;
\item the nilsequence $(e(j_1Q(n+h_1)+\cdots +j_kQ(n+h_k)))_{n\leq x}$ has mean value $\ll x^{-c_s}$.
\end{enumerate}
\end{lemma}

\begin{proof}
Let 
\begin{align}\label{eqqn1}
g(y):=j_1Q(y+h_1)+\cdots +j_kQ(y+h_k)=\sum_{0\leq d\leq s}\beta_dy^d    
\end{align}
for some $\beta_d\in \mathbb{R}$. Then $e(g(y))$ is a nilsequence on $\mathbb{R}/\mathbb{Z}$ of degree $\leq s$, dimension $1$ and Lipschitz constant $O(1)$. Note that $\int_{\mathbb{R}/\mathbb{Z}}e(z)\, dz=0$. Hence, by the quantitative Leibman theorem (Theorem \ref{thm:quant-leibman}), the sequence $(g(n)+ \mathbb{Z})_{n\leq x}$  is totally $x^{-c_s}$-equidistributed with mean value $\ll x^{-c_s}$ for some $c_s>0$, provided that there does not exist any integer $1\leq \ell\leq x^{1/10}$ (say) satisfying $\|\ell g\|_{C^{\infty}(x)}\leq x^{1/10}$. Suppose for the sake of contradiction that such an integer $\ell$ exists. By Lemma \ref{lem:smoothness-norm}, the existence of $\ell$ implies the existence of another integer $1\leq \ell'\ll_s x^{1/10}$ such that 
\begin{align}\label{eqqn3}
\|\ell'\beta_d\|\ll_sx^{1/10-d}\quad \textnormal{for all}\quad 1\leq d\leq s.     
\end{align}

By the binomial formula, we see that
\begin{align}\label{eqqn2}
\beta_d=\sum_{d\leq m\leq s}J_{d,m}\alpha_{m},\quad \textnormal{where}\quad J_{d,m}:=\binom{m}{d}\sum_{1\leq r\leq k}j_rh_r^{m-d}.
\end{align}
We have the crude upper bound 
\begin{align}\label{eqqn4}
 |J_{d,m}|\leq  2^s\cdot k \cdot x^{1/(10s)+1/10}< x^{1/4}/2 
\end{align}
if $x$ is large enough.

Now, using $|\alpha_s-a/q|\leq 1/q^2$, \eqref{eqqn3} and \eqref{eqqn4}, we have
\begin{align*}
x^{-9/10}\gg_s \|\ell'\beta_s\|= \|\ell'J_{s,s}\alpha_s\|\geq \Big\|\ell'J_{s,s}\frac{a}{q}\Big\|-\frac{\ell'|J_{s,s}|}{q^2}\geq \Big\|\ell'J_{s,s}\frac{a}{q}\Big\|-\frac{1}{2q}.
\end{align*}
Since $1/(2q)=1/(2x^{1/2})$, and $\ell' |J_{s,s}|<q$, this is a contradiction unless $J_{s,s}=0$. We will show by backwards induction on $d$ that $J_{d,m}=0$ for all $d\leq m\leq s$. The case $d=s$ has been handled. Suppose that $J_{d+1,m}=0$ for all $d+1\leq m\leq s$. Note that then $J_{d,m}=\binom{m}{d}/\binom{m+1}{d+1}\cdot J_{d+1,m+1}=0$ for $d\leq m\leq s-1$. Therefore, by \eqref{eqqn3} and \eqref{eqqn4}, we have
\begin{align*}
x^{-9/10}\gg_s \|\ell'\beta_d\|= \Big\|\ell'\sum_{d\leq m\leq s}J_{d,m}\alpha_m\Big\|=\|\ell'J_{d,s}\alpha_s\|\geq \Big\|\ell'J_{s,s}\frac{a}{q}\Big\|-\frac{\ell'|J_{d,s}|}{q^2}\geq \Big\|\ell'J_{d,s}\frac{a}{q}\Big\|-\frac{1}{2q}.    
\end{align*}
Again since $1/(2q)=1/(2x^{1/2})$ and $\ell' |J_{d,s}|<q$, this is a contradiction unless $J_{d,s}=0$. 

We have now shown that $J_{d,m}=0$ for all $1\leq d\leq m\leq s$. But then, by \eqref{eqqn1} and \eqref{eqqn2}, the polynomial $g(y)$ is constant, contrary to assumption. 
\end{proof}

\begin{proof}[Proof of Theorem \ref{thm:bdd}] Let $s\geq 1$, $k\in \mathbb{N}$ and $\rho>0$ be fixed, with $k$ large enough in terms of $s$ and $\rho$ small enough in terms of $s,k$.  Let $k<w\ll_{Q}1$ be such that, denoting $W=\prod_{p\leq w}p$, all the rational coefficients of the polynomial $Q(Wy+1)$ (if there are any) are integers. Now, by restricting to the smaller nil-Bohr set $B'=B\cap (W\mathbb{Z}+1)$, and denoting $Q_1(y)=Q(Wy+1)$, it suffices to prove that the set
\begin{align*}
\{Wn+1\in \mathbb{P}:\,\, \|Q_1(n)\|< \rho\}    
\end{align*}
has bounded gaps, with $Q_1$ now having the form
\begin{align}\label{eq17}
Q_1(y)=\alpha_sy^{s}+\cdots+\alpha_1 y+\alpha_0,\quad \alpha_1,\ldots, \alpha_s\in (\mathbb{R}\setminus{\mathbb{Q}})\cup\{0\},\quad \alpha_s\neq 0.
\end{align}
In particular, the leading coefficient $\alpha_s$ of $Q_1$ is irrational. We shall rename $Q_1$ as $Q$ for simplicity. 

Let $H$ be an integer that satisfies the conclusion of Lemma \ref{le_maynard2} (for some fast-growing but fixed function $\mathcal{F}$). Note that $H$ depends only on $k,\rho, Q$.

Since $\alpha_s$ is irrational, we may find an infinite set $\mathcal{Q}$ of integers $q\geq 1$ such that for some $a$ coprime to $q$ we have $|\alpha_s-a/q|\leq 1/q^{2}$. Set $x=q^2$ with $q$ ranging over $\mathcal{Q}$, and note that $x$ tends to infinity along an infinite subsequence of the integers. We are going to show that $\{Wn+1\in \mathbb{N}:\,\, \|Q(n)\|<\rho\}$ contains infinitely many pairs of primes differing by $\leq WH$.

To prove this we shall apply the weighted version of Maynard's sieve given by Proposition~\ref{prop_maynard}. Let $g=g_{\rho,1/(10000k^2)}$ be the minorant for $1_{\|x\|<\rho}$ constructed before. We set our weight $\omega_n$ to be
\begin{align*}
\omega_n=g(Q(n+h_1))\cdots g(Q(n+h_k)),    
\end{align*}
where $(h_1,\ldots, h_k)\in [1,H]^k$ is any $k$-tuple of distinct elements satisfying the conditions of Lemma \ref{le_maynard2}.

Note that $\omega_n$ is a nonnegative minorant for the indicator function of $\{n:\,\,\|Q(n+j)\|<\rho\,\, \forall\,\, j\leq k\}$. By Lemma \ref{le_maynard2}, we have
\begin{align}\label{eq44}
\sum_{n\leq x}\omega_n\geq 0.99(2\rho)^{k} x.
\end{align}

We will show that the weight $\omega_n$ satisfies the hypotheses of Proposition \ref{prop_maynard}, taking $\theta=c_s/10, \delta=1/10,C_0=10$ there, where $c_s$ is the constant in Lemma \ref{le_equidist_poly}. After that from Proposition \ref{prop_maynard} we obtain the lower bound
\begin{align*}
\sum_{\substack{n\leq x\\|\{W(n+h_1)+1,\ldots, W(n+h_k)+1\}\cap \mathbb{P}|\geq (1/(10C_s)) \log k}}\omega_n\gg_{k,s, \rho} \frac{1}{(\log x)^k}\sum_{n\leq x}\omega_n\gg_{k,\rho} \frac{x}{(\log x)^k},    
\end{align*}
for some constant $C_s\geq 1$. By choosing $k=\lceil e^{20C_s}\rceil$, this then proves that gaps of length at most $W\max_{1\leq i\leq j\leq k}|h_i-h_j|\leq WH$ occur infinitely often in our Bohr set\footnote{In fact,  if one lets $k\to \infty$ one sees that for any $m$ there are infinitely many intervals of bounded length containing $m$ primes from our nil-Bohr set.}.

We shall now inspect hypotheses (i)--(iv) of Proposition \ref{prop_maynard} for $x$ large enough in terms of $k, s, \rho, H$. They take the forms 

(i) Prime number theorem: For $1\leq h\leq H$,
\begin{align*}
\sum_{\substack{n\leq x\\W(n+h)+1\in \mathbb{P}}}\omega_n\geq \frac{1}{10}\frac{W}{\varphi(W)}\sum_{n\leq x}\omega_n;    
\end{align*}

(ii) Level of distribution $c_s/10$ for $\omega_n$:
\begin{align*}
\sum_{r\leq x^{c_s/10}}\max_{c\Mod r}\Big|\sum_{\substack{n\leq x\\n\equiv c\Mod r}}\omega_n-\frac{1}{r}\sum_{n\leq x}\omega_n\Big|\ll \frac{x}{(\log x)^{101k^2}};    
\end{align*}

(iii) Level of distribution $c_s/10$ for $\omega_n1_{W(n+h_i)+1\in \mathbb{P}}$ for $1\leq i\leq k$:
\begin{align*}
 \sum_{r\leq x^{c_s/10}}\max_{(L_i(c),r)=1}\Big|\sum_{\substack{n\leq x\\n\equiv c\Mod r\\W(n+h_i)+1\in \mathbb{P}}}\omega_n-\frac{\varphi(W)}{\varphi(Wr)}\sum_{\substack{n\leq x\\W(n+h_i)+1\in \mathbb{P}}}\omega_n\Big|\ll \frac{x}{(\log x)^{101k^2}};     
\end{align*}

(iv) A Brun--Titchmarsh type bound for $\omega_n$:
\begin{align*}
\max_{c\Mod r}\sum_{\substack{n\leq x\\n\equiv c\Mod r}}\omega_n\leq  \frac{10}{r}\sum_{n\leq x}\omega_n, \end{align*}
uniformly for $1\leq r\leq x^{c_s/10}$.

Note that any nonnegative constant sequence satisfies hypotheses (i)--(iv) by the prime number theorem, the Bombieri--Vinogradov theorem, and the Brun--Titchmarsh inequality.
By Lemma \ref{le_vinogradov}, we may expand $\omega_n$ as a Fourier series:
\begin{align*}
\omega_n=(2\rho(1-1/(10000k^2)))^k+\sum_{\substack{j_1,\ldots, j_k\\(j_1,\ldots, j_k)\neq (0,\ldots, 0)}}c_{j_1}\cdots c_{j_k}e(j_1Q(n+1)+\cdots +j_kQ(n+k)),
\end{align*}
where the coefficients $c_j$ have the properties stated in Lemma \ref{le_vinogradov}.
 Then truncate the Fourier expansion of $\omega_n$ to $|j_i|\leq M:=(\log x)^{200k^2}$. Recalling properties (ii) and (iii) of the Fourier expansion of $g(x)$ as well as \eqref{eq44}, it suffices to verify that
\begin{align*}
\omega_n':=e(j_1Q(n+1)+\cdots +j_kQ(n+k))    
\end{align*}
for any integer tuple $(j_1,\ldots, j_k)\in [-M,M]^k\setminus\{0\}$  for which $j_1Q(y+h_1)+\cdots +j_kQ(y+h_k)$ is non-constant, satisfies the following hypotheses:

(i')
\begin{align*}
\Big|\sum_{\substack{n\leq x\\W(n+h)+1\in \mathbb{P}}}\omega_n'\Big|+\Big|\sum_{n\leq x}\omega_n'\Big|\ll \frac{x}{(\log x)^{200k^3}};  
\end{align*}

(ii') 
\begin{align*}
\sum_{r\leq x^{c_s/10}}\max_{c\Mod r}\Big|\sum_{\substack{n\leq x\\n\equiv c\Mod r}}\omega_n'-\frac{1}{r}\sum_{n\leq x}\omega_n'\Big|\ll \frac{x}{(\log x)^{200k^3}};    
\end{align*}

(iii') 
\begin{align*}
 \sum_{r\leq x^{c_s/10}}\max_{(W(c+h_i)+1,r)=1}\Big|\sum_{\substack{n\leq x\\n\equiv c\Mod r\\W(n+h_i)+1\in \mathbb{P}}}\omega_n'-\frac{\varphi(W)}{\varphi(Wr)}\sum_{\substack{n\leq x\\W(n+h_i)+1\in \mathbb{P}}}\omega_n'\Big|\ll \frac{x}{(\log x)^{200k^3}};     
\end{align*}

(iv')
\begin{align}\label{eq24}
\max_{c\Mod r}\Big|\sum_{\substack{n\leq x\\n\equiv c\Mod r}}\omega_n'\Big|=o\Big(\frac{x}{r(\log x)^{200k^3}}\Big)    
\end{align}
for $1\leq r\leq x^{c_s/10}$.

Note that by Lemma \ref{le_equidist_poly} (and the assumption $|\alpha_s-a/q|\leq 1/q^2$ with $q=x^{1/2}$) the sequence $\omega_n'$ is totally $x^{-c_s}$-equidistributed as a nilsequence and has mean $\ll x^{-c_s}$ over $[1,x]$. Hypothesis (i') then follows immediately by combining Lemma \ref{le_weyl} with the converse to the quantitative Leibman theorem (Theorem \ref{lem:converse-leibman}).

Then, to deal with (ii') and (iv'), note that as $\omega_n'$ is a degree $\leq s$ nilsequence that is totally $x^{-c_s}$-equidistributed and has mean value $\ll x^{-c_s}$,  also for any $1\leq c\leq r\leq x^{c_s/2}$ the sequence $(\omega_{rm+c}')_{m\leq x/r}$ is a degree $\leq s$ nilsequence that is  $x^{-c_s/2}$-equidistributed on $[x/r,2x/r]$ with mean value $\ll x^{-c_s/2}$, so \eqref{eq24} automatically holds, in fact in a stronger form where the right-hand side is replaced with $o((x/r)^{1-c_s/9})$. This stronger form immediately implies hypothesis (ii') as well.

Lastly, we deal with (iii'). Note that for any $c$ coprime to $d$ we have
\begin{align*}
\sum_{\substack{n\leq x\\n\equiv c\Mod r\\W(n+h_i)+1\in \mathbb{P}}}\omega_n'&=\sum_{\substack{m\leq Wx\\m\equiv 1\pmod W\\m\equiv W(c+h_i)+1\pmod d}}1_{\mathbb{P}}(m)\omega'_{(m-1)/W-h_i}+O(1)\\
&\ll_W \max_{\substack{(c',dW)=1\\c'\equiv 1\pmod W}}\left|\sum_{\substack{m\leq Wx\\m\equiv c'\pmod{dW}}}1_{\mathbb{P}}(m)\omega'_{(m-1)/W-h_i}\right|+1.   
\end{align*}
Since $\omega_n'$ is totally $x^{-c_s}$-equidistributed with mean value $\ll x^{-c_s}$, also $\omega_{(n'-1)/W-h_i}$ is totally $x^{-c_s/2}$-equidistributed with mean value $\ll x^{-c_s/2}$. Thus, hypothesis (iii') follows from our Bombieri--Vinogradov theorem for equidistributed nilsequences (Theorem \ref{equidist-1/3}) after applying partial summation to replace $1_{\mathbb{P}}(m)$ with $\Lambda(m)$.

Now hypotheses (i')--(iv') have been verified, so Theorem~\ref{thm:bdd} follows.
\end{proof}

\section{Chen primes in nil-Bohr sets}

Our task in this section is to prove Theorem \ref{thm:chen} on Chen primes in nil-Bohr sets.

\begin{proof}[Proof of Theorem \ref{thm:chen}]

Let $Q(x)=\alpha_sx^s+\cdots+\alpha_1 x+\alpha_0$. By assumption, we can find at least one $1\leq j\leq s$ such that $\alpha_j$ is irrational. 
By Dirichlet's approximation theorem, we can then find infinitely many pairs $(a,q)$ of coprime integers with $|\alpha_j-a/q|\leq 1/q^2$. Restrict to those $x$ that can be written as $x=q^2$ for some such $q$; this is a sequence of integers that tends to infinity.

Let $\theta_s>0$ be a small enough constant. Let $g=g_{\rho,1/2}$ be the minorant function arising from Lemma \ref{le_vinogradov} with $\rho=x^{-\theta_s}$. Then we can estimate
\begin{align}\label{eq19}
\sum_{\substack{p\leq x\\p\in \mathcal{P}_{\textnormal{Chen}}}}1_{\|Q(p)\|<p^{-\theta_s}}\geq \sum_{\substack{p\leq x\\p\in \mathcal{P}_{\textnormal{Chen}}}}g(Q(p)).    
\end{align}
Define $\omega_n:=g(Q(n))$. By applying Chen's sieve in the form of Proposition \ref{prop_chen} to $(\omega_n)_{n\leq x}$, we see that the right-hand side of \eqref{eq19} is
\begin{align*}
\gg \frac{1}{(\log x)^2}\sum_{n\leq x}g(Q(n))-O(x^{0.9}),    
\end{align*}
provided that hypotheses (i)--(iii) of Proposition \ref{prop_chen} are satisfied. 

Expanding out the Fourier series of $g(x)$ given by Lemma \ref{le_vinogradov} and truncating it from height $M:=x^{3\theta_s}$, we can estimate
\begin{align}\label{eq20}\begin{split}
\frac{1}{(\log x)^2}\sum_{n\leq x}g(Q(n))&= \frac{1}{2}x^{-\theta_s}\frac{x}{(\log x)^2}+\frac{1}{(\log x)^2}\sum_{0<|j|\leq M}c_j\sum_{n\leq x}e(jQ(n))+O\left(x^{1-2\theta_s}\right)\\
&=\left(\frac{1}{2}+O(x^{-c_s})\right)\frac{x^{1-\theta_s}}{(\log x)^2}  
\end{split}
\end{align}
for $x$ large enough and some constant $c_s>0$, where for the last inequality we used Lemma \ref{le_weyl} and the assumption that $x=q^2$ where $q$ satisfies $|\alpha_j-a/q|\leq 1/q^2$ (so $\|\ell \alpha_j\|\geq \frac{1}{2q}$ for $1\leq \ell\leq q/2$).

Therefore, to conclude the proof of Theorem \ref{thm:chen}, it suffices to verify hypotheses (i)--(iii) of Proposition \ref{prop_chen} for $\omega_n=g(Q(n))$. Since
\begin{align*}
\omega_n=\frac{1}{2}x^{-\theta_s}+\sum_{0<|j|\leq M}c_je(jQ(n))+O(x^{1-2\theta_s})   
\end{align*}
with $|c_j|\ll x^{-\theta_s}$, and since any nonnegative constant sequence satisfies hypotheses (i)--(iii), it suffices to verify that (using the notation of Proposition \ref{prop_chen}) for $0<|j|\leq M=x^{3\theta_s}$ the sequence $\omega_n':=e(jQ(n))$ satisfies\\

(i') 
\begin{align}\label{eq45}
 \Big|\sum_{\substack{r\leq x^{1/2-\varepsilon}\\(r,2)=1}}\lambda_r'\sum_{\substack{n\leq x\\n+2\equiv 0\pmod r\\n\in \mathbb{P}}}\omega_n'\Big| \ll x^{1-4\theta_s};  
\end{align}

(ii')
\begin{align}\label{eq46}
 \Big|\sum_{\substack{r\leq x^{1/2-\varepsilon}\\(r,2)=1}}\lambda_r\sum_{\substack{n\leq x\\n\equiv 0\pmod r\\n+2\in B_i}}\omega_n'\Big| \ll x^{1-4\theta_s};  
\end{align}

(iii') 
\begin{align}\label{eq47}
 \Big|\sum_{\substack{n\leq x\\n\in B_i}}\omega_n'\Big|+\Big|\sum_{\substack{n\leq x\\n+2\in B_i}}\omega_n'\Big|\ll x^{1-4\theta_s},  
\end{align}

where $\lambda_r$ is well-factorable of level $x^{1/2-\varepsilon}$, $\lambda_r'$ is either well-factorable of level $x^{1/2-\varepsilon}$ or $\lambda'_r=1_{p\in [P,P']}*\lambda''_r$ with $\lambda''_r$ well-factorable of level $x^{1/2-\varepsilon}/P$ and $2P\geq P'\geq P\in [x^{1/10},x^{1/3-\varepsilon}]$.

Note that $\omega_n'$ is a degree $s$ nilsequence and that, by the quantitative Leibman theorem and the fact that $\|\ell \alpha_j\|\geq \frac{1}{2q}$ for $1\leq \ell\leq q/2$, it is totally $x^{-c_s}$-equidistributed for some small constant $c_s>0$, Moreover, by Lemma \ref{le_weyl}, $\omega_n'$ has mean value $\ll x^{-c_s}$ over $[1,x]$. 

We first verify hypothesis (i'). By applying partial summation, it suffices to prove (i') with the condition $n\in \mathbb{P}$ replaced by the weight $\Lambda(n)$. Then, by Vaughan's identity, it suffices to prove that
\begin{align*}
 \Big|\sum_{\substack{r\leq x^{1/2-\varepsilon}\\(r,2)=1}}\lambda_r'\sum_{\substack{mn\leq x\\mn\equiv -2\pmod r\\M\leq m\leq 2M}}a_mb_n\omega_{mn}'\Big| \ll x^{1-5\theta_s},  
\end{align*}
where either $M\in [x^{1/2},x^{2/3+o(1)}]$ and $|a_m|,|b_m|\ll d(n)^{O(1)}$ (type II case), or $M\leq x^{1/3+o(1)}$ and $b_n\in \{1,\log n\}$ (type I case).  Note that $\omega_n'$ is a degree $s$ nilsequence which is totally $x^{-c_s}$-equidistributed and has mean $\ll x^{-c_s}$ on $[1,x]$. Now, by applying Proposition~\ref{typeI} in the type I case and Proposition~\ref{typeII-lambda-general} in the type II case (and making use of the ``partial well-factorability'' of $\lambda_r'$ to write it as a convolution of two factors supported on $[1,x^{1-\varepsilon/2}/M]$ and $[1,Mx^{-1/2}]$), we obtain for the left-hand side of \eqref{eq45} a bound of $\ll x^{1-c_{s,\varepsilon}}$, where 
$c_{s,\varepsilon}$ is small enough. If we require that $\theta_s<c_{s,\varepsilon}/5$, the bound obtained is good enough.

For verifying hypothesis (ii'), we note that 
\begin{align}\label{eq47b}
 1_{B_1}(n)=\sum_{\substack{n=mp\\p\geq x^{1/10}}}a_m,\quad 1_{B_2}(n)= \sum_{\substack{n=mp\\p\geq x^{1/10}}}a_m',  
\end{align}
where $a_m$ is the indicator of $m$ being of the form $p_1p_2$ with $x^{1/10}\leq p_1\leq x^{1/3-\varepsilon}$ and $x^{1/3-\varepsilon}\leq p_2\leq ((2x+2)/p_1)^{1/2}$ and $a_m'$ is the indicator of $m$ being of the form $p_1p_2$ with $x^{1/3-\varepsilon}\leq p_1\leq p_2\leq ((2x+2)/p_1)^{1/2}$. Thus $1_{B_1}(n),1_{B_2}(n)$ are both type II convolutions with the $m$ variable supported in $[x^{1/3},x^{2/3+o(1)}]$. Hence, from Proposition \ref{typeII-lambda}, we obtain for the left-hand side of \eqref{eq46} a bound of $\ll x^{1-c_{s,\varepsilon}'}$ for some constant $c_{s,\varepsilon}'>0$. If we require that $\theta_s<c_{s,\varepsilon}'/4$, the bound is good enough.

Lastly, to handle hypothesis (iii'), we appeal to \eqref{eq47b} again to reduce matters to type II sums, and then apply Proposition \ref{type II-1/3} (taking only the $d=1$ term in the sum there).  We obtain for the left-hand side of \eqref{eq47} a bound of $\ll x^{1-c_{s,\varepsilon}''}$ for some $c_{s,\varepsilon}''>0$. If we require that  $\theta_s<c_{s,\varepsilon}''/4$, the bound obtained is strong enough.

The proof is now complete.
\end{proof}

\section{Linear equations in primes in arithmetic progressions}

We now turn to the proof of Theorem~\ref{thm:app3}, which also includes the proof of Theorem~\ref{thm:app3'}.

\begin{proof}[Proof of Theorem~\ref{thm:app3}]
We adapt some arguments from~\cite{green-tao-linear}. Let $w = w(x)$ be a positive integer that tends to $+\infty$ slowly enough. By dividing the variables $\ve{n}$ into residue classes modulo $\mathscr{P}(w)$, it suffices to show that the result holds when $\mathscr{P}(w)\mid q$, with $O_A(Q(\log x)^{-A})$ exceptions. The details of this reduction is similar to the argument in~\cite[Section 5]{green-tao-linear}.

Henceforth assume that $\mathscr{P}(w)\mid q$. By the generalized von Neumann theorem \cite[Proposition 7.1]{green-tao-linear} and the relative version of the inverse theorem for the Gowers norms~\cite[Proposition 10.1]{green-tao-linear} (see also \cite[Theorem 1.3]{gtz}), we have \eqref{eq21} for $q$ outside an exceptional set of size $\ll_A Q(\log x)^{-A}$, provided that the following two conditions hold.\\

(i) For all but $\ll_{A} Q(\log x)^{-A}$ choices of $q\leq Q$ the following holds. For each invertible residue class $a\Mod q$, the function $\Lambda_{a,q}(n) := \frac{\varphi(q)}{q}\Lambda(qn+a)$ is majorized by $C_0\nu(n)$, where $\nu$ is $D$-pseudorandom with $D$ large enough in terms of $t,d,M$ (see \cite[Section 6]{green-tao-linear} for the definition of the pseudorandomness conditions).\\

(ii) For all but $\ll_{A} Q(\log x)^{-A}$ choices of $q\leq Q$ we have the Gowers norm bound $\|\Lambda_{a,q}-1\|_{U^k[x]}=o_{w\to \infty}(1)+o_{w;x\to \infty}(1)$ for all $k\geq 1$. (This is precisely the content of Theorem~\ref{thm:app3'}.) \\

\subsubsection*{Verifying (i):}

For showing condition (i), we follow \cite[Appendix D]{green-tao-linear} that establishes the analogous claim for $q=1$, indicating the necessary modifications (see also \cite[Proposition 6.1]{bienvenu} that handles the case $q\ll_{A}(\log x)^{A}$). Let 
\begin{align*}
\mathcal{Q}:=\{q\leq Q:\,\, \Omega(q)\leq C\log \log x\},    
\end{align*}
where $C$ is a large enough constant and $\Omega(q)$ denotes the number of prime factors of $q$ with multiplicities. We have
\begin{align*}
|[1,Q]\setminus\mathcal{Q}|\ll 2^{-C\log \log x}\sum_{q\leq Q}2^{\Omega(q)}\ll Q (\log x)^{-C/2},   
\end{align*}
so it suffices to prove condition (i) for $q\in \mathcal{Q}$. 

As in \cite[Appendix D]{green-tao-linear}, we define 
\begin{align*}
\nu(n):&=\frac{1}{2}+\frac{1}{2}\wt{\nu}(n),\\ \textnormal{with}\quad \wt{\nu}(n)&=\frac{\varphi(q)}{q}\Lambda_{\chi,R,2}(qn+a):= \frac{\varphi(q)}{q}(\log R)\Big(\sum_{\substack{d\mid qn+a\\d\leq R}}\mu(d)\chi\Big(\frac{\log d}{\log R}\Big)\Big)^2,   
\end{align*}
where $\chi$ is a smooth function compactly supported in $[-1,1]$ and $\chi(0)=1$, with $0\leq \chi(y)\leq 1$ everywhere. 
Note that since $\mathscr{P}(w)\mid q$, the $W$-trick is already incorporated in the definition of $\Lambda_{\chi,R,2}$.

Let $N$ be a prime of size $(C_D+o(1)) x$ for large enough $C_D>0$, and extend $\nu(n)$ to $[1,N]$ by defining it to be $=1$ elsewhere. Embed $\nu(n)$ into the cyclic group $\mathbb{Z}/N\mathbb{Z}$ in the obvious way. Then the function $\nu$ is our choice of a pseudorandom measure.  

We will inspect that the Goldston--Y{\i}ld{\i}r{\i}m estimate of \cite[Theorem D.3]{green-tao-linear} (with $N=x$) holds for the family of linear forms $\ve{n} \mapsto (L_1(q\ve{n}+a), \ldots, L_t(q\ve{n}+a))$. The case $q=1$ is \cite[Theorem D.3]{green-tao-linear}. We will then deduce the linear forms and correlation conditions for $\tilde{\nu}(n)$ by following \cite[pp. 75--77]{green-tao-linear}. Both of these arguments (\cite[Theorem D.3]{green-tao-linear} and \cite[pp. 75--77]{green-tao-linear}) go through in our setting with the following minor modifications: 
\begin{itemize}
    \item The factor $e^{O(X)}(\log R)^{-1/2}$ in \cite[Theorem D.3]{green-tao-linear} is $o(1)$ and therefore harmless, since
    \begin{align*}
    X\ll_{t,d,L} 1+\sum_{p\mid q}p^{-1/2}\ll \log \log \log N     
    \end{align*}
    by the assumption $q\in \mathcal{Q}$ and the fact that the ``exceptional'' primes $p$ in \cite[Theorem D.3]{green-tao-linear} are either $O_{d,t,L}(1)$ or divide $q$.
    
    \item In the proof of \cite[Theorem D.3]{green-tao-linear} there are a few conditions (such as $\alpha(p,B)=1/p$ and $\beta_p=1+O(1/p)$) that only hold for the primes $p\nmid q$ in our setting (as opposed to all large enough primes). This makes little difference in the argument, since we can separate the contribution of the primes $p\mid q$ from the rest and the contribution of the rest of the primes gives the correct local factors, whereas for $p\mid q$ quantities such as $E_{p,\xi}$ in \cite[Lemma D.5]{green-tao-linear} are easy to compute, which ultimately leads to the value $\beta_p = p/\varphi(p)$ for $p\mid q,$ as desired.
    
    \item When verifying the linear forms and correlation conditions in \cite[pp. 75--77]{green-tao-linear}, one should replace $W$ with $q$ and conditions such as $p\leq w$ with $p\mid q$ and $p>w$ with $p\nmid q$. The estimates in \cite[p. 77]{green-tao-linear} go through verbatim with these modifications. 
\end{itemize}

Thus condition (i) has been verified.

\subsubsection*{Verifying (ii):}

We utilize Theorem \ref{thm:1/4}, which implies that for any $A \geq 2$ we have
\begin{align}\label{eq31}
\max_{(a,q)=1}\sup_{\psi\in \Psi_k(\Delta,\log x)} \frac{q}{\varphi(q)} \Big|\sum_{n \leq x} \Lambda_{a,q}(n) \psi(n) - \frac{\varphi(q)V}{\varphi(qV)}\sum_{\substack{n \leq x \\ (qn+a,V)=1}} \psi(n) \Big| \ll_{k,A,\Delta,\varepsilon} \frac{x}{(\log x)^A}, 
\end{align}
for all but $O(Q(\log x)^{-A})$ values of $q \leq Q$, where $V = \mathscr{P}((\log x)^B)$ for some sufficiently large constant $B = B(A,k,\Delta) > 0$.
Let 
\[ \wt{\Lambda}_{a,q}(n) = \frac{\varphi(q)V}{\varphi(qV)} \cdot 1_{(qn+a,V)=1}\]
be the function appearing in the second sum in \eqref{eq31}. We claim that this function (embedded to a cyclic group) is $D$-pseudorandom for any fixed $D$. To see this, let $\chi$ be a smooth function supported on $[-1,1]$ with $\chi(y)=1$ for $|y|\leq 1/2$ and $0\leq \chi(y)\leq 1$ everywhere. Let
\begin{align*}
\Lambda_{\chi,R,2}^{(V)}(n):=(\log R)\Big(\sum_{\substack{d\mid n\\d\mid V\\d\leq R}}\mu(d)\chi\Big(\frac{\log d}{\log R}\Big)\Big)^2,
\end{align*}
where $R=N^{\gamma}$ for small enough $\gamma>0$. Then we have
\begin{align*}
\wt{\Lambda}_{a,q}(n) =(1+o(1)) \frac{\varphi(q)V}{\varphi(qV)(\log R)}\Lambda_{\chi,R,2}^{(V)}(qn+a)+O(\frac{\varphi(q)V}{\varphi(qV)}E(qn+a)),    
\end{align*}
where 
\begin{align*}
E(n)=d(n)^2 1_{\exists\, m\mid n, m\geq R^{1/2},m\,\,\textnormal{is}\,\, V-\textnormal{smooth}}.
\end{align*}
Similarly as in the proof of Theorem~\ref{thm_divisor}, the term $E(qn+a)$ (which is bounded by a divisor-type function) is negligible in the linear forms and correlation conditions for $(a,q)=1$ and all but $\ll_{A}Q/(\log x)^{A}$ values of $q\leq Q$. The function $\frac{\varphi(q)V}{\varphi(qV)(\log R)}\Lambda_{\chi,R,2}^{(V)}(qn+a)$ (embedded into a cyclic group) in turn is a pseudorandom measure for $(a,q)=1$ and $q\leq Q$, outside an exceptional set of moduli $q$ of size $\ll_{A} Q/(\log x)^{A}$, by the same argument that was used to verify condition (i). Thus $\wt{\Lambda}_{a,q}(n)$ itself is a pseudorandom measure.

Now by \cite[Theorem 1.3]{gtz}, \cite[Proposition 10.1]{green-tao-linear} and condition (i), \eqref{eq31} implies
\[ \|\Lambda_{a,q} -  \wt{\Lambda}_{a,q}\|_{U^k[x]} = o(1) \]
for almost all $q$. Thus it remains to establish that $\|\wt{\Lambda}_{a,q}-1\|_{U^k[x]} = o(1)$. But this follows from the the pseudorandomness of $\wt{\Lambda}_{a,q}$ together with \cite[Lemma 5.2]{green-tao-APs}.
\end{proof}

Lastly, in view of Remark \ref{rem:mobius}, we note that as a consequence of the proof method (in fact with some simplifications in the sense that we do not need pseudorandom majorants), we obtain the following result for the M\"obius function.

\begin{corollary}\label{app3-mobius}
 Let $\varepsilon>0$ and $A,t,d,M\geq 1$ be given. Let $x\geq 10$ and $Q\leq x^{1/3-\varepsilon}$. Then for all but $\ll_{\varepsilon,A,t,d,M}Q/(\log x)^{A}$ choices of $1\leq q\leq Q$ the following holds. For every $\ve{a}\in (\Z/q\Z)^d$ and every finite complexity tuple $\Psi=(L_1(\mathbf{n}),\ldots, L_t(\mathbf{n}))$ of non-constant affine-linear forms in $d$ variables of size $\|\Psi\|\leq M$ we have
\begin{align}
\sum_{\mathbf{n}\in [1,x]^d} \mu(L_1(q\mathbf{n}+\ve{a}))\cdots \mu(L_t(\mathbf{n}+\ve{a}))= o_{t,d,M}(x^{d}).     
\end{align}
\end{corollary}

%%% AUTHOR: optional appendix here

%%% AUTHOR: optional acknowledgments here
\section*{Acknowledgments} %%  you may comment this out if no Ackno
 We thank Ben Green, James Maynard and Terence Tao for helpful comments. We are grateful to the anonymous referee for a careful reading of the paper and for comments that improved the quality of this paper.

%%% AUTHOR:
%%% Bibliography goes here. Note that the arXiv cannot process bibtex
%%% or biber bibliographies.  Example of acceptable bibliograpy format:
\bibliographystyle{amsplain}

%% AUTHOR: You can generate such a bibliography from a .bib file by 
%% running pdflatex/bibtex/pdflatex/pdflatex and then pasting the .bbl file
%% between \begin{thebibliography} and \end{bibliography}

%%% AUTHOR: Include a short description of each author following the
%%% structure below. Use the same short tags used previously.  
%%% Use \imageat{} and \imagedot{} instead of "@" and "." in
%%% email addresses-this replaces the symbols with graphics to avoid 
%%% e-mail address harvesting from the .pdf file

\begin{dajauthors}
\begin{authorinfo}[XS]
  Xuancheng Shao\\
  Department of Mathematics\\
  University of Kentucky \\
  Lexington, KY, 40506, USA\\
  xuancheng.shao\imageat{}uky\imagedot{}edu \\
 % \url{https://www.cs.elte.hu/erdos}
\end{authorinfo}
\begin{authorinfo}[JT]
  Joni Ter\"{a}v\"{a}inen\\
  Mathematical Institute\\
  University of Oxford\\
  Radcliffe Observatory Quarter\\
  Woodstock Rd\\
  Oxford OX2 6GG, UK\\
  joni.teravainen\imageat{}maths\imagedot{}ox\imagedot{}ac\imagedot{}uk \\
  %\url{http://www.csc.kth.se/~johanh}
\end{authorinfo}
\end{dajauthors}

\end{document}